\documentclass[a4paper,11pt]{amsart}%{article}
\usepackage[margin=3cm]{geometry}
\usepackage[T1]{fontenc}
\usepackage{lmodern, amsfonts,amsmath,amstext,amsbsy,amssymb,
amsopn,amsthm,upref,eucal,mathptmx,mathtools,url}
\usepackage{comment}
\usepackage{mathrsfs}

\RequirePackage{xcolor} % [dvipsnames]
\definecolor{halfgray}
{gray}{0.55}%chapter numbers will be semi
\definecolor{webgreen}
{rgb}{0,0.4,0}
\definecolor{webbrown}
{rgb}{.8,0.1,0.1}
\definecolor{red}
{rgb}{1,0,0}
\usepackage{tikz}

\newcommand \R {{ \mathbb R}}
\def\C{{\mathbb C}}
\newcommand \Q {{ \mathbb Q}}
\newcommand \Z {{ \mathbb Z}}
\newcommand \N {{ \mathbb N}}
\newcommand \T {{ \mathbb T}}

\newcommand \z {{ \mathfrak h}}

\newcommand*{\diff}{\mathop{}\!\mathrm{d}}
\newcommand{\one}{{\rm 1\mskip-4mu l}}
\newcommand{\norm}[1]{\left\lVert#1\right\rVert}

\newcommand \re {{%
\operatorname{Re}
}}

\newcommand{\SL}{%
\operatorname{SL}
}

\newtheorem{theorem}{Theorem}[section]
\newtheorem {lemma}[theorem]{Lemma}
\newtheorem {proposition}[theorem]{Proposition}
\newtheorem{corollary}[theorem]{Corollary}
\newtheorem{remark}[theorem]{Remark}

\newtheorem{definition}[theorem]{Definition}

\date{\today}

\author{Artur Avila}

\address{Universit\"at Z\"urich, 
Institut f\"ur Mathematik\\
Winterthurerstrasse 190, 
CH-8057 Z\"urich, Switzerland\\ and
}

\address{IMPA, Estrada Dona Castorina 110, Rio de Janeiro, Brazil\\}

\email{artur.avila@math.uzh.ch\\}

\author{Giovanni Forni}

\address{Department of Mathematics\\
Mathematics Building\\
University of Maryland\\
College Park, MD 20742-4015, 
USA\\}

\email{gforni@math.umd.edu\\}

\author{Davide Ravotti}

\address{ %Universität Zürich, Institut für Mathematik, Y27-K36\\ Winterthurerstrasse 190, CH-8057 Zürich
Monash University, School of Mathematics \\ Clayton Campus, 3800 Victoria, Australia
}

\email{davide.ravotti@gmail.com\\}

\author{Corinna Ulcigrai}
   
\address{School of Mathematics \\ University of Bristol \\ BS8 1TW Bristol, United Kingdom\\ and }

\address{Universit\"at Z\"urich, 
Institut f\"ur Mathematik\\
Winterthurerstrasse 190, 
CH-8057 Z\"urich, Switzerland
}

\email{corinna.ulcigrai@bristol.ac.uk\\}  
 \title%[Mixing for Smooth Time-Changes of  Nilflows]
 {Mixing for Smooth Time-Changes of  General Nilflows}

\begin{document}

 \begin{abstract}
 We consider completely irrational nilflows on any  nilmanifold of step at least $2$. We show that there exists a dense set of smooth time-changes such that any time-change in this class  which is not {measurably trivial} gives rise to a \emph{mixing} nilflow. 
This in particular reproves and generalizes to any nilflow (of step at least $2$) the main result proved in \cite{AFU} for the special class of Heisenberg (step $2$)  nilflows, and later generalized in~\cite{Rav2} to a class of nilflows of arbitrary step which are isomorphic to suspensions of higher-dimensional linear toral skew-shifts.  
  \end{abstract}

\maketitle
\section{Introduction}

Dynamical systems can roughly be divided in three categories ({hyperbolic}, {elliptic} and {parabolic}) according to the speed of divergence (if any) of close orbits. A (non-singular) flow is called \emph{hyperbolic}  if nearby orbits diverge  exponentially in time. We say that the flow is \emph{parabolic} if there is divergence of nearby orbits, but this divergence happens at  \emph{subexponential} (usually \emph{polynomial}) speed, while the flow is called \emph{elliptic} if there is no divergence (or perhaps it is slower than polynomial). While there is a classical and well-developed theory of hyperbolic systems and also a systematic study of  elliptic ones, there is no general theory which describes the dynamics of parabolic flows and only classical and isolated examples are well-understood.   This paper contributes to our understanding of \emph{typical properties} of \emph{parabolic flows}. 

%\smallskip
%\subsubsection*{Classical examples of parabolic flows}
%1) Examples: -horocycle. Homogeneous world. In the homogeneous world, parabolic flows coincide with unipotent ones. Among those, unipotent on semi-simple can be though of direct generalization of horocycle. Other main category are %Nilflows on nilmanifolds. Basic example is Heisenberg (perhaps describe?). 
%Another important class of parabolic examples: also flows on surfaces, parabolic with singularities.

Perhaps the most studied example of a parabolic flow is given by the \emph{horocycle flow} on (the unit tangent bundle of) a compact negatively curved surface. In the context of \emph{homogeneous dynamics} (actions given by group multiplication on quotients of Lie groups), parabolic flows coincide with \emph{unipotent flows}. Horocycle flows  can be seen as the simplest example of unipotent flows on semi-simple Lie groups (given by the right action of upper triangular unipotent matrices in $\SL(2,\mathbb{R})$ on \emph{compact}\footnote{In the following, we will always tacitly assume that the lattice is \emph{cocompact}, even though some of the results also hold for finite volume surfaces.} quotients  $\Gamma\backslash \SL(2,\mathbb{R})$). Horocycle flows are the prototype of \emph{uniformly} (homogeneous) parabolic flows, since they display \emph{uniform}\footnote{Uniform here means that the speed of shearing of nearby orbits is uniformly bounded (on  initial points) from above and below.}  (and homogeneous) polynomial \emph{shearing} of nearby trajectories at every point.

In the context of \emph{area-preserving flows} on (higher genus) surfaces, another important class of parabolic flows  is  %(by many authors, including us) considered 
given by \emph{locally Hamiltonian} flows, which are \emph{smooth} flows which preserve a \emph{smooth} area-form\footnote{\emph{Linear flows} on translation surfaces (surfaces endowed with a flat metric with conical singularities)  are area-preserving but not smooth, due to the presence of singularities reached in finite time. They share many features of parabolic systems, but lack others. In this case the orbit divergence is entirely produced by the splitting of trajectories near the singularities. For this reason they can be considered as \emph{elliptic flows with singularities.}}.
A crucial feature in this context is the presence of saddle-type singularities, which create a  \emph{non-uniform} (and non-homogeneous) form of parabolic shearing of nearby orbits. These flows are hence an example of 
 \emph{non-uniformly} parabolic flows.

Another fundamental class of homogeneous  flows is given by \emph{nilflows}, or flows on (compact) quotients of \emph{nilpotent} Lie groups (\emph{nilmanifolds}). The \emph{prototype} example in this class are \emph{Heisenberg nilflows}. Let us recall that these are given by the action by right multiplication of a $1$-parameter subgroup of the \emph{Heisenberg group}, which can be seen as the group of $3\times 3$ upper triangular matrices, quotiented by a lattice (for example the subgroup of matrices of the same form, but with \emph{integer} entries). Nilflows have an elliptic factor, and they are definitely not \emph{uniformly} parabolic. In fact,  since  they have an isometric (central) direction they are an example of \emph{partially parabolic} flows.

%\textcolor{red} {Giovanni, ci viene un dubbio: ma tu li chiami partially parabolic perche' hanno una direzione isometrica o un fattore ellittico?} 
% with saddle-like singularities. In this case the orbit divergence is entirely produced by the splitting of trajectories near the singularities. Examples in this setting include 
% \textcolor{red}  { Probably it should be mentioned that smooth reparametrizations of parabolic flows
%are parabolic, and that singular reparametrizations of elliptic systems with or without singularities, can be parabolic. In fact, locally Hamiltonian flows on higher genus surfaces are  singular reparametrizations of elliptic flows with finitely many discontinuities of saddle-type.   Translation flows share many features of parabolic systems, but lack others, so I am not sure it is wise to advertise them as parabolic. In particular they are not mixing, and I believe it is important to our discussion that mixing should be a generic property of  parabolic systems. G.F.}

A natural and fundamental question in parabolic dynamics is hence which ergodic and spectral properties are \emph{generic} among  smooth parabolic flows. 
There is a large and quite extensive literature on the ergodic and spectral properties of these \emph{classical} parabolic examples, see for example \cite{Fu:uni, Ma2, Pa:hor,  FU, Bur, FF1, Str, BuFo, FoS1} for the horocycle flow, and more in general \cite{St} or \cite{AGH}, and the reference therein, for homogeneous parabolic flows; \cite{Ka:int, Keane, Ma:int, Ve:gau, AF:wea} or lecture notes such as \cite{Ma:rat, Yo, FoMa} and the references therein for translation flows;  \cite{Ul:wea, Ul2, Rav1, KKU, BK} among others for locally Hamiltonian flows.  These results do not show an entirely coherent picture and make the identification and description of characteristic \emph{parabolic} features uncertain. For example, while all smooth time-changes of horocycle flows are mixing and actually mixing of all orders (see \cite{Ma1, Ma2}), nilflows are never weakly mixing (see below). This  difference in behavior can be attributed to the lack of parabolicity in certain directions (more specifically to the existence of an elliptic factor).  In this paper we prove that these obstructions can be broken by a perturbation and that in a dense set of smooth-time changes all flows which are not trivially conjugate to the nilflow itself are indeed mixing.  Our result therefore supports the view that mixing is a generic property among parabolic perturbations of nilflows.

% characteristic property of typical  \emph{uniformely parabolic} parabolic systems. 
%These do not show a coherent picture, giving the impression that there is no hope for a unified understanding of parabolic dynamics and description of \emph{parabolic} features. For example, while horocycle flows are mixing and actually mixing of all orders (see \cite{Ma1, Ma2}), nilflows are never mixing (see below).  The philosphy we want to push  is that these are not \emph{typical} examples, and when considering \emph{perturbations} of these well-understood example, common typical features do indeed appear. In this work, in particular, we verify this for mixing in the set up of niflows. 

\subsection{Time-changes and parabolic perturbations.} 
Starting from the classical examples of parabolic flows mentioned above, one can build new parabolic flows by considering \emph{perturbations}: the simplest perturbations are perhaps  \emph{time-changes} (or time-\emph{reparametrizations}) of a given flow (see \S~\ref{sec:basic} for the definition), i.e. flows that move points along the \emph{same orbits}, but with different \emph{speed}.  
This construction has the advantage that certain ergodic properties, like \emph{ergodicity} and {\it cohomological properties} (which only depend on the orbit structure and hence are independent of the time-change) persist and \emph{smooth} time-changes of parabolic flows are still parabolic,   while finer (in particular \emph{spectral})  properties can (and do, as we will discuss) emerge. 

{Indeed even such simple perturbations as time-changes can produce genuinely \emph{new} parabolic flows, i.~e.~flows which are \emph{not} measurably conjugated to the unperturbed flow.} 
There is an obvious way to produce (measurably or even smoothly) conjugated flows, which corresponds to the case when the time-change is (measurably or smoothly) \emph{trivial} (see \S~\ref{sec:basic} for definitions). These trivial time-changes are described by solutions of the so-called \emph{cohomological equation}. 

A key feature of parabolic dynamics is  the existence of \emph{distributional obstructions} (invariant distributions) to solve the cohomological equation, that is, obstructions which are not signed measures\footnote{The first complete study of this phenomenon is perhaps Katok's work (which although written and circulated in the 80's
only appeared in~\cite{Kat:CC1}, \cite{Kat:CC2}) on linear skew-shifts of the $2$-torus, which are closely related to Heisenberg nilflows. Let us also remark that \emph{finitely many} invariant distributions for horocycle flows in the finite area, non-compact case were first constructed by P.~Sarnak \cite{Sarnak} by methods based on Eisenstein series.}. The structure of the space of obstructions was described in the case of translation flows (and locally Hamiltonian flows on surfaces) in \cite{Fo:sol}, for nilflows in~\cite{FF3} and horocycle flows, 
see~\cite{FF1}\footnote{The space of obstructions has a different structure in each of these cases: it  is in fact finite dimensional in any finite (Sobolev) order for translation flows, and  infinite dimensional for any sufficiently high order for nilflows and horocycle flows. Finally, the 
horocycle flow has obstructions of arbitrarily high order, which nilflows lack, and therefore subsumes features 
of both translation flows and nilflows.}.

 As a consequence, among smooth time-changes, {\it smoothly} trivial time-changes are \emph{rare} (i.e. form a finite or countable codimension subspace) and therefore time-changes can have essentially different dynamical  properties. However, the question whether a nontrivial time-change is also non-isomorphic is in general very difficult and the answer is known only in a few cases. 
 For example, for horocycle flows, it follows from Marina Ratner's work in \cite{Rat} and \cite{FF4} that
sufficiently smooth time-changes 
which are measurably isomorphic to the horocycle flow are actually  \emph{smoothly} trivial. Thus, time-changes which give isomorphic flows are \emph{rare} among  sufficiently smooth-time changes of horocycle flows. 

It should be pointed out that describing more general perturbations (beyond the class of time changes) which produce (new) \emph{parabolic} flows is quite delicate,  since by a perturbation  one \emph{typically} gets a hyperbolic flow. Examples of parabolic perturbations, which are not time changes, can be constructed for example by \emph{twisting} (see for example the work of Simonelli on \emph{twisted} horocycle flows, in~\cite{Si}); the twisting construction is a particular case of general isometric or unipotent extensions,  constructions which also preserve the parabolic nature of the dynamics.  New examples of parabolic perturbations for which one can study ergodic theoretical properties were recently constructed by D.~R.~in~\cite{Rav3}.

\subsection{Previous results on time-changes.} 

\emph{Perturbations} of (homogeneous) parabolic flows are much less understood than the classical homogeneous parabolic examples, even in the simplest case of time-changes. 
For example, while ergodic and spectral properties of horocycle flows have long been well-understood (see for example   
\cite{Fu:uni, Ma2, Pa:hor}), 
much less is known about the ergodic theory of time-changes of unipotent flows or nilflows.
It is a classical result of B.~Marcus \cite{Ma2} that smooth time-changes of horocycle flows are mixing\footnote{Let us recall that a measure-preserving flow $(\varphi_t)_{t\in \mathbb{R}}$ on a probability space $(X,\mathscr{A},\mu)$ is \emph{mixing} if for any two square-integrable functions $f,g\in L^2(X,\mathscr{A},\mu)$ the correlations $\int_X (f\circ \varphi_t) g \diff\mu $ converge to $(\int_X f \diff \mu)( \int_X g \diff \mu)$ as $t\to +\infty$.} (and actually mixing of all orders) (see also a previous result by Kuschnirenko \cite{Kuschnirenko:mixing}, which applies to time-changes  which are sufficiently small in the $\mathscr{C}^1$ topology). 
Effective mixing and spectral results  are recent.  Decay of correlations and the Lebesgue spectral property were proved by two of the authors (G.~F.~and C.~U.) in \cite{FU}, thus partially confirming  the Katok-Thouvenot  conjecture \cite{KT} on the nature of the spectrum (the question on the multiplicity of the spectrum was left open). The absolute continuity of the spectrum was simultaneously proved by Tiedra de Aldecoa in \cite{Tie} (and later extended to the case of semi-simple unipotent flows by Simonelli \cite{Si}).  The methods for proving the Lebesgue spectrum property have been refined by B.~Fayad, G.~F.  and A. Kanigowski  \cite{FFK}, who 
treated the case of Diophantine toral flows with a sufficiently strong power singularity (Kochergin flows). 

In a recent improvement of the paper \cite{FFK} the authors were able to prove the countable multiplicity of the spectrum  for Kochergin flows, as well as for time-changes of the horocycle flows (thereby completing the proof of the Katok-Thouvenot conjecture).

Recent work of  A.~Kanigowski, M.~Lemanczyk and C.~U.~\cite{KLU}  has brought the insight that although horocycle flows are uniformly parabolic, they have special,  non-generic properties coming from the homogeneous nature of the dynamics. In particular, all \emph{powers}  of the horocycle flows are isomorphic through the action of the geodesic flow. For non-trivial smooth time changes  of the horocycle flow,  this phenomenon does \emph{not} happen and one can prove that powers  are all \emph{disjoint} (as shown in \cite{KLU} and also by L.~Flaminio and G.~F.~in \cite{FF4} using Ratner's work \cite{Rat}).

\smallskip

Locally Hamiltonian flows on higher genus surfaces  can be seen as \emph{singular} time-changes of linear flows on translation surfaces (which are well-known \emph{not} to be mixing, see \cite{Ka:int}). Since, as already mentioned, locally Hamiltonian flows  are parabolic, but \emph{not-uniformly} parabolic,  the presence (or genericity) of \emph{mixing} is much more delicate (and less persistent) than in the uniformly parabolic case.   
 It turns out furthermore that  whether mixing is typical depends crucially on the type of singularities: while the presence of degenerate (multi-saddle singularities) was long known to produce mixing \cite{Ko2} if all singularities are non-degenerate (Morse) singularities,  mixing is generic (on each minimal component) if there are \emph{saddle loops} homologous to zero, but absence of mixing (and weak mixing) is generic if the flow has only simple saddles. Strengthenings of the mixing property, such as quantitative mixing estimates, spectral properties  or multiple mixing, and exceptional mixing examples  have also been studied (see e.~g.~\cite{Fa2, Rav1, FFK, CW}).
\smallskip

Not much is known for time-changes of nilflows. In the special case of {\it Heisenberg nilflows}, mixing for non-trivial time-changes  was proved  by three of the authors (A. A., G. F. and C. U.) in \cite{AFU}; G. F. and A.~Kanigowski proved effective mixing results for generic nilflows in \cite{FK1} and the Ratner property, disjointness of powers and multiple mixing \cite{FK2} for the (measure zero) class of Heisenberg nilflows with bounded type frequencies.

The advances in the ergodic theory of parabolic flows recalled above mainly concern time-changes of {\it renormalizable flows},  a class which so far includes  only translation flows, horocycle flows and Heisenberg nilflows. For these flows, renormalization provides a powerful tool to analyze the fine behavior of ergodic averages, which is crucial in several  results on the ergodic theory of their time-changes\footnote{See \cite{FF1,BuFo} for horocycle flows, \cite{FF2} for  Heisenberg nilflows, and \cite{Fo:dev,Bu} for translation flows (or IET's). In the case of Heisenberg nilflow, bounds on ergodic integrals have a long history going back to the work of Hardy and Littlewood on bounds of quadratic Weyl sums more than a century ago until the optimal bounds of H.~Fiedler, W.~Jurkat and O.~K\"orner  \cite{FJK:Weyl}. 
Recent results have refined the analysis of the behavior of ergodic integrals and derived results on their limit distributions (\cite{FF3}, J.~Marklof~\cite{Mr}, ~F.~Cellarosi and J.~Marklof~\cite{CM}, J.~Griffin and J.~Marklof \cite{GM},  A.~Fedotov and F.~Klopp~\cite{FeKl},  G.~F. and A.~Kanigowski~\cite{FK1}).}.  
Results for non-renormalizable flows, such as higher-step nilflows, are much rarer. 
In fact, refined quantitative  estimates, especially pointwise lower bounds on sets of large measures, are \emph{not} available for higher step nilflows, contrary to the step two case\footnote{However, quantitative (but not optimal) equidistribution estimates were proved for Lebesgue almost all points in the class of quasi-Abelian nilflows by \cite{FF3};  see also the related {\it uniform} estimates proved by T.~Wooley, e.g.~in \cite{W1}, \cite{W2} and, independently, by J. Bourgain, C. Demeter  and L.~Guth \cite{BDG} for exponential sums
(see also the Bourbaki seminar~\cite{P}, section 2.1.1).}. None of the known results seem to provide \emph{point-wise} lower bounds (on sets of large measure), although lower bounds in square mean follow from representation theory.

It is however natural to ask whether the results proved for Heisenberg nilflows also hold  for other nilflows, i.e.~when the step and dimension are higher. A first result in this direction was obtained by D.~R., who studied mixing among non-trivial time-changes in the class of \emph{quasi-Abelian filiform nilflows} (this is a special class of nilflows which constitutes a natural higher dimensional extension of Heisenberg nilflows, since they, as in the Heisenberg case, have a  Poincar{\'e} \emph{section} which is a skew-translation on a torus).

\subsection{Main results}
In this paper we consider \emph{a general nilmanifold} $M= \Gamma\backslash G$ (where $G$ is a nilpotent Lie group and $\Gamma< G$ a lattice) of step at least $2$ (to exclude the case when $G$ is Abelian and hence $M$ is a torus). We then consider a nilflow $\phi=\{\phi_t\}_{t \in \R}$ on $M$ and assume only that $\phi$ is (uniquely) ergodic; by \cite{AGH},  this equivalently means that the linear flow on the toral factor is \emph{completely irrational} (see \S~\ref{sec:prelim_nilmfd}). In the following we will write \emph{parabolic nilflows} as a shortening for a nilflow on a nilmanifold of step at least $2$ (since the step $1$ or Abelian case gives rise to \emph{elliptic} flows).  

Our main result shows that, within a dense class of smooth-time changes of ergodic parabolic nilflows, mixing arises as soon as the time-change is \emph{measurably non-trivial} (i.e.~it is not cohomologous to a constant with a measurable transfer function).

\begin{theorem}\label{thm:main}
For any nilmanifold $M$ of step at least $2$, there exists a dense set $\mathscr{P}$ of smooth functions such that, for any completely irrational nilflow on $M$, the time-changes generated by positive elements of $\mathscr{P}$ are either \emph{measurably trivial} or \emph{mixing}.
\end{theorem}

The dense class of time-changes in the statement of Theorem \ref{thm:main} will be described in detail later, in section~\S \ref{sec:induction}.
The assumption that the step is at least~$2$ is needed to exclude the abelian case of completely irrational flows on tori: in that case, non-singular, smooth  time-changes are typically \emph{not} mixing (under a full measure Diophantine condition on the frequencies) by classical KAM-type results, see e.g.~\cite{Kolmogorov}. Thus, if the step is one, we are in the elliptic (and non-parabolic) world. 
We stress that we do \emph{not} require any Diophantine condition on the frequencies of the toral factor, only complete irrationality: as soon as the nilflow is (uniquely) ergodic, non-trivial time-changes in our class are mixing.

Our result also
implies that smooth time-changes of a completely irrational nilflow which are \emph{not} measurably \emph{trivial} are \emph{not}  measurably isomorphic to the nilflow (since they are mixing while nilflows are not). Since the converse implication is obvious, we get the following corollary.

\begin{corollary}
{Given any completely irrational nilflow on a  nilmanifold of step at least $2$, a time-change 
within the dense set given by Theorem~\ref{thm:main} is \emph{measurably  conjugated} (isomorphic) to the original flow  \emph{if and only if}  it is \emph{measurably trivial}.}
\end{corollary}

The conclusion one might want to draw from the main  result of this paper 
 (as well as the results in \cite{KLU} and \cite{FF4} on disjointness for time-changes of horocycle flows)  is hence that (homogeneous) nilflows (as well as the classical horocycle flow)   are indeed \emph{not generic} examples of parabolic flows and 
display "{\em exceptional}" properties. As soon as homogeneity and extra structures (such as toral factors, or isomorphisms between time-$T$ maps) are broken by a perturbation, the expected  "{\em generic}" parabolic features 
 indeed do emerge.

\subsection{Open problems}\label{sec:open}

Our work leaves open some natural questions. 

\smallskip
Both in \cite{AFU} and \cite{Rav2}, as well as in Theorem~\ref{thm:main},  the dynamical dichotomy in the result (between mixing and trivial time-changes) is only claimed within a  special (sub)class of time-changes. Even though the class that we consider is  \emph{dense} in the smooth category (in the $\mathscr{C}^{\infty}$ norm),  a natural question (which is already open even in the Heisenberg case), is to consider (more) \emph{general} smooth time-changes. 

\smallskip
\noindent {\it Problem 1}:  Is \emph{any} measurably non-trivial smooth time-change of a uniquely ergodic 
parabolic (even Heisenberg) nilflow   {weakly mixing}?  mixing?  

\smallskip
{ For renormalizable flows, it is possible to derive {\it cocycle rigidity} results   from results on growth of ergodic integrals of smooth function which are not smooth coboundaries, see \cite{AFU,FF4}. Hence, an \emph{effective} characterization of mixing time-changes  can be given in terms of vanishing of invariant distributions.   For general nilflows,  contrary to the horocycle or Heisenberg case, we are unfortunately not able to explicitly  describe the set of measurably trivial time-changes. }

We therefore pose the following problem.

\smallskip
\noindent {\it Problem 2}:  Give an \emph{effective} description of the class of measurably trivial, non mixing time-changes for higher step nilflows. %section \ref{}.

\smallskip
For Heisenberg nilflows, the results in \cite{AFU} have been recently strengthened by  A.~Kanigowski and G.~F.~in \cite{FK1} and in~\cite{FK2}: under a full measure Diophantine condition,   \emph{quantitative} mixing estimates are given in \cite{FK1} for a larger  class  (in fact, residual) of time-changes, while for  mixing Heisenbeg nilflows of {\it bounded type} multiple mixing is shown in ~\cite{FK2}.
 It it hence natural to ask whether also  these results extend to higher step nilflows.
 
\smallskip
\noindent {\it Problem 3}:  Under  a Diophantine condition on the frequencies of the toral factor of an ergodic, parabolic nilflow, prove quantitative mixing estimates. 
\smallskip

\smallskip
\noindent {\it Problem 4}:  Under similar (or stronger) assumptions, prove that  mixing time-changes are  mixing of all orders. 
\smallskip

We warn the reader that the above Problems 2, 3, 4 seem hard, because of the lack of fine quantitative estimates on ergodic averages discussed above. 

Let us remark that \emph{multiple mixing} is shown in \cite{FK1} by proving  the Ratner property (which, combined with mixing, implies mixing of all orders). However, the Ratner property is not known to hold even for Heisenberg nilflows which are not of bounded type.

\subsection{Strategy of the proof}
Results on mixing  for time-changes of homogeneous parabolic flows, as well as parabolic surface flows, are all based on a geometric mechanism known as \emph{shearing}: short segments \emph{transversal} to the flow, pushed by the flow, get \emph{sheared}  in the direction of the flow\footnote{Shearing (in the direction of the flow) was for example exploited by Marcus to prove mixing for smooth time changes of horocycle flows \cite{Ma1},  and it provides the basis for  all the mixing results in the context of area-preserving flows (e.g.  in \cite{Ko2, Fa2, KS,  Ul1, Rav1, CW}).} (or in a direction which commutes with the flow). 
When the curves are sheared \emph{in the flow direction} and are asymptotically approximated by flow trajectories, this allows in particular  to prove (quantitative)  mixing by exploiting (quantitative) equidistribution of the trajectories of the (uniquely ergodic) flow. 

 Our result is also  based on a \emph{mixing-via-shearing} argument, but the source of shearing is more subtle. Indeed, recall that nilflows are only \emph{partially parabolic}, hence there are \emph{central} directions which are \emph{not} sheared before the time-change. However,  with a carefully chosen inductive procedure, we are able to either show that in these directions shearing is created by the non-trivial time-change, or the nilflow can be seen as an extension over a lower dimensional nilflow.  {In this case mixing can be lifted through shearing in the central direction (a phenomenon that we call \emph{wrapping in the fibers}), which is exactly the mechanism responsible for the  mixing property of nilflows, relative to the elliptic toral factor.}

We will now explain the main ideas of the proof and the new difficulties which arise in the general case and were not present in the case of Heisenberg nilflows~\cite{AFU} or of quasi-Abelian filiform nilflows~\cite{Rav2}.  The starting point in \cite{AFU} is the representation of a (time-change of a) Heisenberg nilflow  as a special flow over a skew-traslation on a two dimensional torus. For the class of time-changes  considered (which essentially consists of trigonometric polynomials) one can show the that curves in the $1$-dimensional central isometric direction, 
 pushed via the flow, get sheared. This gives the geometric shearing mechanism which then allows to prove mixing. 

A~similar strategy is also used by Ravotti in~\cite{Rav2} to prove 
mixing in the case of  quasi-Abelian filiform nilflows. 

Contrary to the quasi-Abelian case, for a general nilflow  the natural sections are isomorphic to non-toral nilmanifolds and return maps are niltranslations, which are more difficult to handle explicitly. Furthermore, an additional difficulty of the general case, which is not present in  the quasi-Abelian filiform class,  is that the  center can be \emph{higher dimensional}. The key idea  is still to study (short) curves in \emph{a} central direction (to choose carefully, so that it is part of a \emph{Heisenberg triple}, see  
 \S~\ref{sec:induction}) and consider  their pushforward by the flow. Here \emph{two} possible scenarios appear: either there is \emph{shearing}, and one can try to directly prove mixing 
or  it is also possible that shearing does \emph{not} occur (this is what we call the \emph{coboundary case}, see section \ref{sec:coboundary_case}). In this case, the idea is to \emph{quotient out}  central toral fibers, chosen appropriately
(see section~\ref{sec:induction}) 
and find a \emph{factor} which is a time-change of a nilflow on a lower dimensional nilmanifold. In this case, if the factor is mixing, one can "\emph{lift}" the mixing property by the  mechanism of  \emph{wrapping in the fibers}
mentioned above:  if a central curve in the factor is pushed by the factor flow, its lift wraps in the toral fiber \emph{faster} than the equidistribution speed in the factor (this is essentially a consequence of the nilflow 
filtration structure).

To implement this strategy, we  use an inductive argument on the dimension $\text{dim} [G,G]$ %=  \text{ dim}[\mathfrak g,\mathfrak g] $, where  $\mathfrak g$ denotes the Lie algebra of the nilpotent Lie group $G$. 
of the commutator subgroup of the nilpotent group $G$.
%\textcolor{red}{We could add the remark that the same scheme is already in Davide' s paper, although the
%implementation is simpler. G.F.}
A delicate point is how to choose a \emph{sequence} of quotients on which to apply the induction. Here the notion of \emph{Heisenberg triples} (defined in \S\ref{sec:induction}, see Definition~\ref{def:Heisenbergtriple}) plays a key role. Each time we quotient by the (toral) closure of a central flow which belongs to a Heisenberg triple. The assumption that the  time-change is non-trivial guarantees that, before reaching the ``base" case, one has to find a nilflow factor where shearing occurs. The construction of the \emph{tower} of nilflow factors considered is described in detail in section~\ref{sec:induction}.

This inductive presentation of the nilflow on a nilmanifold as tower of toral extensions dictates also the class of time-changes for which we can prove mixing. As in \cite{AFU}, it is crucial for the shearing estimates that the time-change has polynomial features (in the sense that derivatives should be controllable in terms of the function itself). Hence, the class that we consider consists of time-changes which, at each level of the tower, behave in each toral fiber  like trigonometric polynomials (see  Definition~\ref{def:trig_pol_tower} in \S\ref{sec:induction} and also Definition \ref{def:trig_pol} in \S\ref{sec:trig_pol}). 
This produces a $\mathscr{C}^\infty$ dense class of time changes in which we have a dynamical dichotomy, i.e.~either the time change is trivial, or it is %weak mixing and automatically 
mixing.

Let us remark that the choice of abandoning the previous set-up from \cite{AFU}, \cite{Rav2}  based on special flow representations of nilflows was motivated, in addition to the greater elegance and simplicity of the arguments,  by significant technical  difficulties which arise in the general higher step case  in working with special flow representations.
%\emph{not} passing through a Poincar\'e section,  in addition to an (a posteriori) motivation of  elegance and simplicity of arguments, is  crucially  aimed at bypassing technical  difficulties which arise in the representation of nilflows as special flows. 
 In fact,  the relation between the discrete time of the Poincar\'e return map and the continuous time of the nilflow is problematic in the higher step case because of the distributional obstructions to solving cohomological equations for non-toral nilflows and the related deviation of ergodic averages from the mean. 
 However, these difficulties are not completely side-stepped, and even in the current approach significant technical work is devoted to overcome the oscillation of the time-change function on the central fibers, an issue that does not arise in the Heisenberg case.

%\emph{not} passing through a Poincar\'e section,  in addition to an (a posteriori) motivation of  elegance and simplicity of arguments, is  crucially  aimed at bypassing technical  difficulties which arise in the representation of nilflows as special flows. 
% In fact,  the relation between the discrete time of the Poincar\'e return map and the continuous time of the nilflow is problematic in the higher step case because of the distributional obstructions to solving cohomological equations for non-toral nilflows and the related deviation of ergodic averages from the mean.  

We underline that the present paper \emph{supersedes} the previous two results  proved in \cite{AFU} and \cite{Rav2}: not only it gives an \emph{independent} proof of prevalence of mixing also among Heisenberg and a larger class of  time-changes than the one considered in \cite{Rav2} for quasi-Abelian filiform nilflows, but it also follows a much more streamlined approach. The arguments, indeed,  drawing on ideas of \cite{AFU} and \cite{Rav2}, recast them in a geometric framework derived from \cite{FU} (already
adpated to Heisenberg nilflows in \cite{FK1}). % which we believe addition to the greater elegance and simplicity of the argument. 
  In this more  \emph{intrinsic} framework (neither section nor coordinate-dependent), %abandon the special flow representation of \cite{AFU} and \cite{Rav2} to 
we work directly on the manifold, and prove shearing by analyzing  the pushforwards of central curves.  Consequently, the proof of mixing is also less intricate  (instead than analyzing the pushforwards of  partitions into small pieces of curves, we can analyze the pushforwards of a single whole curve and directly use integral estimates).

%\textcolor{orange}{ Ma perche' Davide non ha questo problema?? Questa spiegazione non e' chiara su questo
%punto}.
%In section \ref{sec:structure} at the end of this introduction, we outline where the ideas in this outline appear in the sections of the paper. 

\subsection{Structure of the paper}
In Section~\ref{sec:background} we recall basic definitions, such as time-changes and coboundaries  (\S\ref{sec:basic}), as well as some basic material on nilflows on nilmanifolds (\S\ref{sec:prelim_nilmfd}). In Section~\ref{sec:towers_times},  
we  define the dense class of time-changes in Theorem~\ref{thm:main}. This requires also building the tower of extensions which will be used for the induction (in \S\ref{sec:induction}). %Towers with the desired properties are build using the notion of Heisenberg triple and the inductive Lemma \ref{lemma:Heisenberg} on their existence. 
In Section \ref{sec:tools_for_mixing} we state and prove a number of results which will be used as tools to prove mixing, in particular two lemmas which reduce mixing to a statement about shearing (Lemma \ref{lemma:correlationseasy}  and Lemma \ref{lemma:correlations}  in \S\ref{sec:shearing}) and the computation of pushforwards of curves along the flow (Corollary~\ref{corollary:push_forward} in \S\ref{sec:pushing_curves}). We also prove a result on growth of ergodic sums for functions which are \emph{not} coboundaries, which will provide one of the two shearing mechanisms for mixing.  
The heart of the proof of Theorem~\ref{thm:main} is presented in Section~\ref{sec:proof}, where the two different mechanisms for shearing (the ``growth of ergodic sum'' and the ``wrapping in the fibers'') are exploited in order to prove that mixing holds for some factor in the tower and can be then deduced for the original flow. 
The Appendices contain the proofs of some technical lemmas.

\section{Background}\label{sec:background}
Let us recall for the convenience of the reader some basic definitions and background.
% (on mixing and time-changes in \S\ref{sec:basic}) and some basic definitions and background on nilflows, in \S\ref{sec:prelim_nilmfd}.

\subsection{Basic definitions: mixing, time-changes, coboundaries.}\label{sec:basic}

Let $\phi: = \{\phi_t\}_{t \in \R }$  be a measurable flow  on a pro\-ba\-bi\-li\-ty space $(M,\mu)$.  We recall that $\phi$  is said to be \emph{mixing} if for each pair of measurable sets $A$, $B\subset M$, one has 
$$
\lim_{t\rightarrow \infty} \mu( \phi_t(A)\cap B)=\mu(A)\mu(B)\,,
$$
and \emph{weak mixing} if, for each pair of measurable sets $A$, $B \subset M$,
$$
 \lim_{t\to \infty}  \frac{1}{t} \int_0^t \vert \mu( \phi_s(A)\cap B)- \mu(A)\mu(B) \vert \mathrm{d} s  = 0.
$$

\smallskip

 A flow $\{\widetilde{\phi}_t \}_{t \in \R}$ is called a \emph{re\-pa\-ra\-me\-tri\-za\-tion} or a  
\emph{time-change} of a flow  $\{{\phi}_t \}_{t \in \R}$ on $M$  if there exists a measurable
function $\tau \colon M \times \R \rightarrow \R$  such that for all  
$x \in M$ and $t \in \R$ we have $\widetilde{\phi}_{\tau(x,t)}(x) = \phi_t(x) $. Since
$\{\widetilde{\phi}_t \}_{t \in \R}$ is assumed to be a flow 
the function $\tau(x,\cdot) \colon \R \rightarrow \R$ is an \emph{additive cocycle}
over the flow $\{{\phi}_t \}_{t \in \R}$, that is, it satisfies the cocycle identity:
$$
\tau(x, s+t ) = \tau( {\phi}_s(x), t) + \tau(x,s)\,, \quad \text{ for all }  x\in M\,, \,\, s,t \in \R\,.
$$
If $M$ is a manifold and $\{{\phi}_t \}_{t \in \R}$  is a smooth flow, we will say that 
$\{\widetilde{\phi}_t \}_{t \in \R}$ is a smooth re\-pa\-ra\-me\-tri\-za\-tion if the cocycle $\tau$ is a 
smooth function. By the cocycle property, a smooth cocycle is uniquely determined by
its infinitesimal generator, that is the function $\alpha_\tau \colon M \to \R$ defined by the formula:
$$
\alpha_\tau(x) := \frac{\partial \tau}{\partial t}(x,0) \,, \quad \text{ for all } x\in M\,. 
$$
In fact, given any positive function $\alpha \colon M \to \R^+$, the formula
$$
\tau_\alpha (x,t) := \int_0^t  \alpha( \phi_s(x)) \diff s \,, \quad \text{ for all}\, (x,t)\in M\times \R\,
$$
defines a cocycle over the flow $\{\phi_t\}_{t\in \R}$ with infinitesimal generator $\alpha$.

The infinitesimal generators $V$ and $X$ of the flows $\{\widetilde{\phi}_t \}_{t \in \R}$ and $\{\phi_t \}_{t \in \R}$  respectively are related by the identity:
$$
X = \left. \frac{\diff \phi_t}{ \diff t} \right\vert_{t=0} =\left. \alpha_\tau\, \frac{\diff \widetilde{\phi}_t}{\diff t} \right\vert_{t=0} = \alpha_\tau V\,, \quad \text{i.e.\ \ } V = \frac{1}{\alpha_\tau}X.
$$

An additive cocycle $\tau \colon M \times \R \to \R$ over the flow $\{{\phi}_t\}_{t\in \R}$ is called a measurable (respectively smooth)  \emph{coboundary} if there exists a measurable (respectively smooth)  function $u \colon M \to \R$, called the \emph{transfer function},  such that
$$
\tau(x,t) =  u \circ {\phi}_t(x) - u(x)\,, \quad \text { for all } (x,t)\in M\times \R\,.
$$ 

The additive cocycle $\tau$ is a measurable (smooth) coboundary if and only if its infinitesimal generator $\alpha_\tau$ is a measurable (smooth) \emph{coboundary} for the infinitesimal generator $X$ of the flow $\{{\phi}_t\}_{t\in \R}$, that is, if there exists a measurable (smooth) function $u \colon M \to \R$, also called the transfer function, such that $X u = \alpha_\tau$.

Two additive cocycles  are said to be measurably (respectively smoothly) \emph{cohomologous} if their difference is a measurable (respectively smooth) coboundary in the above sense. 
A cocycle is said to be an \emph{almost coboundary} if it is cohomologous to a constant cocycle.

An elementary, but fundamental, result establishes that time-changes given by  
cohomologous cocycles are  isomorphic (see for example \cite{Kat:CC2}, \S 9, for related results). 
\begin{lemma} 
\label{lemma:triviality} Let $\tau_1, \tau_2 \colon M \times \R \to \R$ be cohomologous cocycles over the flow $\{{\phi}_t\}_{t\in \R}$, that is, assume that there exists
a (measurable) function $u: M\to \R$ such that 
$$
\tau_1(x,t)  -\tau_2(x,t) =  u \circ {\phi}_t(x) - u(x)\,, \quad \text { for all } (x,t)\in M\times \R\,.
$$
Then the time-changes $\{\widetilde {\phi}^{(1)}_t\}_{t\in \R}$ and $\{\widetilde{\phi}^{(2)}_t\}_{t\in \R}$ of the flow $\{{\phi}_t\}_{t\in \R}$ defined by the conditions
$$
\widetilde \phi^{(i)}_{\tau_i(x,t)} (x) =  \phi_t(x)   \,, \quad \text { for all } (x,t)\in M\times \R \text{ and for } i =1,2
$$
are  conjugated by the (measurable) isomorphism $\psi:M\to M$ defined  as 
$$
\psi(x) := \widetilde \phi^{(1)}_{u(x)} (x) \,, \quad \text{ for all }  x\in M\,.
$$
In fact, we have the identity
$$
\widetilde \phi^{(1)}_{\tau}  =   \psi \circ \widetilde \phi^{(2)}_{\tau}  \circ \psi^{-1} \,, \quad \text{ for all } \tau \in \R\,.
$$
The isomorphism $\psi$ maps by push-forward every invariant measure for the time-change 
$\widetilde \phi^{(2)}$  onto the unique invariant measure for the time-change $\widetilde \phi^{(1)}$ which has the same 
transverse measure.
\end{lemma}

\begin{proof} 
Since $\tau_2$ is a cocycle,  for all $(x, \tau)\in M\times \R$,  there exists $t\in \R$  (which depends on $x\in M$)
such that  $\tau_2(x,t) =\tau$. We then claim that we have the following chain of identities
$$
 (\psi \circ \widetilde \phi^{(2)}_{\tau}) (x)  =  (\psi \circ \phi_t) (x)   =   \widetilde \phi^{(1)}_{ \tau_1(x,t) + u(\phi_t(x)) } (x) =
    \widetilde \phi^{(1)}_{ \tau_2(x,t) + u(x) } (x) =   (\widetilde \phi^{(1)}_{ \tau }  \circ \psi )(x)\,,
$$
where the first follows from the definition of the time change  $(\widetilde \phi^{(2)}_{t})_{t\in \R}$
and the choice of $\tau= \tau_2(x,t)$, the second from the definition of the isomorphism $\psi$ and the definition of the
time-change $(\widetilde \phi^{(1)}_{t})_{t\in \R}$ ,
the third is given by the cohomological relation between the cocycles, and finally the fourth (last) follows again from the 
definition of the isomorphism $\psi$.

By its definition, the isomorphism $\psi$ fixes every leaf of the common orbit foliation of the flow $\{\phi_t\}_{t\in \R}$ and of its 
time-changes, hence it fixes all of its transverse invariant  measures. Since every invariant measure for
a flow can be locally decomposed as a product of a transverse invariant measure and of the 
Lebesgue  measure $\diff t$ along the orbits, the statement  concerning the action of the isomorphism $\psi$ on the
invariant measures follows immediately.
\end{proof}

The regularity class of the isomorphisms depends on the regularity class of the transfer function. A time-change defined by a measurable (smooth)  almost coboundary is called \emph{measurably (smoothly)  trivial}.  We recall that the existence of a measurable (smooth) isomorphism of time-changes does not implies that the time-changes are cohomologous.  The problem of deriving cohomological relations from isomorphism of time-changes is an important question on the rigidity of time-changes. Results of this type were first proved for the classical horocycle flow by M.~Ratner \cite{Rat} (and  were completed in \cite{FF4}), but are not known for general unipotent flows. 

\subsection{Preliminaries on nilmanifolds}\label{sec:prelim_nilmfd}
Let~$ \mathfrak g $ be a $k$-step nilpotent real Lie algebra ($k\ge 2$) with a minimal set of generators $ \mathcal E:= \{E_{1}, \dots , E_{n}\}\subset \mathfrak g$.  For all $j\in \{1, \dots, k\}$, let $\mathfrak g_{j}$,  denote the {\it descending central series} of $\mathfrak g$:
\begin{equation}
\label{eq:descser}
\mathfrak g_{1}=\mathfrak g, \ \mathfrak g_{2}=[ \mathfrak g,\mathfrak g], \dots, \ \mathfrak g_{j}=[ \mathfrak g_{j-1},\mathfrak g], \dots, \ \mathfrak g_{k}\subset Z(\mathfrak g)\,,
\end{equation}
where $Z(\mathfrak g)$ is the center of $\mathfrak g$. 

Let $G$ be the connected and simply connected nilpotent Lie group with Lie algebra $\mathfrak g$. The corresponding Lie subgroups $G_j=\exp \mathfrak g_j=[G_{j-1},G]$ form the descending central series of $G$. 
Let $\Gamma$ be a lattice in $G$. It exists if and only if $G$ admits rational structure constants (see, for example, \cite{Ra, CG}).

A (compact) {\em nilmanifold} is by definition a quotient manifold $M:=\Gamma\backslash G$ with $G$ a nilpotent Lie group and $\Gamma\subset G$ a lattice. 
On a nilmanifold $M=\Gamma\backslash G$, the group $ G $ acts on the right transitively by right multiplication. By definition, the {\it nilflow $\phi^X$ generated by $X\in \mathfrak g$ }is the flow obtained by the restriction of this action to the one-parameter subgroup $(\exp t X)_{t \in \R}$ of $G$:
\begin{equation}
\label{eq:nilflow}
\phi^X_t(\Gamma x ) = \Gamma x \exp( t X ). 
\end{equation}
It is plain that nilflows on $\Gamma\backslash G$ preserve the probability measure $\mu$ on $\Gamma\backslash G$ given locally by the Haar measure. To simplify the notation, the vector field on $\Gamma\backslash G$ generating the flow $\phi^X$ will also be indicated by $X$.

Every nilmanifold is a fiber bundle over a torus. In fact, the group~$G^{Ab}=G/[G,G]$ is Abelian, connected and simply connected, hence isomorphic to~$\R^n$ and~$\Gamma^{Ab}= \Gamma/[\Gamma,\Gamma]$ is a lattice in~$G^{Ab}$. Thus we have a natural projection
\begin{equation}
\label{eq:projection}
p: \Gamma\backslash G\to \Gamma^{Ab}\backslash G^{Ab} 
\end{equation}
over a torus of dimension~$n$. We recall the following:

\begin{theorem}[\cite{Gr}, \cite{AGH}]
\label{th:Green} 
The following properties are equivalent.
 \begin{itemize}
 \item[(i)] The nilflow $\phi^X$ on $\Gamma\backslash G$ is ergodic.
 \item[(ii)] The nilflow $\phi^X $ on $\Gamma\backslash G$ is uniquely ergodic.
 \item[(iii)] The nilflow $\phi^X$ on $\Gamma\backslash G$ is minimal.
 \item[(iv)] The projected flow $\psi^{X^{Ab}}$ on $\Gamma^{Ab} \backslash G^{Ab}$ $\approx \T^n$ is a completely irrational linear flow on~$ \T^n $, hence it is (uniquely) ergodic and minimal.
 \end{itemize}
\end{theorem}
By property $(iv)$ it follows that a nilflow is never (weakly) mixing, since it has a linear toral flow, which has pure point spectrum, as a factor. However, it is possible to prove by methods of representation theory that any nilflow is  {\it relatively mixing}, in the sense that the limit of correlations of functions with zero average along all fibers of the projection in formula~\eqref{eq:projection} is equal to zero.

The \emph{complete irrationality} condition  mentioned in  our main result, Theorem~\ref{thm:main}, and in  $(iv)$ in the above Theorem~\ref{th:Green}, refers to the rational structure of the Abelianized group $G^{Ab}$ determined by the lattice $ \Gamma^{Ab} \subset G^{Ab}$. This structure is in turn determined by the lattice $\Gamma \subset G$ as follows.  Let us first recall the definition of \emph{Malcev basis} (see~\cite{CG}).

\begin{definition}[Malcev basis] 
\label{def:Malcev}
  A Malcev basis  for $\mathfrak{g}$ through the descending
  central series $\mathfrak{g}_j$ and strongly based at $\Gamma$ is a basis $E_1^1, E_2^1, \dots E_{n_1}^1, E_1^2, \dots,
  E_{n_2}^2, \dots, E_1^k, \dots, E_{n_k}^k$, (with
  $n_1=n$) of $\mathfrak{g}$ 
  satisfying the following
  properties:
  \begin{itemize}
  \item[(i)] if we drop the first $\ell$ elements of the basis we
    obtain a basis of a subalgebra of codimension $\ell$ of
    $\mathfrak{g}$.
  \item[(ii)] if we set $\mathcal E^j:=\{ E_1^j, \dots,
    E_{n_j}^j\}$ the elements of the set $\mathcal
    E^j\cup\mathcal E^{j+1} \cup \dots \cup\mathcal E^k $
    form a basis of $\mathfrak{g}_j$.
  \item[(iii)] every element of $\Gamma$ can be written as a
    product
    $$
    \exp m_1^1 E_1^1 \dots \exp m_{n_1}^1 E_{n_1}^1\dots
    \exp m_1^k E_1^k \dots \exp m_{n_k}^k E_{n_k}^k
    $$
    with integral coefficients $m_i^j$.
  \end{itemize}
 % \end{sloppypar}
\end{definition}
\noindent The existence of a Malcev basis can be derived by combining the
  proofs of Theorems~1.1.13 and~5.1.6 of~\cite{CG}. A Malcev basis determines a \emph{rational structure}  
  on the Lie algebra  $\mathfrak g$ of the nilpotent Lie group $G$ as follows.

\begin{definition}[Rational and completely irrational vectors]
\label{def:rational} An element $V\in \mathfrak g$ is called  \emph{rational} with respect to the lattice $\Gamma$ if it belongs to the span
over $\Q$ of a Malcev basis strongly based at $\Gamma$, and it is called \emph{rational} with respect to the Abelianized lattice $\Gamma^{Ab}$ if
its projection $V^{Ab} \in \mathfrak g / [ \mathfrak g, \mathfrak g]\approx \R^n$ belongs to the span over
$\Q$ of the projection of the Malcev basis.   
 
An element $V\in  \mathfrak g$ is called \emph{completely irrational}  (with respect to $\Gamma^{Ab}$) if the coordinates of its projection $V^{Ab}$ with respect to a \emph{rational} basis of $\mathfrak g / [ \mathfrak g, \mathfrak g]$, namely a basis given by projections of rational vectors, are linearly independent over $\Q$.
\end{definition}

Thus, if the generators $\{E_1,\dots,E_n\}$ of~$\mathfrak g$ are chosen so that the elements $\{\exp E_1,\dots,\exp E_n\}$ project onto generators $\{\exp E^{Ab}_1,\dots,\exp E^{Ab}_n\}$ of $\Gamma^{Ab}$, then for every $X \in \mathfrak g$  there exists a vector $\Omega_X:=\left(\omega_1(X), \dots,\omega_n(X)\right)\in \R^n$ such that
\begin{equation}
\label{eq:XAb}
X^{Ab} =\omega_1(X) E^{Ab}_1+\dots+\omega_n(X) E^{Ab}_n .
\end{equation}
According the above definition, the element $X$ is called {\em completely irrational (with respect to $\Gamma^{Ab}$)} if the numbers  $\omega_1(X), \dots,\omega_n(X)$ are linearly independent  over $\Q$.

\section{Towers and time-changes}\label{sec:towers_times}
The goal of this section is to define the class of time-changes in which the dichothomy in Theorem~\ref{thm:main} holds. In \S\ref{sec:trig_pol}, we first define trigonometric polynomials relative to a central extension (see Definition~\ref{def:trig_pol}) and state some  of the properties which will be needed later (see \S\ref{sec:prop_trig}). In \S\ref{sec:induction} we then explain how to construct inductively the tower of nilflow extensions we will be working with  and, finally, in \S\ref{sec:time_changes}, we define the class of time-changes that we will consider to prove Theorem~\ref{thm:main} (which consists of trigonometric polynomials relative to the tower, see Definition~\ref{def:trig_pol_tower}).

\subsection{Trigonometric polynomials on nilmanifolds}\label{sec:trig_pol}
Let $M:=\Gamma\backslash G$ be a nilmanifold, with $G$ a nilpotent Lie group with Lie algebra $\mathfrak g$ and $\Gamma\subset G$ a lattice. Let $\z  \subset Z(\mathfrak g)$ be a  $d$-dimensional  \emph{rational} subalgebra; that is, a $d$-dimensional subspace which admits a basis of rational vectors (in the sense of Definition~\ref{def:rational}) and let $H:= \exp \z$ be the corresponding analytic subgroup.  The subgroup $\Lambda:=  H \cap \Gamma$ is a lattice in $H$.  As $H$ is Abelian,
the quotient space $H/\Lambda$ is a $d$-dimensional torus. Letting $\overline{G}:= G/H$ and $\overline{\Gamma}:= \Gamma / (\Gamma \cap H)$ we have a fiber bundle $\pi: M\to \overline{M}$ of the nilmanifold $M$ onto the nilmanifold $\overline{M}:=\overline{\Gamma} \backslash \overline{G}$ whose fibers are precisely the orbits of the torus $T := H/\Lambda$.

We can disintegrate the Haar measure $\mu$ with respect to this fibration as
\begin{equation}
\label{eq:disintegration}
\mu = \int_{\overline{M}} \, \mu_{\pi^{-1}\{{\overline{x}}\}}\, \diff \overline{\mu}(\overline{x}),
\end{equation}
where $\overline{\mu} = \pi_{\ast}\mu$ is the Haar measure on the quotient nilmanifold $\overline{M}$ and the conditional measures $\mu_{\pi^{-1}\{{\overline{x}}\}}$ 
on each fiber of the fibration are equivalent (up to a fixed constant) to the push-forward  of the Lebesgue measure on the torus $T$.

\smallskip
We have an orthogonal decomposition
\begin{equation}\label{eq:orthogonal1}
L^2(M,\diff\mu) = \pi^{\ast}(L^2(\overline{M},\diff {\overline{\mu}})) \oplus [\pi^{\ast}(L^2(\overline{M},\diff {\overline{\mu}}))]^{\perp},
\end{equation}
where $\pi^{\ast}(L^2(\overline{M},\diff {\overline{\mu}}))$ is the space of pull-backs of functions over $\overline{M}$, which coincides with the space of functions over $M$ which are $T$-invariant (and hence constant on the toral fibers).  

In the following, we may identify $\z$ with $\R^d$ and let $(\Phi^{\z}_{\bf t})_{{\bf t}\in \R^d}$ denote the action of $\z \approx \R^d$ on the nilmanifold $M$,
and  identify functions on $\overline{M}$ with their pull-backs, which are defined on $M$ so that they are constant with respect to the action $(\Phi^{\z}_{\bf t})_{{\bf t}\in \R^d}$ .

\smallskip
The space $L^2(M, \diff\mu)$ splits as an orthogonal Hilbert sum of eigenspaces $H_{{\bf v}} (\z)$ of the characters of the torus $T$. 
That is, if we let $\Lambda^\ast$ denote the dual lattice of $\Lambda$, we have a Fourier decomposition 
\begin{equation*}
L^2(M,\diff\mu) = \bigoplus_{{{\bf v}} \in \Lambda^*} H_{{\bf v}} (\z),
\end{equation*}
where
$$
H_{{\bf v}} ( \z) = \left\{ f \in L^2(M,\diff\mu) : f \circ \Phi^{ \z} _{{\bf t} }= 
\exp(2 \pi \imath  \langle {{\bf v}}, {\bf t}  \rangle)  f  \right\}.
$$
In particular, for every function $f \in \mathscr{C}^0(M)$, there exists a family  $(f_{{\bf v}})_{{{\bf v}}\in  \Lambda^\ast}$ of continuous functions such that we have the formula
\begin{equation}
\label{eq:Fourier_exp} 
f \circ \Phi^{\z} _{\bf t} (x)  = \sum_{{{\bf v}} \in \Lambda^\ast} f_{{\bf v}} (x) \exp(2\pi\imath \langle {{\bf v}}, {\bf t}\rangle ) \,, \quad \text{ for all } (x, {\bf t}) \in M \times \z\,.
\end{equation}
We remark that the function $f_{0} \in H_{0} ( \z) = \pi^{\ast}(L^2(\overline{M},\diff {\overline{\mu}}))$ is the pull-back of a continuous function on $\overline{M}$ and hence is constant on all fibers $\pi^{-1}\{{\overline{x}}\}$, while functions in $H_{0} ( \z) ^\perp=
\bigoplus_{{{\bf v}} \in \Lambda^* \setminus \{ 0 \}} H_{{\bf v}} (\z)$ are relative trigonometric polynomials \emph{without} the term which is constant on the fibers $\pi^{-1}\{{\overline{x}}\}$.  
Thus,
  we can write the orthogonal decomposition given in~\eqref{eq:orthogonal1} also as
\begin{align}\label{eq:orthogonal}
&L^2(M,\diff\mu) = H_{0} ( \z) \oplus H_{0} ( \z)^{\perp},  \\ \nonumber \text{where } \quad &  H_{0} ( \z):= \pi^{\ast}(L^2(\overline{M},\diff {\overline{\mu}})), \\  \nonumber&\text H_{0} ( \z) ^\perp:=[\pi^{\ast}(L^2(\overline{M},\diff {\overline{\mu}}))]^{\perp}= \bigoplus_{{{\bf v}} \in \Lambda^* \setminus \{ 0 \}} H_{{\bf v}} (\z).
\end{align}

%Let $\mathscr{C}^0(M)$ denote the space of continuous functions $f \colon M \to \C$.

\begin{definition}[Trigonometric polynomial with respect to $\z$]
\label{def:trig_pol} 
Let $\z \subset Z(\mathfrak g)$ be any rational subspace with respect to $\Gamma$. A function $f \in \mathscr{C}^0(M)$ is called a \emph{trigonometric polynomial with respect to $\z$} if in the expansion in the above formula \eqref{eq:Fourier_exp} the family $(f_{{\bf v}})_{{{\bf v}}\in  \Lambda^\ast}$ has finite support.

For all ${{\bf v}} \in \Lambda^{\ast}$, let $d({{\bf v}}) \geq 0$ be the nonnegative integer defined by $\langle {{\bf v}}, \Lambda \rangle = d({{\bf v}}) \Z$. The \emph{degree} of a trigonometric polynomial $f$ is 
$$
\deg(f) = \max \{ d({{\bf v}}) : f_{{\bf v}} \not\equiv 0 \}.
$$
\end{definition}
Trigonometric polynomials with respect to $\z$ will also be called \emph{relative trigonometric polynomials} (when the toral fibers are clear from the context).  

\subsection{Basic properties of relative trigonometric polynomials}\label{sec:prop_trig}

The following lemma provides an estimate on the measure of sub-level sets of \emph{relative} trigonometric polynomials on nilmanifolds (see also \cite[Lemma 4]{AFU} for the case of trigonometric polynomials and Heisenberg nilflows).

\begin{lemma}[Level sets of relative trigonometric polynomials]
\label{thm:sublevel_sets_polynomials}
For each $m \geq 1$ there exist constants $\Delta_m$ and $d_m>0$ such that the following holds.
Let $f \in \bigoplus_{{{\bf v}} \in \Lambda^* \setminus \{ 0 \}} H_{{\bf v}} (\z)$ be a trigonometric polynomial with respect to $\z$ of degree $m$ without constant term in the fibers.  

For any $C \geq 0$ and $\varepsilon \geq 0$, if 
$$
\mu \left( x \in M : \sum_{{{\bf v}} \in \Lambda^\ast \setminus \{0 \}} |f_{{\bf v}}(x)| \leq C \right) \leq \varepsilon,
$$ 
then, for all $\delta >0$,
$$
\mu \left( x \in M : | f(x) | \leq \delta \sum_{{{\bf v}} \in \Lambda^\ast \setminus \{0 \}} |f_{{\bf v}}(x)| \right) \leq \Delta_m \delta^{d_m} + \varepsilon.
$$
Furthermore,
$$
\mu \left( x \in M : | f(x) | \leq \delta C \right) \leq \Delta_m \delta^{d_m} + \varepsilon.
$$      
\end{lemma}
\begin{proof}
Let us recall that each toral fiber $\pi^{-1} \{ \overline{x} \}$ can be identified with the torus $T=H / \Lambda$ and the measure $\mu_{ \pi^{-1}\{\overline{x}\}}$ coincides (up to a constant) with the Lebesgue measure on $T$. By classical results on level sets of trigonometric polynomials, see, e.g., \cite[Theorem 1.9]{Br}, for each $m \geq 1$, if $p$ is a trigonometric polynomial defined on $\pi^{-1} \{ \overline{x} \}$ of degree $m \geq 1$, then there exist constants $D_{m,p}$ and $d_m>0$ such that for every $\delta >0$, we have 
\begin{equation}
\label{eq:estimate_for_p}
\mu_{\pi^{-1} \{ \overline{x} \} } \left( y \in \pi^{-1} \{ \overline{x} \} : | p(y) | \leq \delta \right) \leq D_{m,p} \delta^{d_m}.
\end{equation}
The constant $d_m$ depends only on $m$ and the constant $D_{m,p}$ depends continuously on~$p$.

Let us define 
$$
E_{C} := \left\{ x \in M : \sum_{{{\bf v}} \in \Lambda^\ast\setminus \{0 \}} |f_{{\bf v}}(x)| \leq C \right\}. 
$$
Recall that by assumption $\mu(E_{C})\leq \varepsilon$ and notice that, since $|f_{{\bf v}}(x)|$ is $\Phi^{\z} _{\bf t} (x)$-invariant, the set $E_{C}$ is a union of full fibers.
In particular, the normalized trigonometric polynomial 
$$
{\widehat f}(x) := \frac{f(x)}{\sum_{{{\bf v}} \in \Lambda^\ast \setminus \{0 \}} |f_{{\bf v}}(x)|}
$$
is well-defined on each fiber $\pi^{-1} \{ \overline{x} \} \subset M \setminus E_{C}$, and its restriction ${\widehat f}|_{\pi^{-1} \{ \overline{x} \}}$ to $\pi^{-1} \{ \overline{x} \}$, because of the normalization, lies in a compact subset of the space of trigonometric polynomials of degree $1 \leq \deg ({\widehat f}|_{\pi^{-1} \{ \overline{x} \}}) \leq m$.
Thus, there exist uniform positive constants $\Delta_m$, $d_m$ such that the estimate~\eqref{eq:estimate_for_p} holds for any ${\widehat f}$ as above.
By \eqref{eq:disintegration} and Fubini's Theorem, we deduce that the set
$$
M_{\delta} := \left\{ x \in M \setminus E_{C} : |{\widehat f}(x)| \leq \delta \right\} = \left\{ x \in M \setminus E_{C} : |f(x)| \leq \delta \sum_{{{\bf v}} \in \Lambda^\ast\setminus \{0 \}} |f_{{\bf v}}(x)| \right\} 
$$
has measure at most $\Delta_m \delta^{d_m}$. 
Therefore, for all $x \in M \setminus (M_{\delta} \cup E_{C})$, we have $$|f(x)| > \delta \sum_{{{\bf v}} \in \Lambda^\ast \setminus \{0 \}} |f_{{\bf v}}(x)| > \delta C.$$ Using that by assumption $\mu (M_{\delta} \cup E_{C} ) \leq \Delta_m \delta^{d_m} + \varepsilon$, the proof is complete.
\end{proof}

\subsection{Algebraic induction towers}\label{sec:induction} 
Let $M$ and $\{\phi^X_t\}_{t\in\mathbb{R}}$ be a nilmanifold and a nilflow satisfying the assumptions of Theorem~\ref{thm:main}. 
In this section we show that we can present $M$ as a \emph{tower} of central extensions of nilmanifolds, so that at each step in going down the tower we quotient out a central fiber, chosen in a way that will be convenient for us to show prevalence of mixing by induction in the later sections. 

\smallskip
The following notion of \emph{Heisenberg triple} is central in determining how to build the tower.
\begin{definition}[Heisenberg triple]\label{def:Heisenbergtriple}
Let $\mathfrak{g}$ be a nilpotent Lie algebra of step $k \geq 2$. A triple $(X,Y,Z)$ of elements of  $\mathfrak{g}$ is called a \emph{Heisenberg triple} if $Z \in \mathfrak g_k \subset Z(\mathfrak g)$ and $[X, Y]=Z$.
\end{definition} 

Our towers  will be constructed by successive quotients of the nilmanifold with respect to the action of tori of  the minimal dimension tangent to the central elements of a finite sequence of Heisenberg triples. For this purpose we introduce the following definition. 

\begin{definition}
For any nilpotent group $G$ with Lie algebra $\mathfrak g$ and for  any lattice $\Gamma <G$, for any central element  $Z \in Z(\z)$  the
\emph{rational envelope} of $Z$  is the smallest rational subspace $[Z]_\Gamma \subset Z(\z)$   with respect to $\Gamma$  such that $Z\in [Z]_\Gamma$.
\end{definition}

A fundamental algebraic result, proven below, states that any nilmanifold  $M := \Gamma \backslash G$ is given by a finite tower of Heisenberg extensions,  
in the following sense. 

\begin{definition}
\label{def:tower}
Let $M = M^{(0)} = \Gamma^{(0)} \backslash G^{(0)}$ be a nilmanifold and let $X = X^{(0)} \in \mathfrak g$  be an element not in $ [\mathfrak g, \mathfrak g]$.  
A \emph{tower $\mathcal{T}_{M,X}$ (of height $h\in \N$) of Heisenberg extensions for $M$ based at $X$} is a sequence of nilmanifolds $M^{(i)} = \Gamma^{(i)} \backslash G^{(i)}$, projections
$$
M = M^{(0)} \xrightarrow{\pi^{(1)}} M^{(1)} \xrightarrow{\pi^{(2)}} \cdots \xrightarrow{\pi^{(h)}} M^{(h)},
$$ 
and triples $(X^{(i)}, Y^{(i)}, Z^{(i)}) \in (\mathfrak{g}^{(i)})^3$, where $ \mathfrak{g}^{(i)}$ is the Lie algebra of the nilpotent group $G^{(i)}$, such that 
\begin{enumerate}
\item $(X^{(i)}, Y^{(i)}, Z^{(i)})$ is a Heisenberg triple in $\mathfrak{g}^{(i)}$ for all $0 \leq i \leq h$,\vspace{1.3mm}
\item for all $0 \leq i \leq h-1$,  if we set $\z^{(i)}:= [Z^{(i)}]_{\Gamma^{(i)}}$ and $\Lambda^{(i)}:=\log \Gamma^{(i)}\cap  [Z^{(i)}]_{\Gamma^{(i)}}$, we have 
$$\mathfrak{g}^{(i+1)} = \mathfrak{g}^{(i)} / \z^{(i)},\qquad % 
\Gamma^{(i+1)} =  \Gamma^{(i)} / \exp \Lambda^{(i)} =\Gamma^{(i)}/\left(\Gamma^{(i)} \cap \exp \z^{(i)} \right),$$ %$M^{(i+1)} = \Gamma^{(i+1)} \backslash G^{(i+1)}$ 
and $\pi^{(i+1)} \colon M^{(i)} \to M^{(i+1)}$ is the canonical projection, \vspace{1.3mm}
\item  $X^{(i+1)} = \pi^{(i+1)}_{\ast}(X^{(i)})$  for  all $0 \leq i \leq h-1$.
\end{enumerate}
A tower $\mathcal{T}_{M,X}$ is \emph{maximal} if $G^{(h)}$ is Abelian and $M^{(h)}$ is a torus. 
\end{definition}

  The following lemma, which will be used as one step of the inductive construction, will  guarantee the existence of Heisenberg triples at each step of the induction.

\begin{lemma}[Existence of Heisenberg triples\footnote{We owe the argument presented in the proof, which supersedes our original existence proof and yields the conclusion that it is possible to have $[Z]_\Gamma= \z_k$, 
to the anonimous referee).}]
\label{thm:Heisenberg_triple}  
Let $\mathfrak g$ be a nilpotent Lie algebra of step $k \geq 2$. For any completely irrational element $X \in \mathfrak g$  (with respect to $\Gamma^{Ab}$),  there exist 
$Y, Z$ such that $(X,Y,Z)$ is a Heisenberg triple  and $Z$ is completely irrational in $\mathfrak g_k$, in the sense that  $[Z]_\Gamma= \mathfrak g_k$.
\end{lemma}

\begin{proof} 
Let $E_1^1, E_2^1, \dots E_{n_1}^1, E_1^2, \dots, E_{n_2}^2, \dots, E_1^k, \dots, E_{n_k}^k$, with $n_1=n$, be a Malcev basis for $\mathfrak g$ through the descending central series $\mathfrak g_j$ and strongly based at $\Gamma$, as in Definition~\ref{def:Malcev}. 
For any $i, j  \geq 1$, we have $[\mathfrak g_i,  \mathfrak g_j]\subset {\mathfrak g}_{i+j}$, and therefore the map $q : (\mathfrak g_1/\mathfrak g_2)\otimes (\mathfrak g_{k-1}/\mathfrak g_k) \to \mathfrak g_k$ induced by the map 
$$
V \otimes W \in {\mathfrak g}_1 \otimes {\mathfrak g}_{k-1}   \to   [V,W] \in  {\mathfrak g}_k
$$
is well defined and surjective.

Let $b^s_{ij}$ denote structure constants of the Malcev basis defined by the identities
$$
[E^1_i, E^{k-1}_j] = \sum_{s=1} ^{n_k}   b^s_{ij} E^k_s\,, \quad \text{ for all } i=1, \dots, n_1 \text{ and } j=1, \dots, n_{k-1}\,,
$$
and let $\mathbb K$ denote the finite extension $\Q(\omega_1(X), \dots, \omega_n(X))$ (where $\omega_j(X)  $, for $1\leq j\leq n_1$, are defined by \eqref{eq:XAb}).
Let $y_1, \dots, y_{n_{k-1}}$ be real numbers such that the transcendence degree over $\mathbb K$  of the finite extension $\mathbb K(y_1, \dots, y_{n_{k-1}})$
is maximal, namely equal to $n_{k-1}$. We claim that, if we set 
$$
Y = \sum_{j=1}^{n_{k-1}}  y_j E^{k-1}_j \,,
$$
the element $Z=[X,Y] \in \mathfrak g_k$ is completely irrational in $\Gamma \cap \mathfrak g_k$, in the sense that it has rationally independent coordinates with
respect to a rational basis, hence its rational envelope is $[Z]_\Gamma= \mathfrak g_k$.  In fact, by its definition,
$$
Z=   \sum_{s=1}^{n_k}  \sum_{i=1}^{n_1} \sum_{j=1}^{n_{k-1}}  b^s_{ij} \omega_i(X)  y_j  E^k_s   \,,
$$
thus, given any rational linear form $\lambda$ on $\mathfrak g_k$, we have that
$$
\lambda(Z)=0   \Longleftrightarrow     \sum_{j=1}^{n_{k-1}}  \left( \sum_{s=1}^{n_k}  \sum_{i=1}^{n_1} b^s_{ij} \omega_i(X)   \lambda(E^k_s) \right) y_j \,,
$$
which, since $y_1, \dots, y_{n_k}$ are linearly independent over $\mathbb K$, implies  
$$
 \sum_{s=1}^{n_k}  \sum_{i=1}^{n_1} b^s_{ij} \omega_i(X)   \lambda(E^k_s)   =0,   \quad \text{ for all } j=1, \dots, n_k\,,
$$
and in turn, since $\omega_1(X), \dots, \omega_{n_1}(X)$ are linearly independent over $\Q$, this gives that
$$
\lambda ([E^1_i, E^{k-1}_j])= \sum_{s=1}^{n_k}  b^s_{ij}   \lambda(E^k_s)   =0,   \quad \text{ for all } i=1, \dots, n_1 \text{ and } j=1, \dots, n_k\,.
$$ 
Finally, the above condition implies that  for any $V= \sum_{i=1}^{n_1} V_i E^1_i \in \mathfrak g_1$ and $W= \sum_{j=1}^{n_{k-1}} W_j E^{k-1}_j \in \mathfrak g_{k-1}$,
$$
\lambda ([V,W])= \sum_{i=1}^{n_1} \sum_{j=1}^{n_{k-1}}     V_i W_j   \sum_{s=1}^{n_k}  b^s_{ij}  \lambda(E^k_s) =0\,,
$$
hence $\lambda$ vanishes on the range of the map $q$ defined above, which is surjective onto $\mathfrak g_k$. We conclude that 
$\lambda(\mathfrak g_k)=0$, that is $\lambda=0$. The argument is complete.
\end{proof}

Iterating inductively the previous Lemma (as shown in Corollary \ref{corollary:existence_of_towers} below), we can construct   \emph{towers of extensions} corresponding to Heisenberg triples, in the following sense. 

\begin{corollary}
\label{corollary:existence_of_towers}
If $X \in \mathfrak{g}$ is the generator of a uniquely ergodic nilflow on $M$, then there exists a maximal tower of Heisenberg extensions for $M$ based at $X$,
which has height $h=k-1$ and is given by the descending central series of the Lie algebra $\mathfrak g$,
$$
\mathfrak g_1= \mathfrak g \supset \mathfrak g_1 =[\mathfrak g, \mathfrak g] \supset \mathfrak g_j =[\mathfrak g_{j-1}, \mathfrak g] \supset \dots \supset \mathfrak g_k 
\subset Z(\mathfrak g)
$$
as follows.  In the notation of Definition~\ref{def:tower}, for every $j=0,\dots, k-1$,  the algebra ${\mathfrak g}^{(j)}$ is $k-j$-step nilpotent  and the Abelian subalgebra 
$\z^{(j)}$ is the last non-trivial subalgebra $\mathfrak g^{(j)}_{k-j}$ in the descending central series of $\mathfrak g^{(j)}$, and we have
$$
\mathfrak g^{(j)} \equiv  \mathfrak g / (\mathfrak g_k + \dots + \mathfrak g_{k-j+1}) \,,   \quad  \z^{(j)} \equiv  
 \mathfrak g_{k-j} / (\mathfrak g_k + \dots + \mathfrak g_{k-j+1})\,.
$$ 
\end{corollary}
\begin{proof}
The proof follows from Lemma \ref{thm:Heisenberg_triple} by induction on $\dim [\mathfrak{g},\mathfrak{g}]$ or on the nilpotency depth. 
For each $i=0,\dots, h-1$, Theorem~\ref{th:Green} implies that $X^{(i)}  \in \mathfrak g^{(i)}$ is completely irrational (with respect to the Abelianized of $\Gamma^{(i)}$), 
thus one can apply Lemma \ref{thm:Heisenberg_triple} to get a Heisenberg triple $(X^{(i)}, Y^{(i)}, Z^{(i)})$ such that the rational envelope of $Z^{(i)}$ is $[Z^{(i)}]_{\Gamma^{(i)}} = \mathfrak g^{(i)}_{k-i}$, and a nilmanifold $M^{(i+1)}$ with nilpotency depth exactly one less than that of $M^{(i)}$.  In this construction the  dimension $\dim [\mathfrak{g}^{(i+1)},\mathfrak{g}^{(i+1)}]$  is strictly lower  than the dimension $\dim [\mathfrak{g}^{(i)},\mathfrak{g}^{(i)}]$, and the nilpotency depth of $\mathfrak g^{(i+1)}$ is lower than that of $\mathfrak g^{(i+1)}$, thus  this concludes the proof. 
\end{proof}

\subsection{The class of time-changes}\label{sec:time_changes} 
 We now have all the tools to define the dense class of time-changes we will consider to prove Theorem~\ref{thm:main}. As in the assumptions of Theorem~\ref{thm:main},  let $M$ be any nilmanifold of step at least $2$  and let $X$ be 
 a completely irrational nilflow  $X$ on $M$. Recall that (maximal) towers of Heisenberg extensions were defined in Definition~\ref{def:tower}.
 
\begin{definition}[Trigonometric polynomials with respect to a tower]\label{def:trig_pol_tower}
Let $\mathcal{T}=\mathcal{T}_{M,X}$ be a maximal tower of Heisenberg extensions for $M$ based at $X$. We define the space of trigonometric polynomials $\mathscr{P}_{\mathcal{T}}$ with respect to $\mathcal{T}$ inductively in the following way.
\begin{enumerate}
\item Let $\mathscr{P}^{(h)}$ denote the space of trigonometric polynomials over the torus $M^{(h)}$.\vspace{1.3mm}
\item For all $0 \leq i \leq h-1$, let $\mathscr{P}^{(i)}$ be the space of trigonometric polynomials $f \colon M^{(i)} \to \C$ with respect to $[Z^{(i)}]_{\Gamma^{(i)}}$, where the coefficient $f_{0}$ is in $\mathscr{P}^{(i+1)}$.\vspace{1.3mm}
\item Define $\mathscr{P}_{\mathcal{T}} := \mathscr{P}^{(0)}$.
\end{enumerate}
\end{definition}
The dichotomy in Theorem~\ref{thm:main} will be proved for \emph{time-changes} in $\mathscr{P}_{\mathcal{T}} $, i.e.~ positive, real-valued trigonometric polynomials $\alpha \in \mathscr{P}_{\mathcal{T}}$ with respect to $\mathcal{T}$. 

\begin{remark}\label{remark:independent}
The notion of trigonometric polynomials with respect to a tower in Definition \ref{def:trig_pol_tower} depends only on the rational envelopes of the central elements $Z^{(i)}$. In particular, the set of trigonometric polynomials with respect to the tower of Corollary \ref{corollary:existence_of_towers} is the same for every completely irrational nilflows, so that the space $\mathscr{P}$ in the statement of Theorem \ref{thm:main} can be taken independent of the nilflow. 
In general, the choice of a maximal tower is not unique, and the height of any maximal tower is bounded below by $k$ (the case of Corollary \ref{corollary:existence_of_towers}).  It is easy to see that the space of trigonometric polynomials with respect to a given maximal tower is a subset of the set of trigonometric polynomials for any of its refinements, that is
for any maximal tower of greater height which has additional intermediate nilmanifolds.
\end{remark}

Let $\mathscr{C}^k(M)$ denote the space of  $k$-times differentiable functions for $k\in \N \cup \{+\infty\}$.  We show now that $\mathscr{P}_{\mathcal{T}} $ is \emph{dense} in $\mathscr{C}^k(M)$  (see Corollary~\ref{cor:density}). The key inductive step is provided by the following Lemma.

\smallskip
Let $\overline{M} = M/H$, with $H= \exp \z$, as in the beginning of \S\ref{sec:trig_pol}, and consider relative trigonometric polynomials with respect to $\z$.
\begin{lemma}[Relative density] 
Let $\mathcal F \subset \mathscr{C}^k(\overline{M})$ be a dense set of functions. The set of relative trigonometric polynomials $f$  with respect to $z$ such that $f_0 \in \mathcal F$ is dense in $\mathscr{C}^k(M)$.
\end{lemma} 
\begin{proof}
Since the group $\exp \z$ acts through the torus $\exp \z  / \Gamma \cap \exp \z$, which is a compact
Abelian group, by the theory of unitary representations for compact Abelian groups, and  by formula~\eqref{eq:Fourier_exp},  for any function $f \in \mathscr{C}^k(M)$ 
there exists an $L^2$-orthogonal expansion,
$$
f = \sum_{{{\bf v}}\in \Lambda^\ast} f_{{\bf v}} \,, \qquad \text{ with } \quad   f_{{\bf v}} = \int_{T}   \exp(-2\pi\imath \langle {{\bf v}}, {\bf t}\rangle ) f \circ \Phi^{\z} _{\bf t} (x)  \diff {\bf t}\,,
\quad \text{ for all } {\bf v}\in \Lambda\,,
$$
into eigenfunctions of the action of $H$ on $M$.   If $f \in \mathscr{C}^\infty(M)$, then by integration by parts, since $\z \subset Z(\mathfrak g)$, we derive that
the above series converges to zero faster than any polynomial in the Banach space $\mathscr{C}^k(M)$. It follows that the space $\mathscr{C}^\infty(M)$ belongs to the closure in $\mathscr{C}^k(M)$ of the subspace of relative trigonometric polynomials (with respect to $\z \subset Z(\mathfrak g)$ ). 

Since the space $\mathscr{C}^{\infty}(M)$ is a dense 
subspace of $\mathscr{C}^k(M)$, it follows that the subspace of relative trigonometric polynomials is dense in $\mathscr{C}^k(M)$.  Finally, the function $f_0$ is invariant under the action the subgroup $H=\exp \z$, hence it is the pull-back of a function defined on $\overline{M}$. Since the set $\mathcal F \subset \mathscr{C}^k(\overline{M})$ is dense in 
$\mathscr{C}^k(\overline{M})$, it follows immediately that the set of all trigonometric polynomials with respect to $\z$ such that $f_0 \in \mathcal F$  is dense in $\mathscr{C}^k(M)$. 
\end{proof}

\begin{corollary}[Density of $\mathscr{P}_{\mathcal{T}}$ ]\label{cor:density}
Let $\mathcal{T}=\mathcal{T}_{M,X}$ be any maximal tower of Heisenberg extensions for $M$ based at $X$. The space of trigonometric polynomials $\mathscr{P}_{\mathcal{T}}$ is dense in $\mathscr{C}^\infty(M)$.
\end{corollary}

\section{Tools for mixing}\label{sec:tools_for_mixing}
In this section we prove several preliminary results which will provide fundamental tools for proving mixing of non-trivial time-changes in our class. In \S\ref{sec:shearing} we show how mixing can be deduced from  a form of \emph{shearing} (in measure) of curves in a certain direction. In \S\ref{sec:pushing_curves} we compute the pushforward of curves by the flow. Finally, in the remaining sections, we prove a result on growth of ergodic sums for time-changes which are not given by coboundaries in the central fiber (see Theorem~\ref{thm:stretch_ergodic_integral}), which will be used to produce \emph{shearing} in previously  isometric directions. An outline of the proof of this result is given within \S\ref{sec:growth}.

\smallskip

\smallskip
We recall the setting: $M$ is a nilmanifold of step at least two  and  $\phi^X:=\{\phi^X_t\}_{t\in \R}$ is a uniquely ergodic niflow 
 generated by a vector field $X$ on $M$. Given any $\mathscr{C}^\infty$ function $\alpha\colon M \to \R^+$, we consider the time-change $\phi^V:=\phi^{\frac{1}{\alpha} X}= \{\phi^V_t\}_{t\in \R}$ of $\{\phi^X_t\}_{t\in \R}$  
with generator given by the formula
$$
V=\frac{1}{\alpha} X.
$$
 From now on, throughout the rest of the paper, $\phi^V$ to denote the time change of $\phi^X$ with generator $V=\frac{1}{\alpha} X$. Remark that if $\phi^X$ preserves the Haar measure $\mu$, $\phi^V$ preserves the measure $\alpha \mu$ (which is absolutely continuous with respect to Haar when $\alpha$ is smooth).

\subsection{Mixing via shearing} \label{sec:shearing}
 Let us consider an arc  $\{ \phi^W_r (x), \, 0\leq r\leq s\}$ of the flow $\phi^W$ generated by some $W \in \mathfrak{g}$ and let us \emph{push} it via the flow  $ \phi^V$. When these pushed arcs  $\{ \phi^V_t \circ \phi^W_r (x), \, 0\leq r\leq s\}$  \emph{shear} in the direction of $V$, one can hope to show, exploiting equidistribution of trajectories of $ \phi^V$,  that  they \emph{equidistribute} in $M$; in particular one expects that for every $f\in L^2(M, \alpha\diff\mu)$ of zero $\phi^V$-invariant mean, the functions  
$\int_0^s  f\circ \phi^V_t\circ \phi^W_r (x) \diff r$ converge to zero at $t$ grows.

Lemma~\ref{lemma:correlationseasy} below shows  that if $\int_0^s  f\circ \phi^V_t\circ \phi^W_r (x) \diff r$ converge to zero in measure (as well as  in a weaker sense, see Lemma~\ref{lemma:correlations}), this is sufficient to control correlations. Lemma~\ref{lemma:correlationseasy} and Lemma~\ref{lemma:correlations}  provide our main criterion to prove mixing of the time-changed flow  $\phi^V$ and make precise   
the  heuristic mechanism of \emph{mixing via shearing},  
mentioned in the introduction as a key strategy for proving mixing of parabolic flows.

\begin{lemma}[Mixing via shearing]
\label{lemma:correlationseasy}
Let $W\in \mathfrak g$ be any vector field and let $f\in L^2(M,\diff\mu)$ be bounded. Let us assume that there exists a $\sigma >0$ such that for all $ s\in [0, \sigma]$, the functions $\int_0^s  f\circ \phi^V_t\circ \phi^W_r (x) \diff r$ converge to zero in measure, i.e.~for every $\eta>0$ and any $\delta>0$ there exists a   $T >0$ such that for every $t\geq T$
$$
\mu \left( x \in M : \left\lvert \int_0^s  f\circ \phi^V_t\circ \phi^W_r (x) \diff r \right\rvert \geq \eta \right) \leq \delta.
$$
Then, for every $g\in L^2(M,\diff\mu)$ such that $Wg\in L^2(M,\diff\mu)$ we have
$$
\langle f\circ \phi_t^V , g  \rangle_{L^2(M,\diff \mu)}  \to 0\,.
$$
\end{lemma}

Instead than assuming converge to zero in measure the functions $\int_0^s  f\circ \phi^V_t\circ \phi^W_r (x) \diff r$  (i.e.~that the pushed arcs \emph{equidistribute} in measure) for all $ s\in [0, \sigma]$, to prove mixing via shearing it is actually sufficient to verify the assumptions in the following  Lemma~\ref{lemma:correlations}, which has a more delicate order of the quantifiers. This more general formulation will be used in some parts of the proof (see in particular the arguments at the end of \S\ref{sec:noncoboundary_case}).
%In some parts of a proof, The more delicate order of the quantifiers in Lemma~\ref{lemma:correlations} will be crucial in one of the applications of the Lemma (in \S\ref{sec:noncob"oundary_case}).

%     Lemma:correlations     % 
\begin{lemma}[Mixing via shearing 2]
\label{lemma:correlations}
Let $W\in \mathfrak g$ be any vector field and let $f\in L^2(M,\diff\mu)$ be bounded. The conclusion of Lemma~\ref{lemma:correlationseasy} also holds if  we  assume that for every $\delta >0$ there exists $0< \sigma <1$ such that for every $\eta>0$ there exists $T >0$ such that for every $t\geq T$ we have that for any $ s\in [0, \sigma]$,
$$
\mu \left( x \in M : \left\lvert \int_0^s  f\circ \phi^V_t\circ \phi^W_r (x) \diff r \right\rvert \geq \eta \right) \leq \delta.
$$
\end{lemma}

%\begin{remark}\label{rk:conv_measure} 
Notice that  the assumption of Lemma \ref{lemma:correlations} is satisfied in particular under the assumptions of Lemma~\ref{lemma:correlationseasy} (i.e. if the functions $\int_0^s  f\circ \phi^V_t\circ \phi^W_r (x) \diff r$ converge to zero in measure), thus we will only prove this more general  result, which also implies Lemma~\ref{lemma:correlationseasy} . 

%that there exists a $\sigma >0$ such that for all $ s\in [0, \sigma]$, the functions $\int_0^s  f\circ \phi^V_t\circ \phi^W_r (x) \diff r$ converge to zero in measure (i.e.~the pushed arcs \emph{equidistribute} in measure). The more delicate order of the quantifiers in Lemma~\ref{lemma:correlations} will be crucial in one of the applications of the Lemma (in \S\ref{sec:noncoboundary_case}).
%\end{remark}

%\begin{remark}
%Notice that whenever  the assumptions of Lemma~\ref{lemma:correlations} or of Lemma~\ref{lemma:correlationseasy} hold for every $f,g$ in a dense set of functions in $L^2(M, \diff \mu)$, we can conclude that $\phi^V$ is mixing with respect to its invariant measure $\alpha \mu $ (since we can consider an observable of the form $\alpha g$).
%\end{remark}

\begin{proof}[Proof of Lemma~\ref{lemma:correlations} (and hence of Lemma~\ref{lemma:correlationseasy})]
As the Haar volume form is $W$-invariant, for any $\sigma >0$ we have
$$
\langle f\circ \phi^V_t , g  \rangle_{L^2(M,\diff \mu)} = \frac{1}{\sigma}  \int_0^{\sigma}  \langle f\circ \phi^V_t \circ \phi^W_s ,\, g \circ \phi^W_s  \rangle_{L^2(M,\diff \mu)} \,\diff s. 
$$
Integrating by parts we derive the formula (as in \cite{FU}):
\begin{equation}
\label{eq:int_by_parts}
\begin{split}
\frac{1}{\sigma}  \int_0^{\sigma}  \langle f\circ \phi^V_t \circ \phi^W_s ,\,  g \circ \phi^W_s  \rangle \,\diff s  = &
\frac{1}{\sigma}  \langle \int_0^{\sigma}   f\circ \phi^V_t \circ \phi^W_s  \diff s , \, g \circ \phi^W_{\sigma}  \rangle  \\ &- 
\frac{1}{\sigma} \int_0^{\sigma}   \langle  \int_0^s  f\circ \phi^V_t\circ \phi^W_r  \diff r ,\, Wg \circ \phi^W_s  \rangle \diff s \,.
\end{split}
\end{equation}

Fix $\varepsilon >0$ and let $\delta >0$ be such that 
\begin{equation}
\label{eq:delta}
\sqrt{\delta} < \min \left\{ \frac{\varepsilon}{4 \norm{f}_{\infty} \norm{g}_2}, \frac{\varepsilon}{4 \norm{f}_{\infty} \norm{Wg}_2} \right\}.
\end{equation}
Consider the associated $0 < \sigma <1$ given by the assumption and fix $\eta >0$ such that 
\begin{equation}
\label{eq:eta}
\eta < \min \left\{ \frac{\sigma \varepsilon}{4 \norm{g}_2}, \frac{\varepsilon}{4 \norm{Wg}_2} \right\}. 
\end{equation}
Cosider any $t \geq T$, where $T>0$ is given by the assumption.
For any $s \in [0,\sigma]$, let us define
$$
E_t^s : = \left\{ x \in M : \left\lvert \int_0^s  f\circ \phi^V_t\circ \phi^W_r (x) \diff r \right\rvert \geq \eta \right\}.
$$
Using the triangle and the Cauchy-Schwarz inequalities, we can bound the first term on the right-hand side of \eqref{eq:int_by_parts} by
\begin{equation*}
\begin{split}
&\left\lvert \frac{1}{\sigma}  \langle \int_0^{\sigma}   f\circ \phi^V_t \circ \phi^W_s  \diff s , g \circ \phi^W_{\sigma}  \rangle \right\rvert  \leq \frac{1}{\sigma} \left\lvert  \langle \one_{E_t^{\sigma}} \left(\int_0^{\sigma}   f\circ \phi^V_t \circ \phi^W_s  \diff s \right) , g \circ \phi^W_{\sigma}  \rangle \right\rvert \\
& \quad + \frac{1}{\sigma} \left\lvert  \langle \one_{M \setminus E_t^{\sigma}} \left(\int_0^{\sigma}   f\circ \phi^V_t \circ \phi^W_s  \diff s \right) , g \circ \phi^W_{\sigma}  \rangle \right\rvert \leq \norm{f}_{\infty} \norm{g}_2 \sqrt{\mu (E_t^{\sigma}) } + \frac{\eta}{\sigma} \norm{g}_2 \\
& \leq \norm{f}_{\infty} \norm{g}_2 \sqrt{\delta} + \frac{\eta}{\sigma} \norm{g}_2.
\end{split}
\end{equation*}
By choice of the parameters $\delta$ and $\eta $ (see \eqref{eq:delta} and \eqref{eq:eta}), we get
$$
\left\lvert \frac{1}{\sigma}  \langle \int_0^{\sigma}   f\circ \phi^V_t \circ \phi^W_s  \diff s , g \circ \phi^W_{\sigma}  \rangle \right\rvert < \frac{\varepsilon}{2}.
$$
Similarly, for the second term on the right-hand side of \eqref{eq:int_by_parts},
\begin{equation*}
\begin{split}
\left\lvert \langle  \int_0^s  f\circ \phi^V_t\circ \phi^W_r  \diff r , Wg \circ \phi^W_s  \rangle  \right\rvert  \leq & \left\lvert  \langle \one_{E_t^s} \left(\int_0^{s}   f\circ \phi^V_t \circ \phi^W_r  \diff r \right) , Wg \circ \phi^W_{\sigma}  \rangle \right\rvert \\
& + \left\lvert  \langle \one_{M \setminus E_t^s} \left(\int_0^{s}   f\circ \phi^V_t \circ \phi^W_r  \diff r \right) , Wg \circ \phi^W_{\sigma}  \rangle \right\rvert \\
\leq & \norm{Wg}_2 \sigma \norm{f}_{\infty} \sqrt{ \mu (E_t^s)} + \eta \norm{Wg}_2 < \frac{\varepsilon}{2}, 
\end{split}
\end{equation*}
therefore, again by the choice of parameters  \eqref{eq:delta} and \eqref{eq:eta}, we conclude that for any $t \geq T$
$$
\langle f\circ \phi^V_t , g  \rangle_{L^2(M,\diff \mu)}  < \varepsilon.
$$
This proves that $\langle f\circ \phi^V_t , g  \rangle_{L^2(M,\diff \mu)} $ tends to zero. 
\end{proof}

\begin{remark}\label{rk:mixingconclusion} Notice that whenever  the assumptions of Lemma~\ref{lemma:correlations} or of Lemma~\ref{lemma:correlationseasy} hold for every $f,g$ in a dense set of functions of zero $\phi^V$-invariant mean in $L^2(M, \diff \mu)$, we can conclude that $\phi^V$ is mixing (with respect to the invariant measure $\alpha \mu$, since we can consider an observable of the form $\alpha g$).
\end{remark}

\subsection{Pushforward of curves} 
\label{sec:pushing_curves}
{In order to apply the \emph{mixing via shearing} arguments in the previous \S\ref{sec:shearing}, we need to study pushforwards of curves given by some flow $\phi^W$. In this section we hence compute the infinitesimal pushforward of a vector $W$ in the Lie algebra. }

%Let us consider an arc  $\{ \phi^W_r (x), \, 0\leq r\leq s\}$ of the flow $\phi^W_r $ generated by some $W \in \mathfrak{g}$ and let us \emph{push} it via the flow  $ \phi^V_t$. W
\smallskip

Let $(X,Y,Z)\in \mathfrak g^3$ denote any Heisenberg triple. 
Recall we denote by $V= \frac{1}{\alpha} X$ the time-change of $X$. 
%Let ${\alpha}>0$ denote a smooth time-change function (of the flow $\phi^X$ generated by $X$) and $V= \frac{1}{\alpha} X$. 
We have the commutations
$$
\begin{aligned}
&[V,Y] = \left[\frac{1}{\alpha} X,Y\right] = -\left( Y\frac{1}{\alpha} \right) X +  \frac{1}{\alpha} Z =   \frac{Y\alpha}{\alpha} V  + \frac{1}{\alpha} Z\,,  \\  
&[V,Z] = \left[\frac{1}{\alpha} X,Z\right] =  \frac{Z\alpha}{\alpha} V  \,.
\end{aligned}
$$
In particular, $\{V, Y, Z\}$ generates over $\mathscr{C}^\infty(M)$ a Lie subalgebra of the Lie algebra of smooth vector fields on $M$.
We will compute below the action of the tangent dynamics of the flow 
$\phi^V$ on vectors tangent to the foliation tangent to $\{V, Y, Z\}$.

Let $W$ be any vector in the Lie subalgebra generated by $\{V, Y, Z\}$.  We write 
$$
(\phi^V_t)_* (W) = a_t V + b_t Y + c_t Z \,.
$$
By differentiation we derive
$$
\begin{aligned}
 \frac{\diff a_t}{\diff t} V + \frac{\diff b_t}{\diff t} Y + \frac{dc_t}{dt} Z &= -Va_t V  - Vb_t Y -  b_t [V,Y]  - Vc_t Z - c_t [V,Z]\\ 
 &=  - \left(Va_t   +b_t \frac{Y\alpha}{\alpha} + c_t\frac{Z\alpha}{\alpha} \right) V  - Vb_t Y  - \left(b_t \frac{1}{\alpha} + Vc_t \right) Z \,.
\end{aligned}
$$
or, in other terms,
$$
\begin{aligned}
 \frac{\diff a_t}{\diff t} &=  -Va_t   - b_t \frac{Y\alpha}{\alpha} - c_t\frac{Z\alpha}{\alpha}  \,,\\
  \frac{\diff b_t}{\diff t} &=  - Vb_t  \,, \\
\frac{\diff c_t}{\diff t} &=  - Vc_t - b_t \frac{1}{\alpha} \,.
\end{aligned}
$$
It follows that
$$
\begin{aligned}
&\frac{\diff }{\diff t} (a_t \circ \phi^V_t) =  - (b_t \circ \phi^V_t) \frac{Y\alpha}{\alpha}\circ \phi^V_t - (c_t\circ \phi^V_t) \frac{Z\alpha}{\alpha}\circ \phi^V_t  \,,\\
&\frac{\diff }{\diff t}  (b_t \circ  \phi^V_t) =  0 \,, \\
&\frac{\diff }{\diff t} (c_t\circ \phi^V_t) =  -(b_t \circ \phi^V_t) \left(\frac{1}{\alpha}\circ \phi^V_t \right)  \,.
\end{aligned}
$$ 
By solving the system of ODEs above, we get the following expressions for the pushforwards of the vector fields $Y$ and $Z$ by the flow $\phi^V$.

\begin{lemma}[Infinitesimal pushforwards]
\label{corollary:push_forward}
For all $x \in M$ and for all $t \in \R$, the pushforwards of the vector fields $Z$ and $Y$ via $\phi^V_t$ at the point $\phi^V_t(x)$ are
\begin{equation*}
\begin{split}
&[(\phi^V_t)_{\ast}(Z) ](\phi^V_t(x)) = Z(\phi^V_t(x))  - \left( \int_0^t \frac{Z\alpha}{\alpha} \circ \phi^V_{\tau}(x) \diff \tau \right) V(\phi^V_t(x)), \\
&[(\phi^V_t)_{\ast}(Y) ](\phi^V_t(x)) = Y(\phi^V_t(x)) - \left( \int_0^t \frac{1}{\alpha} \circ \phi^V_{\tau}(x) \diff \tau \right) Z(\phi^V_t(x)) \\
& \quad - \left( \int_0^t \frac{Y\alpha}{\alpha} \circ \phi^V_{\tau}(x) \diff \tau - \int_0^t \left(\frac{Z\alpha}{\alpha} \circ \phi^V_{\tau}(x) \right)\left( \int_0^{\tau} \frac{1}{\alpha} \circ \phi^V_{r}(x) \diff r \right)  \diff \tau \right) V(\phi^V_t(x)).
\end{split}
\end{equation*} 
\end{lemma}

Notice that the \emph{shear}, for example of an arc in direction of $Z$ (see the first equation of Lemma~\ref{corollary:push_forward}) is described in terms of ergodic integrals of a function along the flow $\phi^V$ (for example the coefficient in the direction of $V$ for the pushforward of an arc in direction $Z$ is given by the ergodic integral of the function $Z\alpha/\alpha$, see first equation of Lemma~\ref{corollary:push_forward}). For this reason, to understand shear, we will now study the growth of ergodic integrals.

\subsection{Growth of ergodic integrals} \label{sec:growth}

Let $(X,Y,Z)\in \mathfrak g^3$ be a Heisenberg triple and let $[Z]_\Gamma \subset Z(\mathfrak g)$ be the smallest $\Gamma$-rational subspace containing $Z$.  
We recall from \S\ref{sec:trig_pol} that we have a Fourier decomposition
\begin{equation}
\label{eq:Fourier_decomposition}
L^2(M,\diff\mu) = \bigoplus_{{{\bf v}} \in \Lambda^\ast} H_{{\bf v}} ([Z]_\Gamma),
\end{equation}
where
\begin{equation}\label{eq:Hv}
H_{{\bf v}} ( [Z]_\Gamma) = \left\{ f \in L^2(M,\diff\mu) : f \circ \Phi^{ [Z]_\Gamma} _{{\bf t} }= 
\exp(2 \pi \imath  \langle {{\bf v}}, {\bf t}  \rangle)  f  \right\}.
\end{equation}

Let $\beta \in H_0 ([Z]_\Gamma)$ be a smooth, positive, $[Z]_\Gamma$-invariant function, and denote by $\phi^U$ the smooth time-change of the nilflow $\phi^X$ generated by $U = \frac{1}{\beta} X$. 
Then, $\phi^U$ is uniquely ergodic, with invariant measure $\diff \mu_\beta := \beta \diff \mu$.
Moreover,  $\phi^U$ and $\phi^Z$ commute.
In fact we have the following
\begin{lemma} 
\label{lemma:time_change_commute}
If $U = \frac{1}{\beta} X$ is the generator of a time-change with $\beta>0$ a $[Z]_{\Gamma}$-invariant function, then 
$$
\phi^U_s \circ \Phi^{[Z]_{\Gamma}}_{\bf t} =  \Phi^{[Z]_{\Gamma}}_{\bf t} \circ \phi^U_s  \,, \quad \text{ for all } s \in \R\, \text{ and } {\bf t} \in [Z]_{\Gamma}.
$$
\end{lemma}
\begin{proof} 
Let $W \in [Z]_{\Gamma}$.
Since two smooth flows commute if and only if their generators commute, it is
enough to verify that the generators $U= \frac{1}{\beta} X$ and $W$ commute. Since $W \in Z(\mathfrak g)$ and
 the function $\beta$ is $W$-invariant, we have
$$
[U, W] = \left[\frac{1}{\beta} X, W\right] = \frac{1}{\beta} [X,W] - \left( W\frac{1}{\beta}\right) X =  0 \,.
$$
The argument is therefore complete. 
\end{proof}

Let $F_T$ denote ergodic integral  
$$
F_T (x): = \int_0^T  f \circ \phi^{U}_\tau (x) \diff \tau.
$$
We remark that for every $t,T \in \R$, the following standard cocycle relation holds
\begin{equation}
\label{eq:cocycle_relation}
F_t(x) + F_T(\phi^{U}_t(x)) = F_{t+T}(x).
\end{equation}

We will be interested, to prove mixing in \S\ref{sec:proof}, to establish the growth in measure, as $T$ grows, of the integral function $F_T$, where  $f$ (in the integral defining $F_T$) is the function $Z\alpha^\perp$. Indeed, by the calculations presented in \S\ref{sec:pushing_curves}, after the change of variable $\tau = \tau(x,t)$ (see Remark~\ref{rk:equiv_tautilde}), these types of integrals quantify the \emph{shearing} of small curves which we  crucially exploit to prove mixing (see Section~\ref{sec:noncoboundary_case} for details).

This section is devoted to the proof of the following result.

\begin{theorem}
\label{thm:stretch_ergodic_integral}
Let $f  \in H_0([Z]_{\Gamma})^\perp
\in L^2(M,\diff\mu)$ be a trigonometric polynomial w.r.t.~$[Z]_{\Gamma}$. %, with $f_{{\bf v}} \in  H_{{\bf v}} ( [Z]_\Gamma)$. 
Assume that $f$ is not a measurable coboundary for $\phi_{\R}^{U}$. Then, for every $C>1$, we have
$$
\lim_{T \to \infty} \mu \{ |F_T| < C\} = 0.
$$
\end{theorem}

\begin{remark}\label{rk:equivalent}
Since $\mu$ and $\mu_\beta$ (invariant measure for $\phi^U$) are absolutely continuous with respect to each other, the conclusion of Theorem~\ref{thm:stretch_ergodic_integral} is equivalent to showing that   $\lim_{t \to \infty} \mu_\beta \{ |F_t| < C\} =0$.
\end{remark}

%     Strategy     %
The proof of this Theorem will take the remainder of the section. 

\subsection{Outline of the proof of Theorem \ref{thm:stretch_ergodic_integral}}\label{sec:growth_outline}
To prove Theorem~\ref{thm:stretch_ergodic_integral}, which  follows the same scheme than in \cite{AFU}, we proceed in the following way. By a standard  Gottschalk-Hedlund argument (see Lemma \ref{lemma:Gottschalk-Hedlund} below) we  first prove that the \emph{Cesaro averages}  (in $T$) 
 of the measure of the sets $\{ |F_T| < C\} $ tend to zero. From this, we want to derive that the  measure of the set $\{ |F_T| < C\}$ tend to zero at $T$ grows. To do this, we use an argument that we call \emph{decoupling}: we show indeed that $F_T$ and $F_T \circ \phi_t^U$ become sufficiently independent for large values of $t$,  so that one can guarantee that the functions $ F_T \circ \phi_t^U$ and $ F_T $  are unlikely (in measure) to be \emph{simultaneously}  large  when $t$ is large  (see 
 Proposition~\ref{proposition:decoupling} for the precise statement). The proof of Proposition~\ref{proposition:decoupling} is given separately in~\S\ref{sec:decoupling}; the proposition is  a higher dimensional version of  \cite[Lemma 5]{AFU} (see also \cite[Lemma 5.2]{Rav2}), with the important difference that we are working with a time-change of the homogeneous flow instead of its return map; 
the geometric mechanism behind the proof (which is essentially a \emph{wrapping in the fiber} through shearing estimates) still applies in the more general context considered (under the crucial assumption that $\phi^U$ and $\phi^Z$ commute, see Lemma~\ref{lemma:time_change_commute}). At the end of  this section we show that the decoupling result given by Proposition~\ref{proposition:decoupling} can be combined with the result on the average growth of ergodic integrals of Lemma~\ref{lemma:Gottschalk-Hedlund} to deduce the growth in Theorem~\ref{thm:stretch_ergodic_integral}: we use more precisely that one can find an arithmetic progressions $\{ i t_0 \}_{i=1}^{\ell}$  such that $\mu_{\beta} \{ |F_{i t_0}| < C \} < \varepsilon$ (see Corollary~\ref{cor:arithmetic}) and then, thanks to the decoupling argument, apply a version of the inclusion-exclusion principle to conclude.
 
\medskip

%     Gottschalk-Hedlund     %
The following is a standard Gottschalk-Hedlund type of result.
For completeness, we include  a proof for convenience of the reader in  \S\ref{sec:Appendix_B} (Appendix A).
\begin{lemma}[Growth in Cesaro averages]
\label{lemma:Gottschalk-Hedlund}
Assume $f$ is not a measurable coboundary for $\phi^{U}_{\R}$. 
Then, for all $C >1$, 
$$
\lim_{T \to \infty} \frac{1}{T} \int_0^T \mu_{\beta} \{ |F_t| < C\} \diff t = 0.
$$
\end{lemma}

From the convergence of Cesaro averages, we can prove convergence along an arithmetic progressions in the following sense.
%     Arithmetic progressions     %
\begin{corollary}[Growth along arithmetic progressions]
\label{corollary:arithmetic_progression}\label{cor:arithmetic}
Assume $f$ is not a measurable coboundary for $\phi^U$ and let $C >1$.
For all $\varepsilon >0$ and for all $\ell \geq 1$, there exists an arithmetic progressions $\{ i t_0 \}_{i=1}^{\ell}$ of length $\ell$ such that we have $\mu_{\beta} \{ |F_{i t_0}| < C \} < \varepsilon$, for all $i=1, \dots, \ell$. 
\end{corollary}
\begin{proof}
Fix $\varepsilon >0$ and consider
$$
{\rm{Bad}}_{\varepsilon} := \left\{ t \geq 0 : \mu_{\beta} \{ |F_t| < C\} \geq \varepsilon\right\}.
$$
Let $\ell \in \N$ and fix $0 < \delta <1/\ell$. 
By Lemma \ref{lemma:Gottschalk-Hedlund}, there exists $T_0 >0$ such that $\text{Leb} ( [0,T] \cap {\rm{Bad}}_{\varepsilon}) \leq \delta T$ for all $T \geq T_0$, so that, in particular, $\text{Leb} ( [0,jT] \cap {\rm{Bad}}_{\varepsilon}) \leq \delta jT$ for all $j = 1, \dots, \ell$.
We want to find $t_0 \in (0,T)$ such that $t_0, 2t_0, \dots, \ell t_0 \notin {\rm{Bad}}_{\varepsilon}$; in other words, $t_0 \notin (0,T) \cap \frac{1}{j} {\rm{Bad}}_{\varepsilon}$ for all $j = 1, \dots, \ell$.

We estimate the measure 
\begin{equation*}
\begin{split}
& \text{Leb} \left( [0,T] \setminus \Big( \bigcup_{j=1}^{\ell} (0,T) \cap \frac{1}{j} {\rm{Bad}}_{\varepsilon} \Big)\right) \geq T - \sum_{j=1}^{\ell} \text{Leb} \Big( (0,T) \cap \frac{1}{j} {\rm{Bad}}_{\varepsilon} \Big) \\
& \quad =  T \left( 1- \sum_{j=1}^{\ell} \frac{1}{jT} \text{Leb} \big(  (0,jT) \cap {\rm{Bad}}_{\varepsilon} \big) \right) \geq T(1-\delta \ell),
\end{split}
\end{equation*}
which is greater than zero since $\delta< 1/\ell$. 
In particular, the set 
$$
\{ t_0 \in (0,T) : j t_0 \notin {\rm{Bad}}_{\varepsilon} \text{ for all } j = 1, \dots, \ell \}
$$
is not empty and the result follows.
\end{proof}

%     Decoupling     %
The following is our decoupling result, which generalizes \cite[Lemma 5]{AFU} and \cite[Lemma 5.2]{Rav2} to the case of trigonometric polynomials with respect to $[Z]_\Gamma$ (in the sense of Definition \ref{def:trig_pol}).  
Let us consider trigonometric polynomials relative to $\z:=[Z]_{\Gamma}$ in the sense of Definition ~\ref{def:trig_pol}.  Recall the orthogonal decomposition \eqref{eq:orthogonal} of $L^2(M,\diff\mu)=H_0(\z) \oplus H_0(\z)^\perp$.
%The proof is contained in the next section.
\begin{proposition}[Decoupling] \label{proposition:decoupling}
Let $f 
\in H_0(\z)^\perp$ 
be a trigonometric polynomial w.r.t.~$[Z]_{\Gamma}$ of degree~$m$. Assume $f$ is not a measurable coboundary for $\phi^U$ and let $C >1$.
For every $\varepsilon >0$ there exist $C_0 >1$ and $\varepsilon_0 >0$ such that, for every $t >0$ for which $\mu_\beta \{ |F_t| < C_0 \} < \varepsilon_0$, there exists $T_0 >0$ such that for all $T \geq T_0$ we have
$$
\mu_\beta \{ | F_T \circ \phi_t^{U} - F_T | < 2C \} < \varepsilon.
$$
\end{proposition}
The proof of Proposition \ref{proposition:decoupling}  is given in the following subsection \S\ref{sec:decoupling}. We now use Corollary~\ref{cor:arithmetic} and Proposition~\ref{proposition:decoupling} to 
prove Theorem \ref{thm:stretch_ergodic_integral}. 

%     Proof of the Theorem     %
\begin{proof}[Proof of Theorem \ref{thm:stretch_ergodic_integral}] 
By Remark~\ref{rk:equivalent},  the statement of the theorem is equivalent to the statement that $\lim_{t \to \infty} \mu_\beta \{ |F_t| < C\} =0$, for any $C>1$.
It is therefore enough to prove the latter.

%it is enough to prove the equivalent statement that,  for any $C>1$,   $\lim_{t \to \infty} \mu_\beta \{ |F_t| < C\} =0$. 
By Lemma \ref{lemma:Gottschalk-Hedlund}, we have that $\liminf_{t \to \infty} \mu_{\beta} \{ |F_t| < C\} =0$.  Thus, it is enough to assume by contradiction that 
 that $L := \limsup_{T \to \infty} \mu_{\beta} \{ |F_T| < C\} >0$. 
Let us choose $\ell \in \N$ and $\varepsilon >0$ such that
\begin{equation}
\label{eq:assumption_limsup}
\frac{1}{\ell} + \frac{\ell+1}{2} \varepsilon < \frac{L}{2}.
\end{equation}
%Equivalence of $\mu_\beta$ and $\mu$ means that there exists  constants $c_\beta, C_\beta>0$ such that $c_\beta\mu(E) \leq \mu_\beta(E)\leq C_\beta \mu(E)$ for every measurable set $E\subset M$. 
Let us then consider  $C_0 >1$ and $\varepsilon_0>0$ given by Proposition \ref{proposition:decoupling}  for a given %$\varepsilon/C_\beta$. 
$\varepsilon >0$.

Let $\{ i t_0 \}_{i=1}^{\ell}$ be an arithmetic progression of length $\ell$, such that  
%$\mu \{ |F_{i t_0}| < C_0 \} <c_\beta^{-1}$ $\mu_{\beta} \{ |F_{i t_0}| < C_0 \} < \varepsilon_0$. 
 $\mu_{\beta} \{ |F_{i t_0}| < C_0 \} < \varepsilon_0$, for all $i=1, \dots,  \ell$, which exists by Corollary~\ref{corollary:arithmetic_progression}.
%Let $T_0(i) > 0$ be given by Proposition \ref{proposition:decoupling} for $t = i t_0$ and  let  $T_0$ denote the maximum of all $T_0(i)$ for $i = 1, \dots, \ell$. 
Thus, by Proposition \ref{proposition:decoupling},  there exists $T_0>0$ such that for all $1\leq k\leq \ell$ and for all $T\geq T_0$,
%(since $T_0\geq T_0(k)$), 
we have

\begin{equation}\label{eq:decoupledbeta}
\mu_{\beta} \left( |F_T \circ  \phi_{k t_0}^{U} - F_T|  < 2C \right) \leq   \varepsilon\,.
\end{equation}
% \leq C_\beta \mu \left( |F_T \circ  \phi_{k t_0}^{U} - F_T| < 2C \right)\leq C_\beta\frac{\varepsilon}{C_\beta}=\varepsilon.
For all $T \geq T_0$, the Inclusion-Exclusion Principle yields
\begin{equation*}
\begin{split}
\mu_{\beta} \left( \bigcup_{i=1}^{\ell} \phi_{-it_0}^{U} \{ |F_T| <C \} \right) \geq &\sum_{i=1}^{\ell} \mu_{\beta} \left( \phi_{-it_0}^{U} \{ |F_T| <C \} \right) \\
& - \sum_{1 \leq j <i \leq \ell} \mu_{\beta} \left( \phi_{-it_0}^{U} \{ |F_T| <C \} \cap \phi_{-jt_0}^{U} \{ |F_T| <C \}\right).
\end{split}
\end{equation*}
Since $\phi^U_{\R}$ preserves the measure $\mu_\beta$,  for every $t\in\mathbb{R}$, 
For all $1 \leq j <i \leq \ell$,  we have
\begin{equation*}
\begin{split}
&\mu_{\beta} \left( \phi_{-it_0}^{U} \{ |F_T| <C \} \cap \phi_{-jt_0}^{U} \{ |F_T| <C \}\right) = \mu_{\beta} \left(\{ |F_T \circ  \phi_{it_0}^{U}| <C \} \cap \{ |F_T \circ \phi_{jt_0}^{U}| <C \}\right) \\
& \quad \leq \mu_{\beta} \left( |F_T \circ  \phi_{it_0}^{U} - F_T \circ \phi_{jt_0}^{U} | < 2C \right) = \mu_{\beta} \left( |F_T \circ  \phi_{(i-j)t_0}^{U} - F_T| < 2C \right),
\end{split}
\end{equation*}
which, by \eqref{eq:decoupledbeta}, is less than $\varepsilon$, since $0 \leq i-j \leq \ell$ and $T \geq T_0 $.

Choose now $T \geq T_0$ for which $ \mu_{\beta} \{ |F_t| < C\} > L/2$. 
By the computations above and since the flow $\phi^{U}$ is measure-preserving, we thus obtain
$$
1 \geq \mu_{\beta} \left( \bigcup_{i=1}^{\ell} \phi_{-it_0}^U \{ |F_T| <C \} \right) \geq \ell  \mu_{\beta} ( |F_T| <C ) - \frac{\ell (\ell+1)}{2} \varepsilon \geq \ell \frac{L}{2} - \frac{\ell (\ell+1)}{2} \varepsilon. 
$$
This yields the inequality $L/2 \leq 1/\ell + (\ell +1) \varepsilon/2$, in contradiction with the initial assumption~\eqref{eq:assumption_limsup}. 
The proof is therefore complete.
\end{proof}

%%%%%%%%%%

\subsection{Decoupling: proof of Proposition \ref{proposition:decoupling}} \label{sec:decoupling}

This subsection (the final one in this Section 4) is devoted to the proof of Proposition \ref{proposition:decoupling}.
Let us first make some preliminary remarks. 

We first observe that, as the measures $\mu$ and $\mu_\beta$ are equivalent, in the sense that 
that there exists a constant $C_\beta>0$ such that $C_\beta^{-1}\mu  \leq \mu_\beta \leq C_\beta \mu$,  it is enough to prove 
Proposition~\ref{proposition:decoupling} when in the  hypothesis and in the conclusion we replace the measure $\mu_\beta$ by the measure $\mu$.
We will therefore carry our the argument in the latter case.  Let then
$$
f = \sum_{{{\bf v}} \in \Lambda^\ast \setminus \{0\}} f_{{\bf v}}  \in L^2(M,\diff\mu)  \quad  \text{ and }  \quad F_T = \sum_{{{\bf v}} \in \Lambda^\ast \setminus \{0\}} F^{{\bf v}}_T 
 \in L^2(M,\diff\mu)
$$ 
denote the fiberwise Fourier decompositions with respect to the action of the torus tangent to the central Lie subalgebra $[Z]_\Gamma$. Since, by Lemma~\ref{lemma:time_change_commute} the flow $\phi^U$ commutes with the action of this torus, we have
$$
F^{{\bf v}}_T = (f_{\bf v} )_T =  \int_0^T   f_{\bf v} \circ \phi^U_t   dt \,.
$$
Since by assumption $f$ is a trigonometric polynomial with respect to $[Z]_\Gamma$ of degree $m\geq 1$, according to Definition \ref{def:trig_pol}, 
and since $[Z]_\Gamma \subset Z(\mathfrak{g})$ and $\beta \in H_0([Z]_\Gamma)$ (see also Lemma \ref{lemma:time_change_commute} above),
 for all $T, t >0$ the functions $F_T$ and $F_T \circ \phi ^U_t -F_T$  are also trigonometric polynomial with respect to $[Z]_\Gamma$ of degree $m\geq 1$.
 In fact, their  Fourier coefficients  $F^{\bf v}_T$ and $(F_T \circ \phi ^U_t -F_T)_{\bf v} =F^{\bf v}_T \circ \phi ^U_t -F^{\bf v}_T$ are equal to zero  
 for all ${\bf v} \in \Lambda^\ast$ such that $f_{\bf v}$ is equal to zero.

By the cocycle relation for ergodic integrals \eqref{eq:cocycle_relation}, for any $t,T >0$, we have
$$
 F_T \circ \phi^{U}_t - F_T =  F_t \circ \phi^{U}_T - F_t.
$$

Hence, by \eqref{eq:Fourier_exp}, we have 
$$
(F_t \circ \phi^{U}_T - F_t) \circ \Phi^{[Z]_\Gamma}_{\bf t}(x) =  \sum_{{{\bf v}} \in \Lambda^\ast \setminus \{0\}}(F^{{\bf v}}_t \circ \phi^{U}_T - F^{{\bf v}}_t) (x)\cdot \exp(2\pi\imath \langle {{\bf v}}, {\bf t}\rangle ).
$$

In the proof of Proposition \ref{proposition:decoupling}, we will deal with the push-forward of short curves in direction $Y$ under the flow $\phi^U$. 
Let us define 
$$
A_t(x) := - \int_0^{t}   \frac{Y\beta}{\beta} \circ \phi^U_\tau (x) \diff \tau  \quad \text{ and } \quad  B_t(x):=  \int_0^{t}   \frac{1}{\beta} \circ \phi^U_\tau (x) \diff \tau\,.
$$
By Lemma~\ref{corollary:push_forward} (with $\alpha$ replaced by $\beta$), taking into account that $Z\beta =0$, for any $x \in M$ we can write
$$
\frac{\diff}{\diff s} \phi^U_t \circ \phi^Y_s (x) = [A_t \circ \phi^Y_s(x)] U(\phi^U_t \circ \phi^Y_s(x) ) +   Y  (\phi^U_t \circ \phi^Y_s(x) ) - [ B_t \circ \phi^Y_s(x) ] Z(\phi^U_t \circ \phi^Y_s(x)) \,.
$$
By unique ergodicity of the flow $\phi^{U}$, whose invariant measure is $\diff \mu_\beta= \beta \diff \mu$, we have that 
$$
\lim_{t\to +\infty} \frac{B_t(x)}{t} = \int_M  \frac{1}{\beta} \diff \mu_\beta = 1 \text{\ \ \ and\ \ \ } \lim_{t\to +\infty} \frac{A_t(x)}{t} = -\int_M  \frac{Y\beta}{\beta} \diff \mu_\beta = 0,
$$
uniformly in $x \in M$.

\medskip

\begin{proof}[Proof of Proposition \ref{proposition:decoupling}]
Let us first fix the parameters that we will use in the proof.

\smallskip
\noindent {\it Choice of parameters.} 

Notice that, since $[Z]_\Gamma$ is the smallest $\Gamma$-rational subspace containing $Z$, then $\langle {{\bf v}}, Z \rangle \neq 0$ for all ${{\bf v}} \in \Lambda^\ast \setminus \{0\}$.
Let $c_{m,Z}>0$ be the minimum of $|\langle {{\bf v}}, Z \rangle |$ for all ${{\bf v}} \in \Lambda^\ast \setminus \{0\}$.

Let $\Delta_m$, $d_m$ be given by Lemma \ref{thm:sublevel_sets_polynomials}.
Let $C>1$ and $\varepsilon >0$ and let $N\geq 1$ denote the cardinality of ${{\bf v}}\in \Lambda^\ast$ such that  
 $f_{{\bf v}}$ is non zero. Let us then fix $C_0 > N^2$
 and $\varepsilon_0 >0$ such that 
\begin{equation}
\label{eq:C_0_e_0}
\Delta_m \left( \frac{4C}{\sqrt{C_0}} \right)^{d_m} +  \frac{2 N}{\sqrt{C_0} } + \varepsilon_0 < \varepsilon.
\end{equation}
Define $\delta = 4C / \sqrt{C_0}$.

Let $0< \eta < 1/2$ be such that $(1+\eta)/(1-\eta) < \sqrt{2}$.
Choose $\theta \in (0,\pi/2)$ such that 
$$
\frac{1}{\sqrt{C_0}} \leq \sin \theta \leq \frac{1-\eta}{1+\eta} \sqrt{\frac{2}{C_0}}. 
$$
In this way, we also have
\begin{equation}\label{eq:choicetheta}
\frac{\theta}{\pi} < \sin \left(\frac{\theta}{2} \right) \leq \frac{\sin \theta}{\sqrt{2}} \leq \left( \frac{1-\eta}{1+\eta} \right) \frac{1}{\sqrt{C_0}}.
\end{equation}

Fix $t >0$ as in the assumptions of Proposition \ref{proposition:decoupling}, i.e.~such that $\mu \{ |F_t| < C_0 \} < \varepsilon_0$. In particular, since $|F_t| \leq \sum_{{\bf v}} |F_t^{{\bf v}}|$, we have that 
\begin{equation}
\label{eq:decoupling2}
\mu (N_t) < \varepsilon_0, \text{\ \ \ where\ \ \ } N_t :=  \left\{ x \in M :  \sum_{{{\bf v}} \in \Lambda^\ast \setminus \{0\}}  | F_t^{{\bf v}}(x)| < C_0 \right\}. 
\end{equation}
By uniform continuity of $F_t^{{\bf v}}$, let $s_0 >0$ be such that 
%if $|s|\leq s_0$ and $x':=\varphi^Y_s(x)$ 
 if the distance between any two points $x, x\rq{} \in M$ is less than $s_0$,
then $|F_t^{{\bf v}}(x) - F_t^{{\bf v}}(x\rq{})| < 1/4$ for all $ {{\bf v}} \in \Lambda^\ast \setminus \{0\}$ (we use here that $f$ being a trigonometric polynomial, the number of non-zero $f_{{\bf v}}$ and hence of non-zero $F_{{\bf v}}$ is finite).

Fix $T_0 > 8/(c_{m,Z}s_0)$
be such that for all $T \geq T_0$ and all $x \in M$ we have
$$
\left\lvert \frac{B_T(x)}{T} - 1 \right\rvert \leq \eta \text{\ \ \ and\ \ \ } \left\lvert \frac{A_T(x)}{T} \right\rvert \leq \frac{c_{m,Z}}{8 \| \beta^{-1}\|_\infty}s_0.
$$
Finally, let $T \geq T_0$ and define $S= 2/(Tc_{m,Z})$.

\medskip
\noindent {\it Strategy.}  
We want to apply  Lemma~\ref{thm:sublevel_sets_polynomials} to $F_T \circ \phi^{ U }_t - F_T =  F_t \circ \phi^{U }_T - F_t$. Consider the set 
$$
E_{C_0} := \left\{ x \in M :  \sum_{{{\bf v}} \in \Lambda^\ast \setminus \{0\}} | (F_t^{{\bf v}} \circ \phi^{U}_T - F_t^{{\bf v}})(x) | < \frac{\sqrt{C_0}}{2} \right\}.
$$
If we can prove that 
\begin{equation}\label{eq:EC0estimate} \mu (E_{C_0})<\varepsilon_0+\frac{2N}{\sqrt{C_0}},
\end{equation}
 then  Lemma~\ref{thm:sublevel_sets_polynomials} (for $\delta:= 4C/ \sqrt{C_0}$) and the choice of parameters (see \eqref{eq:C_0_e_0}) give that 
\begin{equation*}
\mu \{ | F_T \circ \phi^{U}_t - F_T | < 2C \} \leq \Delta_m \left( \frac{4C}{\sqrt{C_0}} \right)^{d_m} + \mu (E_{C_0})  < \Delta_m \left( \frac{4C}{\sqrt{C_0}} \right)^{d_m}+ \varepsilon_0+ \frac{2N}{\sqrt{C_0}} < \varepsilon,
\end{equation*}
which proves the Lemma. We hence just have to prove the estimate \eqref{eq:EC0estimate} for  $\mu(E_{C_0})$.

\smallskip \noindent {\it Reduction of space to time estimates.} 
%We consider the first term on the right-hand side of \eqref{eq:decoupling1}.
%Since we fixed $\delta = 4C / \sqrt{C_0}$, define the set
Let us remark that, since  the Haar measure $\mu$ is invariant under the flow $\phi^Y_{\R}$,  we have
\begin{equation}
\label{eq:decoupling3}
\begin{split}
& \mu (E_{C_0}) = \frac{1}{ S } \int_0^{S} \left( \int_M \one_{ E_{C_0} } \circ \phi^Y_{s}  \diff \mu \right)\diff s = \int_M \left( \frac{1}{S } \int_0^{S} \one_{ E_{C_0} } \circ \phi^Y_{s}  \diff s \right) \diff \mu \\
& \quad = \int_M \frac{1}{S} \text{Leb} \left( s \in [0, S ] :  \sum_{{{\bf v}} \in \Lambda^\ast \setminus \{0\}} | (F_t^{{\bf v}} \circ \phi^{U}_T - F_t^{{\bf v}})  \circ \phi^Y_s (x)| < \frac{\sqrt{C_0}}{2}  \right) \diff \mu.
\end{split}
\end{equation}

For any fixed ${{\bf v}}$ and $x \in M$, consider the term $(F_t^{{\bf v}} \circ \phi^{U}_T - F_t^{{\bf v}})  \circ \phi^Y_s (x)$.

Recalling the definition of $H_{{\bf v}} ([Z]_\Gamma)$ in \S\ref{sec:trig_pol}, for any $r \in \R$ we can write $F_t^{{\bf v}} = e_{{\bf v}}(-r) F_t^{{\bf v}} \circ \phi^Z_r$, where we denoted $e_{{\bf v}}(r) = \exp(2 \pi \imath r \langle {{\bf v}}, Z  \rangle) $. 
We now choose $r=R(s)$ as follows. 

Let $R(s):= R(x,T,s)$ be the function defined by the condition that the curve 
$$
\gamma_{x,T}(s):= \phi^Z_{R(s)} \circ \phi^U_T \circ \phi^Y_s\,, \quad \text{ for } \, s\in [0, S],
$$ 
be ``horizontal'', that is, tangent to the distribution spanned by $\{U, Y\}$. This condition is tantamount to requiring that $R(s)$ be a solution of the ODE
$$
\frac{\diff}{\diff s} R(s) = B_T \circ \phi_{R(s)}^Z \circ \phi^Y_s(x),
$$
with the initial condition $R(0)=0$ (indeed, the solution to the ODE above is unique and defined for all $s \in \R$).
Remark that, by the bounds on $B_T$, we have
\begin{equation}\label{eq:bound_on_Rs}
sT(1-\eta) \leq R(s) \leq sT(1+\eta), \text{\ \ \ for all\ } s \in [0,S].
\end{equation}
Then, 
$$
 (F_t^{{\bf v}} \circ \phi^U_T - F_t^{{\bf v}})  \circ \phi^Y_s = e_{{\bf v}}(-R(s)) F_t^{{\bf v}} \circ \phi^Z_{R(s)} \circ \phi^U_T \circ \phi^Y_s - F_t^{{\bf v}}  \circ \phi^Y_s.
$$
%For the sake of notation, let us call $\gamma_{x,T}(s) = \phi^Z_{R(s)} \circ \phi^U_T \circ \phi^Y_s(x)$.

We hence have to estimate  for  fixed $x$ (which will be chosen in certain set)
$$
\text{Leb} \left( s \in [0, S ] : \sum_{{{\bf v}} \in \Lambda^\ast \setminus \{0\}}   \left\lvert 
e_{{\bf v}}(-R(s)) F_t^{{\bf v}} \circ \gamma_{x,T}(s)- F_t^{{\bf v}}  \circ \phi^Y_s (x)
\right\rvert < \frac{\sqrt{C_0}}{2} \right).
$$
In order to use decoupling arguments  to estimate the measure of the above parameters $s\in  [0, S]$, we first want to \emph{localize} the function, i.e.~show that 
the term in the modulus above does not vary much in $s$, so we can eliminate the dependence on  $s$.

\smallskip
\noindent{\it Localization.} 
Let us show that we can assume that $F_t^{{\bf v}} \circ \gamma_{x,T}(s)$
 and $  F_t^{{\bf v}} \circ \phi^Y_s$ are essentially constant and equal to $F_t^{{\bf v}}  \circ \phi^{U }_T $ and $  F_t^{{\bf v}}$ (corresponding to $s=0$).  
 
We shall estimate from below the difference $\vert F^{\bf v}_T \circ \phi^U_T \circ \phi^Y_s - F^{\bf v}_T \circ \phi^Y_s\vert$ by taking
advantage of the decomposition of the path $\phi^U_T \circ \phi^Y_s$, for $s\in [0,T]$,  into the ``horizontal'' path $\gamma_{x,T}(s)$ and the 
``vertical" path $\phi^{Z}_{-R(s)}$, for $s\in [0,S]$, along which the function $F^{\bf v}_T$  is controlled.
 
By triangle inequality, adding and subtracting $\big(F_t^{{\bf v}} + e_{{\bf v}}(-R(s))   F_t^{{\bf v}} \circ \phi^U_T \big) (x)$, we can write

\begin{equation*}
\begin{split}
 \sum_{{{\bf v}} \in \Lambda^\ast \setminus \{0\}}  & | (F_t^{{\bf v}} \circ \phi^U_T - F_t^{{\bf v}})  \circ \phi^Y_s (x)|   = \sum_{{{\bf v}} \in \Lambda^\ast \setminus \{0\}}| e_{{\bf v}}(-R(s))  F_t^{{\bf v}} \circ \gamma_{x,T}(s) - F_t^{{\bf v}} \circ \phi^Y_s (x) |\\ & \qquad 
 \quad \geq \sum_{{{\bf v}} \in \Lambda^\ast \setminus \{0\}} \Big(  | e_{{\bf v}}(-R(s))  F_t^{{\bf v}}  \circ \phi^U_T (x) - F_t^{{\bf v}}  (x) |  - |F_t^{{\bf v}}  \circ \phi^Y_s (x) -  F_t^{{\bf v}}  (x)| \\ & \qquad \qquad \qquad \qquad \qquad \qquad \qquad \qquad
- |e_{{\bf v}}(-R(s)) | \cdot | F_t^{{\bf v}} \circ \gamma_{x,T}(s) - F_t^{{\bf v}}  \circ \phi^U_T (x)|  \Big). 
\end{split}
\end{equation*}
Since $0\leq s\leq S \leq s_0/2$, by definition of $s_0$, the term $ |F_t^{{\bf v}}  \circ \phi^Y_s (x) -  F_t^{{\bf v}}  (x)| $ in brackets above is less than $1/4$. 
Once we show that the distance between $\gamma_{x,T}(s)$ and $\phi^U_T(x)$ is less than $s_0$ as well, then (recalling $N$ is the number of non-zero coefficients $F_t^{{\bf v}}$) we will get
\begin{equation}\label{eq:term_to_estimate}
\begin{split}
\sum_{{{\bf v}} \in \Lambda^\ast \setminus \{0\}} | (F_t^{{\bf v}} \circ \phi^U_T - F_t^{{\bf v}})  \circ \phi^Y_s (x)| & \geq  \sum_{{{\bf v}} \in \Lambda^\ast \setminus \{0\}} \left(  | e_{{\bf v}}(-R(s))  F_t^{{\bf v}}  \circ \phi^U_T (x) - F_t^{{\bf v}}  (x) |  - \frac{1}{4} - \frac{1}{4}\right) \\
& \geq \left( \sum_{{{\bf v}} \in \Lambda^\ast \setminus \{0\}} | e_{{\bf v}}(-R(s))  F_t^{{\bf v}}  \circ \phi^U_T (x) - F_t^{{\bf v}}  (x) |  \right) - \frac{N}{2}.
\end{split}
\end{equation}

Let us prove our claim: the distance between $\gamma_{x,T}(s)$ and $\phi^U_T(x) = \gamma_{x,T}(0)$ can be bounded by $S$ times the maximum of the modulus of the tangent vectors of $\gamma_{x,T}(s)$. 
Using the chain rule and recalling that $Z$ commutes with $Y$ and $U$, we compute
\begin{equation*}
\begin{split}
&\frac{\diff}{\diff s} \gamma_{x,T}(s) = \frac{\diff}{\diff s} \phi^Z_{R(s)} \circ \phi^U_T \circ \phi^Y_s(x) = \left(\frac{\diff}{\diff s} R(s) \right) Z(\gamma_{x,T}(s)) + [D\phi^Z_{R(s)}] \left(\frac{\diff}{\diff s} \phi^U_T \circ \phi^Y_s(x) \right) \\
&= \left(\frac{\diff}{\diff s} R(s) - [ B_T \circ \phi^Z_{R(s)} \circ \phi^Y_s(x) ] \right) Z(\gamma_{x,T}(s)) + [A_T \circ \phi^Z_{R(s)} \circ \phi^Y_s(x)] U(\gamma_{x,T}(s) ) +   Y  (\gamma_{x,T}(s) ).
\end{split}
\end{equation*}
The first term in the right hand-side above is zero, by our choice of $R(s)$. By definition of $T_0$, we can bound
$$
\left\lvert \frac{\diff}{\diff s} \gamma_{x,T}(s) \right\rvert \leq 1 + \|A_T \|_\infty \cdot \| \beta^{-1} \|_\infty \leq 1+  \frac{c_{m,Z}}{8}Ts_0.
$$
Therefore, the distance between $\gamma_{x,T}(s)$ and $\gamma_{x,T}(0) = \phi^U_T(x)$ can be bounded by
$$
S \cdot \left\lvert \frac{\diff}{\diff s} \gamma_{x,T}(s) \right\rvert \leq S \left( 1+ \frac{c_{m,Z}}{8}Ts_0 \right) \leq  S + \frac{s_0}{2} \leq s_0,
$$
which proves our claim, and hence gives us \eqref{eq:term_to_estimate}.

\smallskip
\noindent{\it Decoupling. }
We now hence want to estimate the sum in the right hand-side of \eqref{eq:term_to_estimate}. We shall do so by showing that the quantity $R(s)$, for $s \in [0,S]$, grows
sufficiently and nearly linearly; consequently the complex number $e_{\bf v} (-R(s))$ 
is far away from any given point on the unit circle in $\C$,  for a subset of parameters
$s \in [0,S]$ of large relative measure.

Let $c_1:= F_t^{{\bf v}}  (x) \in \C$ and $c_2 :=F_t^{{\bf v}}  \circ \phi^U_T (x) \in \C$. 
Since, as we already remarked, $ \langle {{\bf v}}, Z \rangle \neq 0$ for all ${{\bf v}}\in \Lambda^\ast \setminus \{0\}$, by elementary trigonometry,
any point $c\rq{} \in \C$ outside the cone of $1/2$-angle $\theta$ about the line $\R c_1$ has distance $|c\rq{}-c_1|$ from $c_1$ larger than the distance of $c_1$ from the boundary of the cone  (see Figure \ref{fig:cone}).
Thus, for the  parameter $\theta$ chosen as in the beginning, 
$| e_{{\bf v}}(-R(s))c_2 -c_1|> |c_1|\sin \theta $ as long as the phase $-\langle {{\bf v}}, Z \rangle R(s)\mod 2\pi$ does not fall in an interval of size $2\theta$. If $R(s)$ was linear in $s$, namely if $R(s) = sT$, the measure of the set of $s \in [0,S]$  for which $|e_{{\bf v}}(-sT) c_2 - c_1 |\leq |c_1| \sin \theta$ would be bounded by $S \theta/\pi$ plus a \lq\lq{}boundary error\rq\rq{}\ term of size $2( | \langle {{\bf v}}, Z \rangle | T)^{-1}$.
In our case, although $R(s)$ is not necessarily linear in $s$, it satisfies the bounds in \eqref{eq:bound_on_Rs}; thus, the measure of the set of $s \in [0,S]$ for which $-\langle {{\bf v}}, Z \rangle R(s) \mod 2\pi$ does not fall in an interval of size $2\theta$ can be bounded by the measure of the set of $s \in [0,S]$ for which $-\langle {{\bf v}}, Z \rangle sT \mod 2\pi/(1+\eta)$ does not fall in an interval of size $2\theta/(1-\eta)$.
In particular,
the measure of the set of $s \in [0, S]$ for which $|e_{{\bf v}}(-R(s)) c_2 - c_1 |\leq |c_1| \sin \theta$ can bounded by $S \frac{\theta(1+\eta)}{\pi (1-\eta)}$ plus a \lq\lq{}boundary error\rq\rq{}, namely
\begin{multline*}
\text{Leb} \left( s \in [0,S] : \left\lvert e_{{\bf v}}(-R(s))  F_t^{{\bf v}}  \circ \phi^U_T (x) - F_t^{{\bf v}}  (x) \right\rvert \leq | F_t^{{\bf v}}  (x) | \sin \theta \right) \\ 
\leq S \frac{\theta (1+\eta)}{\pi(1-\eta)} + 2 \frac{\theta(1+\eta)}{\pi(1-\eta)} \cdot \frac{1}{ | \langle {{\bf v}}, Z \rangle | T(1-\eta)} \leq \frac{\theta(1+\eta)}{\pi(1-\eta)} \left(S + \frac{4}{ | \langle {{\bf v}}, Z \rangle | T} \right).
\end{multline*}

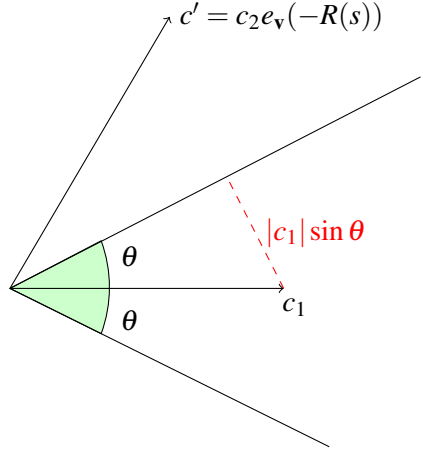
\begin{figure}[ht!]
\centering
\begin{tikzpicture}[scale=3]
\clip (-0.2,-0.8) rectangle (2.5,1.3);
\draw [->] (0,0) -- (0.7,1.2) node [anchor=west] {$c\rq{} = c_2  e_{{\bf v}}(-R(s)) $};
\filldraw[fill=green!20!white] (0,0) -- (0.4,-0.2) arc (-20:20:6mm)  -- cycle;
\draw [->] (0,0) -- (1.2,0);
\draw (1.25,0) node [anchor=north] {$c_1$};
\draw (0,0) -- (1.4,-0.7);
\draw (0,0) -- (1.8,0.935);
\draw (0.44,0.14) node [anchor=west] {$\theta$};
\draw (0.44,-0.14) node [anchor=west] {$\theta$};
\draw[dashed, red] (1.2,0) -- (0.95,0.495) node [pos=0.5,right] {\textcolor{red}{$|c_1| \sin \theta$}};
\end{tikzpicture}
\caption{ Any point $c\rq{} \in \C$ outside the cone of 1/2-angle $\theta$ about the line $\R c_1$ has distance $|c\rq{}-c_1|$ from $c_1$ larger than the distance of $c_1$ from the boundary of the cone.} %
\label{fig:cone}
\end{figure}

Recall we defined $c_{m,Z}>0$ to be the minimum of $ | \langle {{\bf v}}, Z \rangle | $ for all ${{\bf v}} \in \Lambda^\ast \setminus \{0\}$. Thus, outside a set of  $ s \in [0, S ]$ of measure at most $N \frac{\theta(1+\eta)}{\pi(1-\eta)}(S + \frac{4}{c_{m,Z} T} )$, we have that 
\begin{equation*}
\begin{split}
\left( \sum_{{{\bf v}} \in \Lambda^\ast \setminus \{0\}} | e_{{\bf v}}(- R(s) )  F_t^{{\bf v}}  \circ \phi^U_T (x) - F_t^{{\bf v}}  (x) |  \right) - \frac{N}{2} & > \left( \sum_{{{\bf v}} \in \Lambda^\ast \setminus \{0\}} | F_t^{{\bf v}} (x) | \sin \theta \right) - \frac{N}{2} \\
&  \geq \frac{\sum_{{{\bf v}}\in \Lambda^\ast \setminus \{0\} } | F_t^{{\bf v}} (x) |}{\sqrt{C_0}}- \frac{\sqrt{C_0}}{2} .
\end{split}
\end{equation*}

When $x \notin N_t$, by the definition in \eqref{eq:decoupling2}, the last term above is greater or equal to $\sqrt{C_0}/2$.
Therefore, for all $x \notin N_t$,
\begin{equation*}
\text{Leb} \left( s \in [0, S ] :  \sum_{{{\bf v}} \in \Lambda^\ast \setminus \{0\}} | (F_t^{{\bf v}} \circ \phi^U_T - F_t^{{\bf v}})  \circ \phi^Y_s (x)| < \frac{\sqrt{C_0}}{2} \right) \leq  N \frac{\theta(1+\eta)}{\pi(1-\eta)} \left( S + \frac{4}{c_{m,Z} T} \right).
\end{equation*}
After dividing by $S$ and integrating over $M \setminus N_t$, by our choice of $\eta$ and $S$, we get
\begin{equation}
\begin{split}
\label{eq:decoupling4}
& \int_{M \setminus N_t} \frac{1}{S} \text{Leb} \left( s \in [0,S] : \sum_{{{\bf v}} \in \Lambda^\ast \setminus \{0\}} | (F_t^{{\bf v}} \circ \phi^U_T - F_t^{{\bf v}})  \circ \phi^Y_s (x)| < \frac{\sqrt{C_0}}{2} \right) \diff \mu \\ 
&\qquad  \quad \leq N \frac{\theta(1+\eta)}{\pi(1-\eta)}  \left( 1 + \frac{4}{c_{m,Z} S T} \right) < \frac{2 N}{\sqrt{C_0} }.
\end{split}
\end{equation}
From \eqref{eq:decoupling3} and \eqref{eq:decoupling4}, it follows that 
$
\mu (E_{C_0}) < \mu (N_t) + \frac{2 N}{\sqrt{C_0} } \leq \varepsilon_0 +  \frac{2 N}{\sqrt{C_0} },
$
which concludes the proof of \eqref{eq:EC0estimate} and hence of the Proposition.
 \end{proof}

%%%%%%%%%%

\section{Inductive proof of mixing}\label{sec:proof}
We can now prove the main result, Theorem~\ref{thm:main}. We will decompose the time-change using the towers of nilflows extensions built in \S\ref{sec:induction} (see \S\ref{sec:decomposition}), then explain the two mechanisms which, at each step of the tower, either allow to directly produce mixing from stretching of Birkhoff sums (this is the \emph{non-coboundary} case in \S\ref{sec:noncoboundary_case}), or, in the \emph{coboundary case} in \S\ref{sec:coboundary_case}, allow to \emph{lift} mixing from the factor to the extension. The proof by induction, which combines these two steps, is then given at the end, in \S\ref{sec:final}.

\smallskip

Let us recall (see \S\ref{sec:basic}) that the time-change induced by $V = \frac{1}{\alpha}X$ is measurably trivial if and only if $\frac{1}{\alpha}$ is a measurable almost coboundary for $\phi^V$, or, equivalently, if and only if $\alpha$ is a measurable almost coboundary for $\phi^X$.

 \subsection{Decomposition of the time-change along the tower}\label{sec:decomposition}

 Given a uniquely ergodic nilflow $\phi^X$ on the nilmanifold $M$,   let  $\mathcal{T}_{M,X}$  be a   be a tower of Heisenberg extensions for $M$ based at $X$, whose existence is guaranteed by Corollary \ref{corollary:existence_of_towers}. Consider time-changes which are positive, real-valued trigonometric polynomials belonging to the class   $\mathscr{P}_{\mathcal{T}} $ consisting of  trigonometric polynomials with respect to this tower (see Definition~\ref{def:trig_pol_tower}). Density of this class follows from Corollary~\ref{cor:density}. 
  
Let $\alpha \in \mathscr{P}_{\mathcal{T}} $ be a trigonometric polynomials with respect to this tower (see Definition~\ref{def:trig_pol_tower}), and set $\alpha^{(0)}:=\alpha$. 
Recall (from \S\ref{sec:induction}) that we denote by $\z^{(i)}:= [Z]_{\Gamma^{(i)} }$ and $\Lambda^{(i)} := \log \Gamma^{(i)} \cap [Z]_{\Gamma^{(i)} }$ for all $0 \leq i \leq h$.

Assume we defined  $\alpha^{(i)} \in \mathscr{P}^{(i)}\subset \mathscr{C}^{\infty}(M^{(i)})$ (where $\mathscr{P}^{(i)}$ is the space of trigonometric polynomials relative to $\z^{(i)}$ defined in Definition~\ref{def:trig_pol_tower}) for $0\leq i<h$. 
We consider the decomposition associated to the orthogonal decomposition (defined in \S\ref{sec:trig_pol})
$$
L^2(M^{(i)}) = H_0(\z^{(i)})\oplus H_0(\z^{(i)})^{\perp}, \qquad \text{where}\quad  H_0(\z^{(i)}) := (\pi^{(i+1)})^{\ast} \big(L^2(M^{(i+1)}) \big) ,
$$
namely we write
$$
\alpha^{(i)} = \alpha^{(i)}_0 + (\alpha^{(i)}) ^{\perp},
$$
where
$$
\alpha^{(i)}_0= \int_{\z^{(i)} / \Lambda^{(i)} } \alpha \circ \Phi_{\bf t}^{\z^{(i)}}   \diff {\bf t} \in H_0(\z^{(i)})$$
is a strictly positive, $\z^{(i)}$-invariant function, and 
$$
(\alpha^{(i)})^{\perp} = \alpha^{(i)}  - \alpha^{(i)}_0  \in H_0(\z^{(i)})^{\perp}.
$$
Since $\alpha^{(i)}_0$ is constant under the action of the torus $\z^{(i)} / \Lambda^{(i)} $, that is, along the fibers of the fibration $ \pi^{(i+1)}: M^{(i)} \to M^{(i+1)}$,  there
exists a function $\overline{\alpha}^{(i)}_0$ such that $\alpha^{(i)}_0 = \overline{\alpha}^{(i)}_0 \circ \pi^{(i+1)}$ on $M^{(i)}$. 

We now set $\alpha^{(i+1)}:= \overline{\alpha}^{(i)}_0$. 
By definition of $\mathscr{P}_{\mathcal{T}}$  and  $\mathscr{P}^{(i)}$ (see Definition~\ref{def:trig_pol_tower}),  $\alpha^{(i+1)} \in \mathscr{P}^{(i+1)}$, thus we can continue the inductive definition until $i+1=h$. 

For each $0 \leq i \leq h$, let  $V^{(i)}$ denote the time-change of the vector field $X^{(i)}$ on $M^{(i)}$ given by
$$
V^{(i)}=\frac{1}{\alpha^{(i)} } X^{(i)}.
$$

\smallskip
In the final proof, for some $0\leq i\leq h$,  we will need to  consider two cases:  
\begin{itemize}
 \item[(i)]  {\it the function $(\alpha^{(i)})^{\perp}$ is a measurable coboundary w.r.t.~$X^{(i)}$;}

\noindent In this case (treated in \S\ref{sec:coboundary_case}), we will show that if the flow generated by $V^{(i+1)}$ on $M^{(i+1)}$ is mixing, then also the flow generated by $V^{(i)}$ on $M^{(i)}$ is mixing (see Proposition~\ref{coboundary_case} in \S\ref{sec:coboundary_case}) 
 
 \item[(ii)] {\it the function $(\alpha^{(i)})^{\perp}$ is \emph{not} a measurable coboundary w.r.t.~$X^{(i)}$}; 
 
 \noindent In this  case (treated in \S\ref{sec:noncoboundary_case}) we will show directly (using the growth of Birkhoff sums proved in \S\ref{sec:growth}) that the flow generated by $V^{(i)}$ on $M^{(i)}$ is mixing (see Proposition~\ref{noncoboundary_case} in \S\ref{sec:noncoboundary_case}).  
\end{itemize}

The argument will proceed by (finite) reverse induction, with the initial step for the smallest index $i_0<h$ such that $(\alpha^{(i_0)})^{\perp}$ is not a measurable coboundary 
w.r.t.~$X^{(i_0)}$, given by case $\rm{(ii)} $, and the remaining induction steps given by case $\rm (i)$ above.

\smallskip
\noindent 
\textbf{Notation.}
In order to keep the notation simpler, in the following two sections we will drop the index $i$, hence we will simply write $M = \Gamma \backslash G$ for a  nilmanifold, $X$ for the generator of a uniquely ergodic nilflow, $\{ X, Y, Z\}$ for a Heisenberg triple, and  $\alpha\in\mathscr{C}^{\infty}(M)$ for a positive, real-valued trigonometric polynomial with respect to $\z:=[Z]_{\Gamma}$, that is the smallest $\Gamma$-rational subspace containing~$Z$.
 The Propositions proved in the next two sections will then be applied to $M^{(i)}$, $X^{(i)}$ and the time-change $V^{(i)} = \frac{1}{\alpha^{(i)}} X^{(i)}$  given by $\alpha^{(i)}$ defined above.

\subsection{Case (i) (coboundary case)}\label{sec:coboundary_case}

In this section we consider the case in which the function $
\alpha^{\perp}= \alpha - \alpha_0 \in H_0([Z]_{\Gamma})^{\perp}
$
is  a measurable coboundary w.r.t.~$X$. In this case (thanks to Lemma~\ref{lemma:time_change_commute}), we will show that the flow $\phi^V$ projects on a flow on the quotient manifold $\overline{M}= M /\exp [Z]_\Gamma = M/H $ and that if the projected flow is mixing, also the original flow was mixing (see Proposition~\ref{coboundary_case}). This is the case in which to prove mixing we will exploit the intrinsic dynamics of nilflows and especially the shearing mechanism of \emph{wrapping in the fibers}.

\smallskip
We first want to show that it is possible to define a flow on the quotient nilmanifold $\overline{M}$. 
By the standard theory of time-changes, see for instance Lemma~\ref{lemma:triviality},  in case $\rm (i)$ the time-change is measurably conjugate to a time-change with 
$[Z]_{\Gamma}$-invariant time-change function. We can therefore assume that the function $\alpha$ is $[Z]_{\Gamma}$-invariant (hence $\alpha^\perp=0$ and $\alpha= \alpha_0$).

\smallskip 
Let then $\overline{\mathfrak g} = \mathfrak g/ [Z]_\Gamma$, $\overline{G} = G / \exp [Z]_\Gamma$ and let 
$\overline{M}= M /\exp [Z]_\Gamma$. 
 
Since the action $(\Phi^{[Z]_\Gamma}_{\bf t})_{{\bf t} \in \R^d}$ induces a non-singular toral action, the quotient 
$\overline{M}= \overline{\Gamma} \backslash \overline{G} $ is a nilmanifold, and $M$ is a toral bundle $\pi \colon M \to \overline{M}$ over $\overline{M}$.  By Lemma~\ref{lemma:time_change_commute} the time-change $\phi^V$ projects to a time-change 
$\phi^{\overline{V}}$ of a nilflow $\phi^{\overline{X}}$ on $\overline{M}$. 
We remark that the invariant measure for $\phi^{\overline{V}}$ is $\pi_{\ast} (\alpha \diff \mu) = {\overline{ \alpha}} \diff {\overline{\mu}}$, where $ \diff {\overline{\mu}}$ is the Haar measure on 
$\overline{M}$. We will prove the following relation between  $\phi^{V}$ on $ M$ and  $\phi^{\overline{V}}$ on $\overline{M}$.

 \begin{proposition}\label{coboundary_case}
If the projected flow $\{\phi^{\overline{V}}_t \}_{t \in \R}$ on $\overline{M}$ generated by ${\overline{V}} = (\overline{\alpha})^{-1}  {\overline{X}}$ is mixing (w.r.t.~${\overline{\alpha}} \diff {\overline{\mu}}$), then the flow $\{ \phi^{V}_t \}_{t \in \R}$ on $ M$ is mixing (w.r.t.~$\alpha \diff \mu$).  
\end{proposition}
\smallskip

\begin{proof}
We consider correlations for the flow $\phi^V$.  Let $\omega_V$ the $V$-invariant volume form. It follows from the definition that $\omega_V =   \alpha  \diff \mu$.  We have
$$
\langle f\circ \phi^V_t , g \rangle_{L^2(M,\omega_V)}  = \langle f\circ \phi^V_t , {\alpha}{g} \rangle_{L^2(M,\diff \mu)}\,,
$$
hence it is equivalent to control correlations with respect to the Haar measure $\mu$. 

As above, the Hilbert space $L^2(M,\diff\mu)$ decomposes as an orthogonal sum 
$$
L^2(M,\diff\mu) = H_0([Z]_\Gamma) \oplus^\perp H_0([Z]_\Gamma)^\perp \,,  \quad  \text{ with } \quad  H_0([Z]_\Gamma) 
=  \pi^* (L^2(\overline{M},\diff{\overline{\mu}})) \,. 
$$
Let $f =\pi^*(\overline{f}) \in  \pi^* (L^2(\overline{M},\diff {\overline{\mu}})) $,  $g \in L^2(M,\diff\mu)$  and let 
$\pi^*(\overline{g})$ with $\overline{g}\in L^2(\overline{M},\diff {\overline{\mu}})$ denote the orthogonal component of 
$g$ in $\pi^* (L^2(\overline{M},\diff {\overline{\mu}}))$. We have
$$
\begin{aligned}
\langle f\circ \phi^V_t, g \rangle_{L^2(M, \diff \mu)} &=  \langle \pi^* (\overline{f} \circ \phi^{\overline{V}}_t) , g \rangle_{L^2(M, \diff \mu)} \\ &=  \langle \pi^* (\overline{f} \circ \phi^{\overline{V}}_t) , \pi^*(\overline{g}) \rangle_{L^2(M, \diff \mu)}  
=  \langle \overline{f} \circ \phi^{\overline{V}}_t , \overline{g} \rangle_{L^2(\overline{M}, \diff {\overline{\mu}})} \,.
\end{aligned}
$$
By assumption, the latter term above has the correct asymptotics, that is
$$
\langle \overline{f} \circ \phi^{\overline{V}}_t , \overline{g} \rangle_{L^2(\overline{M}, \diff {\overline{\mu}})} \to \left( \int_{\overline{M}} {\overline{f}}  \overline{\alpha} \diff {\overline{\mu}} \right) \cdot \left( \int_{\overline{M}} {\overline{g}} \diff {\overline{\mu}} \right) = \left( \int_{M} {f} \alpha \diff {\mu} \right) \cdot \left( \int_{M} {g} \diff {\mu} \right).
$$ 
It remains to prove that whenever $f\in H_0([Z]_\Gamma)^\perp = [ \pi^* (L^2(\overline{M},\diff {\overline{\mu}})) ]^\perp $ we have
$$
\langle f\circ \phi^V_t, g \rangle_{L^2(M, \diff \mu)}   \to 0\,.
$$ 
%     Case (i): end of the proof     %   
In the present case we choose $W=Y$ in Lemma~\ref{lemma:correlationseasy} and, for every $\sigma>0$, we consider the pushed curves 
$$
\gamma^\sigma_{x,t} (s)  =  \phi^V_t\circ \phi^Y_s (x) \,, \quad \text{ for } s\in [0, \sigma]\,.
$$
For all $t\in \R$, let $A_t$ and $B_t$ denote the functions
$$
A_t(x) := -\int_0^t   \frac{Y\alpha}{\alpha} \circ \phi^V_\tau (x) \diff \tau  \quad \text{ and } \quad  B_t(x):=  \int_0^t   \frac{1}{\alpha} \circ \phi^V_\tau (x) \diff \tau\,.
$$
By Lemma~\ref{corollary:push_forward}, by taking into account that $Z\alpha =0$, we have
\begin{equation}
\label{eq:gamma_Y}
\frac{d\gamma^\sigma_{x,t}}{ds}  (s) = [A_t \circ \phi^Y_s(x)] V(\gamma^\sigma_{x,t}(s)) +   Y  (\gamma^\sigma_{x,t}(s)) - [ B_t \circ \phi^Y_s(x)] Z(\gamma^\sigma_{x,t}(s)) \,.
\end{equation}
Let $a >0$ denote the minimum of $1/\alpha$ on $M$. Then, for all $t \in \R$ and for all $x \in M$, we have $B_t(x)>at$ and, by the ergodic theorem (and unique ergodicity of completely irrational nilflows), 
$$
\lim_{t\to +\infty} \frac{B_t(x)}{t} = \int_M  \frac{1}{\alpha} \omega_V = 1 \text{\ \ \ and\ \ \ } \lim_{t\to +\infty} \frac{A_t(x)}{t} = -\int_M  \frac{Y\alpha}{\alpha} \omega_V = -\int_M Y \alpha \diff \mu =0,
$$
uniformly in $x \in M$.
We write
$$
\int_{\gamma^\sigma_{\cdot,t}} f = \int_0^{\sigma}  f\circ \phi^V_t\circ \phi^Y_s  \diff s =   \int_0^{\sigma} (B_t\circ \phi^Y_s)^{-1}      (B_t \circ \phi^Y_s)  f\circ \phi^V_t\circ \phi^Y_s  \diff s
$$
and after integration by parts
\begin{equation}
\label{eq:induction_case_1}
\begin{split}
\int_0^{\sigma}  f\circ \phi^V_t\circ \phi^Y_s  \diff s  = & (B_t\circ \phi^Y_{\sigma})^{-1}  \int_0^{\sigma}  (B_t \circ \phi^Y_s)  f\circ \phi^V_t\circ \phi^Y_s  \diff s \\ &- \int_0^{\sigma} \left(\frac{\diff}{\diff s} (B_t\circ \phi^Y_s)^{-1}\right) 
\left( \int_0^s  (B_t \circ \phi^Y_r)  f\circ \phi^V_t\circ \phi^Y_r  \diff r \right) \diff s \,.
\end{split}
\end{equation}
By  Lemma~\ref{lemma:correlationseasy} for $W=Y$, it is enough to prove that the two terms on the right-hand side of \eqref{eq:induction_case_1} converge pointwise to zero for any $f\in H_0([Z]_\Gamma)^\perp  $.

\smallskip
\noindent \emph{First term of the RHS of \eqref{eq:induction_case_1}.}
By \eqref{eq:gamma_Y}, we can write the second integral in the RHS of  \eqref{eq:induction_case_1} in terms of a path integral along $\gamma_{x,t}^{\sigma}$, that is,
$$
\int_0^{\sigma}  (B_t \circ \phi^Y_s)  f\circ \phi^V_t\circ \phi^Y_s  \diff s = \int_{\gamma_{x,t}^{\sigma}}  f \hat Z \,.
$$
Let us notice that the Lie derivative operator $\mathscr{L}_Z$ is invertible on $H_0([Z]_\Gamma)^\perp$; we can therefore write $f=Zg$.
%In our case we have that $f = Zg$, in fact $f\in H_0([Z]_\Gamma)^\perp$. 
It is not restrictive to assume that $g$ is smooth.
Let $\{\hat V, \hat Y, \hat Z\}$ be a frame of $1$-forms dual to $\{V, Y, Z\}$. 
By \eqref{eq:gamma_Y}, we have
$$
\int_{\gamma_{x,t}^{\sigma}} Z g \hat Z = \int_{\gamma_{x,t}^{\sigma}} \diff g - \int_{\gamma_{x,t}^{\sigma}} Y g \hat{Y} - \int_{\gamma_{x,t}^{\sigma}} V g \hat{V};
$$
hence, we derive the formula
$$
\begin{aligned}
\int_0^{\sigma}  (B_t \circ \phi^Y_s)  f\circ \phi^V_t\circ \phi^Y_s  \diff s = & \int_{\gamma_{x,t}^{\sigma}} \diff g -  \int_0^{\sigma}  Yg \circ \phi^V_t\circ \phi^Y_s  \diff s \\ & - 
\int_0^{\sigma}  (A_t \circ \phi^Y_s) Vg \circ \phi^V_t\circ \phi^Y_s  \diff s  \,.
\end{aligned}
$$
By Stokes Theorem,
$$
\left\lvert \int_0^{\sigma}  (B_t \circ \phi^Y_s)  f\circ \phi^V_t\circ \phi^Y_s  \diff s \right\rvert \leq 2 \norm{g}_{\infty} + \sigma \norm{Yg}_{\infty} + \norm{Vg}_{\infty} \int_0^{\sigma} |A_t \circ \phi^Y_s| \diff s,
$$
thus, we can bound the first term on the right-hand side of \eqref{eq:induction_case_1} by
$$
\begin{aligned}
& \left\lvert (B_t\circ \phi^Y_{\sigma})^{-1}  \int_0^{\sigma}  (B_t \circ \phi^Y_s)  f\circ \phi^V_t\circ \phi^Y_s  \diff s \right\rvert \\
& \quad \leq \frac{2 \norm{g}_{\infty} + \sigma \norm{Yg}_{\infty}}{at} + \frac{\norm{Vg}_{\infty}}{a} \int_0^{\sigma} \frac{|A_t \circ \phi^Y_s|}{t} \diff s.
\end{aligned}
$$
Since $|A_t \circ \phi^Y_s|/t$ converges to zero uniformly in $M$, we deduce that 
\begin{equation}
\label{eq:induction_1_first_term}
\lim_{t \to \infty}  (B_t\circ \phi^Y_{\sigma})^{-1}  \int_0^{\sigma}  (B_t \circ \phi^Y_s)  f\circ \phi^V_t\circ \phi^Y_s  \diff s =0.
\end{equation}

\smallskip
\noindent \emph{Second term of the RHS of \eqref{eq:induction_case_1}.}
We can rewrite the second term as 
$$
\begin{aligned}
& - \int_0^{\sigma} \left(\frac{\diff}{\diff s} (B_t\circ \phi^Y_s)^{-1}\right) \left( \int_0^s  (B_t \circ \phi^Y_r)  f\circ \phi^V_t\circ \phi^Y_r  \diff r \right) \diff s \\
& \quad \qquad \qquad \qquad \qquad =  \int_0^{\sigma} \frac{ \frac{\diff}{\diff s} B_t\circ \phi^Y_s}{ (B_t\circ \phi^Y_s)^2}\left( \int_0^s  (B_t \circ \phi^Y_r)  f\circ \phi^V_t\circ \phi^Y_r  \diff r \right) \diff s. 
\end{aligned}
$$
By definition of $B_t$, we have
$$
\frac{\diff}{\diff s} B_t\circ \phi^Y_s = \int_0^t \frac{\diff}{\diff s} \left( \frac{1}{\alpha} \circ \phi^V_{\tau} \circ \phi^Y_s \right) \diff \tau = \int_0^t [(\phi^V_{\tau})_{\ast}(Y)] \frac{1}{\alpha} \circ \phi^V_{\tau} \circ \phi^Y_s \diff \tau,
$$
and by Lemma~\ref{corollary:push_forward}, since $Z\alpha =0$, we get 
$$
\frac{\diff}{\diff s} B_t\circ \phi^Y_s = \int_0^t   ( A_\tau \circ \phi^Y_s )(  V \frac{1}{\alpha} \circ \phi^V_\tau \circ \phi^Y_s )  \diff \tau  + \int_0^t   Y\frac{1}{\alpha} \circ \phi^V_\tau \circ \phi^Y_s  \diff \tau.  
$$
The term
$$ 
\int_0^t    Y \frac{1}{\alpha}   \circ \phi^V_\tau \circ \phi^Y_s \diff \tau 
$$
clearly grows at most linearly with time. 
Integration by parts gives 
\begin{align*}
\int_0^t   A_\tau  \left( V \frac{1}{\alpha}   \circ \phi^V_\tau  \right)  \diff \tau &=  A_t   \int_0^t   V \frac{1}{\alpha}  \circ \phi^V_\tau  \diff \tau - \int_0^t \frac{\diff A_\tau}{\diff \tau}  \left( \int_0^\tau  V\frac{1}{\alpha}  \circ \phi^V_s \diff s \right)\diff \tau\\
&=  A_\tau  \left(\frac{1}{\alpha}    \circ \phi^V_\tau  -\frac{1}{\alpha} \right)  \diff \tau-   \int_0^t \left( - \frac{Y\alpha}{\alpha} \circ \phi^V_\tau  \right)  \left(\frac{1}{\alpha}  \circ \phi^V_\tau -\frac{1}{\alpha} \right)\diff \tau \\
&=\frac{A_t}{\alpha \circ \phi^V_t } - \frac{A_t}{\alpha} +  \int_0^t \left( \frac{Y\alpha}{\alpha^2}  \circ \phi^V_\tau \right) \diff \tau - \frac{1}{\alpha} \int_0^t \left( \frac{Y\alpha}{\alpha}  \circ \phi^V_\tau \right) \diff \tau
\\
&=\frac{A_t}{\alpha  \circ \phi^V_t} - \frac{A_t}{\alpha} -  \int_0^t Y\frac{1}{\alpha}  \circ \phi^V_\tau  d\tau +\frac{A_t}{\alpha}  =\frac{A_t}{\alpha \circ \phi^V_t } -  \int_0^t Y\left(\frac{1}{\alpha}\right)  \circ \phi^V_\tau  d\tau .
\end{align*}
Hence, it follows from the definition of $A_t$ that there exists a constant $C>1$ such that for all $x \in M$,
$$
\left\lvert \frac{\diff}{\diff s} B_t\circ \phi^Y_s(x) \right\rvert \leq C t.
$$
Therefore, 
\begin{equation*}
 \left\lvert \frac{ \frac{\diff}{\diff s} B_t\circ \phi^Y_s}{ (B_t\circ \phi^Y_s)^2}\left( \int_0^s  (B_t \circ \phi^Y_r)  f\circ \phi^V_t\circ \phi^Y_r  \diff r \right) \right\rvert %\\
 \leq \frac{C}{a} \left\lvert \frac{1}{B_t\circ \phi^Y_s} \int_0^s  (B_t \circ \phi^Y_r)  f\circ \phi^V_t\circ \phi^Y_r  \diff r \right\rvert,
\end{equation*}
which converges to zero by \eqref{eq:induction_1_first_term}.
This implies that the second term on the right-hand side of \eqref{eq:induction_case_1} converges to zero as well. We conclude that 
$$
\int_0^{\sigma}  f\circ \phi^V_t\circ \phi^Y_s (x) \diff s \to 0
$$ 
pointwise, hence the proof is complete by an application of Lemma~\ref{lemma:correlationseasy} and Remark~\ref{rk:mixingconclusion}.
\end{proof}

\subsection{Case (ii) (non-coboundary case)}\label{sec:noncoboundary_case}

In this subsection we consider the complementary case in which $
\alpha^{\perp}(x) = \alpha(x) - \overline{\alpha} \in H_0([Z]_{\Gamma})^{\perp}$
is  \emph{not} a  measurable coboundary w.r.t.~$X$. In this case we will show directly that the flow $\phi^V$ on $M$ is mixing, by proving % (see Proposition \ref{noncoboundary_case}).
 the Proposition~\ref{noncoboundary_case} below. It is for this case that, to prove mixing, we will exploit the shearing of curves produced by the growth of Birkhoff sums for non-coboundaries proved in \S\ref{sec:growth}.

The notation in the following proposition is the one fixed at the end of \S\ref{sec:decomposition}.
 \begin{proposition}\label{noncoboundary_case}
If $\alpha^{\perp}$ is not a measurable coboundary with respect to $X$, the flow $\{ \phi_t^{ V} \}_{t \in \R}$ given by $V=\frac{1}{\alpha}X$ on $ M$ is mixing. 
\end{proposition}

The rest of the section is devoted to the proof of Proposition \ref{noncoboundary_case}.

\smallskip

We are assuming that the function $\alpha^{\perp}$ is not a measurable $X$-coboundary. Recall that 
$\alpha^{\perp}$ is the projection of $\alpha$ onto $H_0([Z]_{\Gamma})^{\perp}$; in particular it is also a trigonometric polynomial with respect to $\z=\exp([Z]_\Gamma)$ of the same degree as $\alpha$ without constant term.
For every $\sigma>0$, let $\gamma^\sigma_{x,t}$ be the path defined as
$$
\gamma^\sigma_{x,t} (s) = (\phi^V_t \circ \phi^{Z}_s) (x) \,, \quad \text{ for all } s\in [0,\sigma]\,.
$$
Let $D_t$ denote the function on $M$ defined as
\begin{equation}
\label{eq:time_change_int}
D_t (x) :=  -  \int_0^t  \frac{Z \alpha}{\alpha} \circ \phi^V_{\tau} (x) \diff \tau \,.
\end{equation}
By Lemma~\ref{corollary:push_forward}, we have 
$$
\frac{\diff \gamma^\sigma_{x,t}}{\diff s}  (s) = [D_t \circ \phi^Z_s(x)] V(\gamma^\sigma_{x,t}(s)) +   Z (\gamma^\sigma_{x,t}(s))\,,
$$
thus the function $D_t$ describes shearing of the curves $\gamma^\sigma_{x,t}$ in direction $V$. We can write
\begin{equation}
\label{eq:int_along_gamma}
\int_{\gamma^\sigma_{x,t}} f \hat V   =  \int_0^\sigma   (f \circ \phi^V_t \circ \phi^Z_s)(x)  \, 
(D_t \circ \phi^Z_s)(x) \diff s\,.
\end{equation}
In the above formula the LHS  is a line integral of a $1$-form, which equals the (uniformly bounded)  integral of an exact $1$-form up to a uniformly bounded error (with respect to time $t\geq 0$),  whenever the function $f$ is a coboundary for the flow $\phi^V$.  In our argument, mixing will follow from the decay of correlations of coboundaries, which by ergodicity form a dense set in space of square integrable functions. The decay of correlations of coboundaries will in turn be derived from the boundedness of the RHS in formula
\eqref{eq:int_along_gamma} and from the growth in measure of the function $D_t$.

\begin{remark}\label{rk:equiv_tautilde}
We can equivalently express $D_t$ as an integral along orbits of $\phi^X$ as follows.  Let us recall (see \S\ref{sec:basic}) that, for all $x\in M$, we have $\phi^V_{\tau(x,t)}(x) = \phi^X_t(x)$, where $\tau(x,t) = \int_0^t \alpha \circ \phi^X_r(x) \diff r$.
By changing variable (setting $\tau(x,r)=\int_0^r \alpha \circ \phi^X_s(x) \diff s$ so $\diff \tau =\alpha\circ \phi^X_r\diff r$) and since $\overline{\alpha}$ is $Z$-invariant we can rewrite $D_t$ in \eqref{eq:time_change_int} as follows:
\begin{align}
%\label{eq:D}
\nonumber
D_t (x) & =  -\int_0^t \frac{Z\alpha}{\alpha}\circ  \phi^V_{\tau(x,r)}(x) \diff \tau =- \int_0^{\widetilde{\tau}(x,t)} \frac{Z\alpha}{\alpha}\circ  \phi^X_{r}(x) \left( \alpha\circ \phi^X_r (x)\right) \diff t=\\ \nonumber
&  = -\int_0^{\widetilde{\tau}(x,t)} Z\alpha \circ \phi_r^X(x) \diff r= -\int_0^{ \widetilde{\tau} (x,t)} Z ( {\alpha}^{\perp} ) \circ \phi^X_{r}(x) \diff r, 
%= \widetilde{D}_{ \widetilde{\tau} (x,t)}(x),
\end{align}
where $\widetilde{\tau} (x,t)$ is such that $t = \int_0^{ \widetilde{\tau} (x,t) }\alpha \circ \phi^X_r(x) \diff r$, or, in other words,
\begin{equation}\label{def:tildetau}
 \widetilde{\tau} (x,t) = \int_0^t \frac{1}{\alpha} \circ \phi^V_r(x) \diff r.
\end{equation}
\end{remark}

The key estimate to prove mixing in this case is given by the following shearing result. 

\begin{proposition}[Growth of $D_t$]\label{prop:D_t_grows}
 For every constants $C>1$,  
$$
\lim_{t\to \infty} \mu \left\{x\in M :\,   \vert D_t (x) \vert <C\right\} \,=\, 0\,.
$$
\end{proposition}

The next subsection is fully devoted to the proof of Proposition~\ref{prop:D_t_grows}.

\subsubsection{Proof of Proposition~\ref{prop:D_t_grows}.} Let $\{\alpha_{\bf v}\vert {\bf v} \in \Lambda^\ast \}$ denote the (finite) set of Fourier components of the trigonometric
polynomial $\alpha^\perp$ and, for every ${\bf v}\in \Lambda^\ast$ and  let 
$$
D^{\bf v}_t(x) :=  -\int_0^{ \widetilde{\tau} (x,t)} Z\alpha_{\bf v}  \circ \phi^X_{r}(x) \diff r=-\int_0^{ \widetilde{\tau} (x,t)} Z\alpha^\perp_{\bf v}  \circ \phi^X_{r}(x) \diff r\,.
$$
We split the proof Proposition~\ref{prop:D_t_grows} into two Lemmas. The first (Lemma~\ref{lemma:sumD_t_grows}) shows  that  $\sum_{{{\bf v}}\in \Lambda^\ast}|D^{\bf v}_t|$ grows in measure;  the second (Lemma \ref{lemma:D_t_comparison}) compares $D_t$  with  $\sum_{{{\bf v}}\in \Lambda^\ast}|D^{\bf v}_t|$. 

 Along the way, we establish a quasi-invariance property of the
function $\sum_{{{\bf v}}\in \Lambda^\ast}|D^{\bf v}_t|$, which we apply in the the
proof of Lemma~\ref{lemma:D_t_comparison} and of Lemma~\ref{lemma:distortion}
on distortion bounds.

\begin{lemma}
\label{lemma:sumD_t_grows}
For every $C>1$, we have
$$
\lim_{t \to \infty} \mu \left\{ x\in M : \,  \sum_{{\bf v}\in \Lambda^\ast   |D^{\bf v}_t(x)| < C} \right\} = 0.
$$
\end{lemma}

\begin{proof}
Let $W$ be any central vector field such that  $W \in [Z]_\Gamma$. From the definition \eqref{def:tildetau} of $\widetilde{\tau}$, by differentiation by $W$ we have
\begin{equation}
\label{eq:Wtau}
0= W (t) = W \left( \int_0^{ \widetilde{\tau} (x,t) }\alpha \circ \phi^X_r(x) \diff r \right)  =
\left (W \widetilde{\tau}  (x,t) \right)  \alpha \circ \phi^X_{ \widetilde{\tau} (x,t) }(x)  +  \int_0^{ \widetilde{\tau} (x,t) } W\alpha \circ \phi^X_r (x)\diff r \,.
\end{equation}
It follows that 
\begin{equation}
\label{eq:tau_diff_bound}
\vert W\widetilde{\tau} (x,t)  \vert \leq  \frac{1}{\min \alpha}  \left \vert  \int_0^{ \widetilde{\tau} (x,t) } W\alpha \circ \phi^X_r(x) \diff r   \right\vert  \,.
\end{equation}
We also have
$$
\begin{aligned}
\frac{\diff}{\diff s}   D_t^{\bf v} \circ \phi^W_s (x) &=- \frac{\diff}{\diff s}  \left( \int_0^{\widetilde\tau(\phi^W_s(x),t)} Z\alpha_{\bf v} \circ \phi^X_r \circ \phi^W_s(x) \diff r   \right) \\ &=  -  \left( W\widetilde{\tau} (\phi^W_s(x),t)\right)   Z\alpha_{\bf v} \circ \phi^X_{\widetilde{\tau} (\phi^W_s(x),t)  }  \circ \phi^W_s(x)  - \int_0^{\widetilde\tau(\phi^W_s(x),t)} WZ \alpha_{\bf v} \circ \phi^X_r \circ \phi^W_s(x) \diff r
\end{aligned} 
$$
By the above identity combined with \eqref{eq:tau_diff_bound}, we have the following estimate
$$
\begin{aligned}
\left\vert \frac{\diff}{\diff s} D^{\bf v}_t \circ \phi^W_s (x) \right\vert  \leq  \frac{\max \vert Z \alpha \vert}{\min \alpha} & \left\vert   
\int_0^{ \widetilde{\tau} (\phi^W_s (x),t) } W\alpha \circ \phi^X_r \circ \phi^W_s (x) \diff r  \right \vert  \\ &+
\vert \langle W,{\bf v}\rangle \vert  \left \vert \int_0^{\widetilde\tau(\phi^W_s(x),t)} Z\alpha_{\bf v} \circ \phi^X_r \circ \phi^W_s(x) \diff r\right\vert\,,
\end{aligned}
$$
since $\langle Z, {\bf v}\rangle \not=0$ for all ${\bf v} \in \Lambda^\ast$, by the triangular inequality
we have
\begin{equation}
\label{eq:dir_der_bound}
\left\vert   
\int_0^{ \widetilde{\tau} (\phi^W_s (x),t) } W\alpha \circ \phi^X_r \circ \phi^W_s (x) \diff r  \right \vert
\leq \sum_{{\bf v}\in \Lambda^\ast  }  \left \vert \frac{\langle W,{\bf v}  \rangle}{\langle Z,{\bf v}  \rangle    } \right\vert \,  
 \left\vert   
\int_0^{ \widetilde{\tau} (\phi^W_s (x),t) } Z\alpha_{\bf v} \circ \phi^X_r \circ \phi^W_s (x) \diff r  \right \vert \,.
\end{equation}
Hence there exists a constant $K(\alpha)>1$ such that, for every $W \in [Z]_\Gamma$ and every ${\bf v}\in \Lambda^\ast$,  we have 
$$
\left\vert \frac{\diff}{\diff s} D^{\bf v}_t \circ \phi^W_s (x) \right\vert  \leq K(\alpha)  \sum_{{\bf v}\in \Lambda^\ast  }  
 \left\vert   
\int_0^{ \widetilde{\tau} (\phi^W_s (x),t) } Z\alpha_{\bf v} \circ \phi^X_r \circ \phi^W_s (x) \diff r  \right \vert
= K(\alpha) \sum_{{{\bf v}}\in \Lambda^\ast  }   \vert  D^{\bf v}_t \circ \phi^W_s(x) \vert\,.
$$
From the identity
$$
\begin{aligned}
\frac{1}{2} \frac{\diff}{\diff s} \left( \sum_{{\bf v} \in \Lambda^\ast} \vert D^{\bf v}_t \circ \phi^W_s (x) \vert^2 \right)  =
 \re \left[ \sum_{{\bf v} \in \Lambda^\ast} \overline{ D^{\bf v}_t} \circ \phi^W_s (x)   \frac{\diff}{\diff s} D^{\bf v}_t \circ \phi^W_s (x) \right]\,,
\end{aligned} 
$$
since $\alpha^\perp$ is a trigonometric polynomial with respect to $\z=\exp [Z]_\Gamma$ and hence all the sums  in the above formula are finitely supported,
we finally derive  that there exists a constants $K'(\alpha)>1$ such that the following bound holds:
\begin{equation}
\label{eq:D_diff_bound}
\frac{\diff}{\diff s} \left( \sum_{{\bf v} \in \Lambda^\ast} \vert D^{\bf v}_t \circ \phi^W_s (x) \vert^2 \right)  \leq K(\alpha)  \left( \sum_{{\bf v} \in \Lambda^\ast} \vert D^{\bf v}_t \circ \phi^W_s (x) \vert \right)^2 \leq 
K'(\alpha)  \left( \sum_{{\bf v} \in \Lambda^\ast} \vert D^{\bf v}_t \circ \phi^W_s (x) \vert^2 \right)\,.
\end{equation}
Let ${\mathcal B}_t (C)\subset M$ denote the set    
$$
{\mathcal B}_t (C) := \left\{x \in M :  \sum_{{\bf v}\in \Lambda^\ast  \vert D^{\bf v}_t (x)\vert < C} \right\} \,.
$$
Any rational central vector field $W \in [Z]_\Gamma$ has only closed orbits  of the same 
period~$T_W>0$, so that by the differential form of Gr\"onwall inequality (that is, by comparison) we derive from the bound in formula
\eqref{eq:D_diff_bound}  the following estimate  on  the set ${\mathcal B} ^{W}_t (C)$, defined as the saturation of the set ${\mathcal B}_t (C)$ by the orbit foliation of the flow $\phi^W$:  there exists a constant $K_W(\alpha) >1$ such that 
$$
\sum_{{\bf v} \in \Lambda^\ast} \vert D^{\bf v}_t (x) \vert^2 \leq K_W(\alpha)  C^2 \,, \quad \text{ for all } x\in {\mathcal B}^{W}_t (C)\,.
$$
Let then $\{W_1, \dots, W_d\}  \subset [Z]_\Gamma$ denote a rational basis. Let $\hat {\mathcal B}_t (C)
:= {\mathcal B}^{W_1, \dots, W_d}_t (C)$ denote the set obtained by successive saturation of the set
${\mathcal B}_t(C)$  by the orbit foliations of the completely periodic flows $\phi^{W_1}, \dots, \phi^{W_d}$. 
By the above argument, based on the estimate~\eqref{eq:D_diff_bound}, assuming that there exists a 
constant $K_{W_1, \dots, W_j} (\alpha)>1$ such that
$$
 \sum_{{\bf v} \in \Lambda^\ast} \vert D^{\bf v}_t (x) \vert^2 \leq K_{W_1, \dots, W_j} (\alpha) C^2 \,, \quad \text{ for all } x\in {\mathcal B}^{W_1, \dots, W_j}_t (C)\,,
$$
we derive that there exists a constant $K_{W_1, \dots, W_{j+1}} (\alpha)>1$ such that
$$
 \sum_{{\bf v} \in \Lambda^\ast} \vert D^{\bf v}_t (x) \vert^2 \leq K_{W_1, \dots, W_{j+1}} (\alpha) C^2 \,, \quad \text{ for all } x\in {\mathcal B}^{W_1, \dots, W_{j+1}}_t (C)\,.
$$
Thus, by finite induction we derive that  there exists a constant ${\hat K}(\alpha)>1$ such that 
\begin{equation}
\label{eq:D_bound}
\sum_{{\bf v} \in \Lambda^\ast} \vert D^{\bf v}_t (x) \vert  \leq {\hat K} (\alpha) C \,, \quad \text{ for all } x\in 
 \hat {\mathcal B}_t (C)\,.
\end{equation}
From the estimates~\eqref{eq:tau_diff_bound}, \eqref{eq:dir_der_bound} and \eqref{eq:D_bound} it
also follows that there exists a constant ${\hat K}'(\alpha)>1$ such that, for all $W \in [Z]_\Gamma$,
\begin{equation}
\label{eq:Wtau_bound}
\vert W\widetilde{\tau} (x,t)  \vert \leq {\hat K}' (\alpha) C \,, \quad \text{ for all } x\in 
 \hat {\mathcal B}_t (C)\,.
\end{equation}
Let $\widetilde{\tau}_Z (x,t)$ denote the ($\phi^{[Z]_\Gamma}$-invariant) average of the function $\widetilde{\tau} (y,t)$ over the orbit $y\in \T_{ [Z]_\Gamma}(x)$ of $x\in M$ under the central Abelian action  
$\phi^{[Z]_\Gamma}$ (which is a finite dimensional torus) with respect to the conditional volume measure. 
By the bound in formula~\eqref{eq:Wtau_bound} (which holds for all $W \in [Z]_\Gamma$ and hence for $W=Z$ and, by definition of $\mathcal{B}_t (C)$, for all points on the orbit $(\phi_s^Z(x))_{s\geq 0}$ if $x\in {\mathcal B} _t (C)$), there exists a constant $K''(\alpha)>1$  such that
\begin{equation}
\label{eq:D-tau}
\vert \widetilde{\tau}(x,t) - \widetilde{\tau}_Z(x,t) \vert \leq   {\hat K}''(\alpha) C  \,, \quad \text{ for all }  x \in \hat {\mathcal B} _t (C)\,,
\end{equation}
and, as a consequence, we have that 
\begin{equation}
\label{eq:meas_bound}
\left\vert \int_0^{ \widetilde{\tau}_Z (x,t)} Z {\alpha}^{\perp}  \circ \phi^X_{r}(x) \diff r   \right\vert  \leq 
({\hat K}(\alpha) + {\hat K}''(\alpha) \max\vert Z\alpha\vert      ) C   \,, \quad \text{ for all }  x \in \hat {\mathcal B}_t (C)\,.
\end{equation}
Since $\tilde \tau_Z(\cdot,t)$ is constant along the orbits of $\phi^{[Z]_\Gamma}$, for all $W\in [Z]_\Gamma$
we have that
$$
\int_0^{\widetilde\tau_Z(x,t)} W \alpha^{\perp} \circ \phi_{r}^X  (x) \diff r = \left.\frac{\diff }{\diff s}\right\vert_{s=0} \left( \int_0^{\widetilde\tau_Z(\cdot,t)}  \alpha^{\perp} \circ \phi_{r}^X \diff r \right) \circ \phi_s^W(x).
$$
It follows that there exists a constant ${\hat K}^{(3)}(\alpha)>1$ such that
\begin{equation}
\label{eq:perp_bound_1}
\left\vert \int_0^{ \widetilde{\tau}_Z (x,t)} {\alpha}^{\perp}  \circ \phi^X_{r}(x) \diff r   \right\vert  \leq 
 {\hat K}^{(3)}(\alpha) C   \,, \quad \text{ for all }  x \in \hat {\mathcal B}_t (C)\,.
\end{equation}
By the estimate~\eqref{eq:D-tau} it follows that, for all $x \in  \hat {\mathcal B}_t (C)$, we have
$$
\begin{aligned}
&\left\vert \int_0^{ \widetilde{\tau}_Z(x,t)} \overline{\alpha} \circ \phi^X_{r}(x) \diff r  -t \right\vert
= \left\vert \int_0^{ \widetilde{\tau}_Z(x,t)} \overline{\alpha} \circ \phi^X_{r}(x) \diff r  -
\int_0^{ \widetilde{\tau}(x,t)} \alpha \circ \phi^X_{r}(x) \diff r \right\vert  \\ 
&\leq \int_0^{ \widetilde{\tau}_Z(x,t) -  \widetilde{\tau}(x,t) } \overline{\alpha} \circ \phi^X_{r}(x) \diff r 
+  \left \vert \int_0^{ \widetilde{\tau}(x,t)
-\widetilde{\tau}_Z(x,t)} \alpha^\perp \circ \phi^X_{r}(x) \diff r+ \int_0^{ \widetilde{\tau}_Z(x,t)} \alpha^\perp \circ \phi^X_{r}(x) \diff r  \right\vert  \\
&\leq \left[(\max \overline{\alpha}  + \max \vert \alpha^\perp \vert){\hat K}''(\alpha)  + {\hat K}^{(3)}(\alpha)\right]   C 
\end{aligned} 
$$
For every $x \in M$,  let $\overline{\tau} (x,t)$ denote the unique solution of the equation
$$
t = \int_0^{\overline{\tau} (x,t)}  \overline{\alpha} \circ \phi^X_r (x) \diff r  \,.
$$
By the above estimates it follows immediately that 
\begin{equation}
\label{eq:D-tau_bis}
\vert \overline{\tau}(x,t) - \widetilde{\tau}_Z(x,t) \vert \leq 
 \left[ \frac{ (\max \overline{\alpha}  + \max \vert \alpha^\perp \vert) {\hat K}''(\alpha)  + {\hat K}^{(3)}(\alpha)}
 { \min  \overline {\alpha}  }\right]    C  \,, \quad \text{ for all }  x \in \hat {\mathcal B} _t (C)\,,
\end{equation}
hence by \eqref{eq:perp_bound_1} there exists a constant ${\hat K}^{(4)}(\alpha)>1$ such that
\begin{equation}
\label{eq:perp_bound_2}
\left\vert \int_0^{ \overline{\tau} (x,t)} {\alpha}^{\perp}  \circ \phi^X_{r}(x) \diff r   \right\vert  \leq 
 {\hat K}^{(4)}(\alpha) C   \,, \quad \text{ for all }  x \in \hat {\mathcal B}_t (C)\,.
\end{equation}
Since the function $\overline{\alpha}$ is $\phi^{[Z]_\Gamma}$-invariant  and  $\phi^X$ commutes
with  $\phi^{[Z]_\Gamma}$, it follows that the function $\overline{\tau} (x,t)$ is $\phi^{[Z]_\Gamma}$-invariant
for every (fixed) $t \in \R$.  In addition, by its definition the function $\overline {\tau}$ is a cocycle and defines a time-change of $\phi^X$. In fact, we have
$$
\begin{aligned}
t= t+s- s&= \int_0^{\overline{\tau} (x,t+s) }  \overline{\alpha} \circ \phi^X_r (x) \diff r - 
\int_0^{\overline{\tau}(x,s)}  \overline{\alpha} \circ \phi^X_r (x) \diff r \\ &= \int_{\overline{\tau}(x,s)} ^{\overline{\tau} (x,t+s)}  \overline{\alpha} \circ \phi^X_r (x) \diff r  = \int _0 ^{\overline{\tau} (x,t+s) - \overline{\tau} (x,s)}   \overline{\alpha} \circ \phi^X_r \circ \phi^X_{\overline{\tau} (x,s)}(x)  \diff r 
\end{aligned} 
$$
which by uniqueness of the solution implies that
$$
  \overline{\tau} (\phi^X_{\overline{\tau} (x,s)}(x) ,t) =  \overline{\tau} (x,t+s) - \overline{\tau} (x,s) \,.
$$
We can therefore express the ergodic integral of formula \eqref{eq:perp_bound_2}  for the flow $\phi^X$, up to the $\phi^{[Z]_\Gamma}$-invariant time $\overline{\tau}(x,t)$, as an ergodic integral up to time $t$ for a time-change $\phi^U$ of $\phi^X$ which commutes with $\phi^{[Z]_\Gamma}$. 
Let $U$ be the vector field defined as $U= (1/{\overline\alpha}) X$. For every continuous function $f$ on $M$
and for all $(x,t) \in M\times \R$ we have 
$$
\int_0^t  (f /\overline{\alpha})\circ \phi^U_r (x) \diff r =   \int_0^{\overline{\tau}(x,t)}  f \circ \phi^X_r (x) \diff r 
$$
Since the function $1/{\overline\alpha} \in$ is $[Z]_\Gamma$-invariant (that is it belongs to $H_0([Z]_\Gamma)$, the function $\alpha^\perp$ is a trigonometric polynomial of degree $m$ with respect to $[Z]_\Gamma$, and $\alpha^\perp$ is not a coboundary with respect to $X$, or, equivalently, $\alpha^\perp/ {\overline\alpha}$ is not a coboundary with respect to $U$, by Theorem~\ref{thm:stretch_ergodic_integral}, for every $C>1$, we have that
\begin{equation}
\label{eq:D^Wgrowth}
\lim_{t \to \infty} \mu 
\left \{ \left\vert  \int_0^{\overline {\tau}(x,t)}\alpha^\perp \circ \phi^X_r (x) \diff r \right\vert < C\right\}=\lim_{t \to \infty} \mu \left\{ \left\vert  \int_0^t(\alpha^\perp/ {\overline \alpha}) \circ \phi^U_r (x) \diff r \right\vert < C\right\} =   0.
\end{equation}
hence there exists $T_0 > 0$ such that for all $t \geq T_0$ we have
$$
\mu \left\{ x \in M :\left\vert  \int_0^{\overline {\tau}(x,t)}\alpha^\perp \circ \phi^X_r (x) \diff r \right\vert  \leq
{\hat K}^{(4)}(\alpha)  \,C \right\} \leq \varepsilon,
$$
which  by \eqref{eq:perp_bound_2} implies that, for all $t\geq T_0$, we have
$$
\mu \left\{x\in M :  \sum_{{\bf v}\in \Lambda^\ast} \vert D^{\bf v}_t(x)\vert <C \right\}  =  \mu \left( \hat {\mathcal B} _t (C)  \right) \leq \varepsilon\,.
$$
\end{proof}

We then prove that $D_t$ is comparable with $\sum_{{{\bf v}}\in \Lambda^\ast} \vert D^{\bf v}_t\vert$ on sets of large measure (Lemma~\ref{lemma:D_t_comparison} below), so that we can then deduce that it also grows in measure. To do so, we estimate the oscillations of the function $\widetilde{\tau}$ along the orbits of the flow $\phi^Z$ and show a \emph{quasi-invariance} result (Lemma~\ref{lemma:quasi_invariance} below) which shows that the sum of the moduli of the components $D^{\bf v}_t(x)$  changes at most by a fixed multiplicative constant along any bounded orbit segment of  $\phi^Z$ (in addition to 
Lemma~\ref{lemma:D_t_comparison}, Lemma~\ref{lemma:quasi_invariance} is also used later in Lemma~\ref{lemma:distortion} to control \emph{distorsion}). 

\begin{lemma}[Quasi-invariance]  
\label{lemma:quasi_invariance} There exists a constant $C_X(\alpha) >1$ such that, for all 
$x\in M$ and for all $s \in [-1,1]$ and for all $t\geq 0$, we have
$$
 \sum_{{\bf v} \in \Lambda^\ast} \vert D^{\bf v}_t( \phi^Z_s(x)) \vert
 \leq  C_X(\alpha) \sum_{{\bf v} \in \Lambda^\ast} \vert D^{\bf v}_t(x) \vert \,.
$$
\end{lemma}
\begin{proof}
Let us  first estimate the oscillation of the function
$\widetilde{\tau}$ along the orbits of the flow $\phi^Z$.  From the definition \eqref{def:tildetau} of $\widetilde{\tau}$, by differentiation, we have
$$
\frac{\diff}{\diff s} \widetilde{\tau}(\phi^Z_s(x), t) = Z \widetilde{\tau}(\phi^Z_s(x), t) =
-[ \alpha \circ \phi^X_{\widetilde{\tau}(\phi^Z_s(x),t) } (\phi^Z_s(x))]^{-1} \int_0^{\widetilde{\tau}(\phi^Z_s(x),t)}  Z\alpha^\perp \circ \phi^X_r (\phi^Z_s(x)) \diff r\,.
$$
Thus, we derive the following a priori estimate: there exists a constant $C'_X(\alpha)>1$ such that 
$$
\left\vert \frac{\diff}{\diff s} \widetilde{\tau}(\phi^Z_s(x), t) \right\vert \leq  C'_X(\alpha) \left( \vert \widetilde{\tau}(\phi^Z_s(x), t)-  \widetilde{\tau}(x, t)\vert   + \sum_{{{\bf v}}\in \Lambda^\ast} \vert D^{\bf v}_t(x)  \vert  \right) \,.
$$
By Gr\"onwall inequality the above a priori bound implies that, for all $s\in \R$, we have
\begin{equation}\label{eq:taudiff}
\vert \widetilde{\tau}(\phi^Z_s(x), t)-  \widetilde{\tau}(x, t)\vert  \leq (e^{C'_X(\alpha) \vert s \vert} -1)
  \sum_{{{\bf v}}\in \Lambda^\ast} \vert D^{\bf v}_t(x)\vert \,.
\end{equation}
Since for every ${\bf v} \in \Lambda^\ast$, the function $Z\alpha_{\bf v}$ is
an eigenfunction for $\phi^Z$ (which commutes with $\phi^X$), it follows that we have the bound
$$
\begin{aligned}
\vert D^{\bf v}_t( \phi^Z_s(x)) \vert &\leq \left\vert   \int_{\widetilde{\tau}(x, t)}^{\widetilde{\tau}(\phi^Z_s(x), t) } Z\alpha_{\bf v} \circ \phi^X_r \circ \phi_s^Z(x) \diff r      \right\vert  + \left\vert   \int_0^{\widetilde{\tau}(x, t)} Z\alpha_{\bf v} \circ \phi^X_r \circ \phi_s^Z(x) \diff r  \right\vert  \\
& \leq  \max \vert Z\alpha_{\bf v} \vert  (e^{C'_X(\alpha) \vert s \vert} -1) 
  \sum_{{{\bf v}}\in \Lambda^\ast} \vert D^{\bf v}_t(x)\vert + \vert D_t^{\bf v} (x) \vert\,,
\end{aligned}
$$
from which the statement easily follows.
\end{proof}

In the following,  for any $C>1$  and for all $t\geq 0$,  we recall  the notation
\begin{equation}
\label{eq:B_t_C_set}
\mathcal B_t(C) := \left\{x\in M :\,  \sum_{{\bf v} \in \Lambda^\ast} \vert D^{\bf v}_t(x) \vert \leq C \right\} \,.
\end{equation}

\begin{lemma}[Comparison]  
\label{lemma:D_t_comparison}
There exist constants $\Delta(\alpha)>0$, $d_m>0$  and, for every $\delta \in (0,1)$, 
there exists a constant $C_\delta:=C_\delta(\alpha)>1$ such that the following holds. For every $t\geq 0$ and every $\varepsilon>0$, whenever $\mu (\mathcal B_t(C_\delta)) <\varepsilon$, then we also have 
$$
\mu \left\{ x\in  M :\,   \left\vert D_t(x) \right\vert \leq \delta \sum_{{\bf v} \in \Lambda^\ast} \left\vert D^{\bf v}_t(x) \right\vert \right\} 
\leq \Delta(\alpha) \delta^{d_m} + \varepsilon \,.
$$ 	
\end{lemma} 
\begin{proof}
By unique ergodicity of $\phi^X$,  since $\alpha$ is a relative trigonometric polynomial (and so it has
finitely many non-zero coefficients $\alpha_{\bf v}$), for every $\delta \in (0,1)$ there exists a constant $\overline{C}_\delta (\alpha)>1$ such that, for all $x\in M$ and for all $s\in [-1, 1]$, 
\begin{equation}\label{eq:ue}
\sum_{{\bf v} \in \Lambda^\ast}\left\vert \int_{\widetilde{\tau}(x, t)}^{\widetilde{\tau}(\phi^Z_s(x), t) } Z\alpha_{\bf v}\circ \phi^X_r (\phi_s^Z(x)) \diff r \right\vert< \delta(e^{C'_X(\alpha)} -1)^{-1} \vert \widetilde{\tau}(\phi^Z_s(x), t) - \widetilde{\tau}(x, t) \vert     + \overline{C}_\delta (\alpha) \,.
\end{equation}
Thus by the estimate in formula \eqref{eq:taudiff} on $\vert \widetilde{\tau}(\phi^Z_s(x), t) - \widetilde{\tau}(x, t) \vert$, it follows that, for any $x\in M$ and any $s\in [-1,1]$, we have
\begin{equation}\label{eq:diff}
\left\vert \int_{\widetilde{\tau}(x, t)}^{\widetilde{\tau}(\phi^Z_s(x), t) } Z\alpha^\perp \circ \phi^X_r \circ \phi_s^Z(x) \diff r \right\vert<  \delta \sum_{{{\bf v}}\in \Lambda^\ast} \vert D^{\bf v}_t(x)\vert   + \overline{C}_\delta(\alpha)\, .
\end{equation}
We now choose $C_\delta:=C_\delta(\alpha)$ to be any constant bigger than  $\overline{C}_\alpha(\delta)/\delta $, so that if $x$ does not belong to the set 
$\mathcal B_t(C_\delta)$ then, recalling its definition in \eqref{eq:B_t_C_set}, we have
\begin{equation}
\label{eq:B_t_C_bound}
 \delta \sum_{{{\bf v}}\in \Lambda^\ast} \vert D^{\bf v}_t(x)\vert \geq \delta C_\delta \geq \overline{C}_\delta(\alpha).
\end{equation}
 Then, we claim that there exists a constant $C'_X(\alpha)>1$ such that,  for any $x\in M\setminus \mathcal B_t(C_\delta) $, we have
$$
\left\{ s\in [-1, 1]:\,   \left\vert D_t\left(\phi^Z_s(x)\right) \right\vert < \delta \sum_{{{\bf v}}\in \Lambda^\ast} \vert D^{\bf v}_t(\phi^Z_s(x))\vert\right\}\subset 
\left\{s \in [-1, 1]:\,  \vert f_{x,t}(s) \vert < C_X'(\alpha)
 \delta \sum_{{{\bf v}}\in \Lambda^\ast} \vert D^{\bf v}_t(x)\vert \right\}\,.
 $$
where 
$$f_{x,t}(s):= \int_0^{\widetilde{\tau}(x, t)}  Z\alpha^\perp \circ  \phi^X_r \circ \phi^Z_s(x) \diff r .$$
Indeed,  if $s$ belongs to the set on the LHS in the above inclusion, using, in order, the definition of~$f_{x,t}$,  the definition of $D_t$ (see Remark~\ref{rk:equivalent}) and  formula~\eqref{eq:diff}, the definition of the set and formula~\eqref{eq:B_t_C_bound}, and finally Lemma~\ref{lemma:quasi_invariance}, we get
\begin{align*}
\vert f_{x,t}(s)\vert & = \left\vert \int_0^{\widetilde{\tau}(x, t)}  Z\alpha^\perp \circ  \phi^X_r \circ \phi^Z_s(x)\right\vert \leq \vert D_t(\phi^Z_s(x)) \vert  +   \delta  \sum_{{{\bf v}}\in \Lambda^\ast} \vert D^{\bf v}_t(x)\vert + \overline{C}_\delta(\alpha)  \\ & \leq   \delta  \sum_{{{\bf v}}\in \Lambda^\ast} \vert D^{\bf v}_t(\phi^Z_s(x))\vert + 2 \delta  \sum_{{{\bf v}}\in \Lambda^\ast} \vert D^{\bf v}_t(x)\vert\leq  (C_X(\alpha)+2)\delta  \sum_{{{\bf v}}\in \Lambda^\ast} \vert D^{\bf v}_t(x)\vert ,
\end{align*}
which shows that, if we set $C'_X(\alpha) := C_X(\alpha)+2$,  $s$ belongs to the set on the RHS. This concludes the proof of the claim.

\smallskip
\noindent Notice now that the function $f_{x,t}(s):= \int_0^{\widetilde{\tau}(x, t)}  Z\alpha^\perp \circ  \phi^X_r \circ \phi^Z_s(x) \diff r$ is a  trigonometric polynomial in the variable $s\in \R$ (since we removed the dependence from $s$ in the integration limits and $\alpha$ -and hence $Z\alpha^\perp$- are relative trigonometric polynomials).  Thus, by the basic properties of (level sets of) trigonometric polynomials (see also Lemma~\ref{thm:sublevel_sets_polynomials} and its proof), for constants $\Delta_m$, $d_m>0$ depending only on the degree $m$,  for all $x\in M$, 
\begin{equation}\label{trig:estimate}
\text{\rm Leb} \left\{s\in [-1, 1]:\, \vert f_{x,t}(s)\vert  < C'_X(\alpha)
 \delta \sum_{{{\bf v}}\in \Lambda^\ast} \vert D^{{{\bf v}}}_t(x)\vert\right\} \leq \Delta_m (C'_X(\alpha) \delta)^{d_m} = \Delta(\alpha)  \delta^{d_m}  \,,
 \end{equation}
where we defined $\Delta(\alpha):= \Delta_m C'_X(\alpha)^{d_m}$. 
 Hence, by the inclusion proved above we derive that, whenever $x\in M\setminus \mathcal B_t(C_\delta)$, we have
 \begin{equation}\label{eq:meas_est}
\text{\rm Leb} \left\{ s\in [-1, 1]:\,   \left\vert D_t\left(\phi^Z_s(x)\right) \right\vert < \delta \sum_{{{\bf v}}\in \Lambda^\ast} \vert D^{\bf v}_t(\phi^Z_s(x))\vert\right\} \leq \Delta(\alpha)  \delta^{d_m}  \,.
 \end{equation} 
We now show that these estimates allow to conclude by a Fubini-type argument. Consider a partition of $M$ as a union of orbit segments $I$ of unit length of the flow $\phi^Z$. The union $M'$ of all partition intervals~$I$ such that $I \cap (M\setminus \mathcal B_t(C_\delta))= \emptyset$ is a subset of the set $\mathcal B_t(C_\delta)$, hence by hypothesis it has volume at most $\varepsilon>0$.  Let us then consider a partition interval $I$ such that $I \cap (M\setminus \mathcal B_t(C_\delta))\not= \emptyset$ and let $x$ be any point in the intersection.  Since $x \in M\setminus\mathcal B_t(C_\delta)$ and since $I \subset \{\phi^Z_s(x) \vert  s\in [-1, 1]\}$, by the bound in 
formula \eqref{eq:meas_est} it follows that 
 $$
 \text{\rm Leb}_I \left\{ y \in I :\,   \left\vert D_t(y) \right\vert < \delta \sum_{{{\bf v}}\in \Lambda^\ast} \vert D^{\bf v}_t(y)\vert\right\} \leq \Delta(\alpha)  \delta^{d_m} \,.
 $$
 By Fubini's theorem, on the union $M''$ of all partitions intervals $I$ such that $I \cap (M\setminus \mathcal B_t(C_\delta))\not= \emptyset$ we derive the bound
 $$
\mu \left\{ x \in M''  :\,   \left\vert D_t(x) \right\vert < \delta \sum_{{{\bf v}}\in \Lambda^\ast} \vert D^{\bf v}_t(x)\vert\right\} \leq \Delta(\alpha)  \delta^{d_m} \,.
 $$
Since $M= M' \cup M''$, with $\mu(M') <\epsilon$, as remarked above,  the statement follows. 
\end{proof}

\subsubsection{Distortion bound}\label{sec:distortion}
Finally, for the proof of mixing (which will be given in the next subsection~\ref{sec:finalmixing}), we also need to control the \emph{distortion} of the pushforwards through estimates on $ZD_t$.  

\begin{lemma}[Distortion control]\label{lemma:distortion}
Let $D_t(x)$ be the function defined in \eqref{eq:time_change_int}. There exists a constant~$C'_X(\alpha)>1$ such that the following hold:
\begin{itemize}
\item[(i)] For every $x \in M$ and for every $t\geq 0$,
$$
\vert Z D_t(x)  \vert \leq C'_X(\alpha) \sum_{{\bf v} \in \Lambda^\ast} \vert D^{\bf v}_t(x) \vert \,.
$$	
\item[(ii)]
Let $\Delta(\alpha), d_m>0$ be constants as in Lemma~\ref{lemma:D_t_comparison}. For every $\delta \in (0,1)$,
there exist constants $C'_\delta(\alpha)>1$ and $\sigma_0=\sigma_0(\delta)>0$ such that,  whenever for some $t\geq 0$
 and $\varepsilon>0$, we have
$$\mu \left\{x\in M :\, \sum_{{\bf v} \in \Lambda^\ast} \vert D^{\bf v}_t(x) \vert \leq C'_\delta(\alpha) \right\} \,<\,\varepsilon\,,$$ 
then we also have
$$
\mu \left\{ x\in  M :\,   \vert Z D_t\left(\phi_s^Z(x)\right) \vert \leq  \frac{C'_X(\alpha)}{\delta} \vert D_t\left(\phi_s^Z(x)\right) \vert \, \, \textrm{for\, all}\,\, s\in [0,\sigma_0]\right\} 
\geq 1- \Delta(\alpha) \delta^{d_m} - \varepsilon \,.
$$
\end{itemize}
\end{lemma} 

\begin{proof}
Let us recall that, by Remark~\ref{rk:equiv_tautilde}, we have 
$$
D_t(x) =- \int_0^{\widetilde{\tau}(x,t)}  Z\alpha^\perp \circ \phi^X_r(x) \diff r 
$$
and that, by formula \eqref{eq:Wtau} (for $W=Z$), we also have
$$
Z  \widetilde{\tau}(x,t) = -[ \alpha \circ \phi^X_{\widetilde{\tau}(x,t) } (x)]^{-1} \int_0^{\widetilde{\tau}(x,t)}  Z\alpha^\perp \circ \phi^X_r(x) \diff r = [ \alpha \circ \phi^X_{\widetilde{\tau}(x,t) } (x)]^{-1} D_t (x) \,.
$$
From the above formula it follows that
$$
ZD_t(x)= - \int_0^{\widetilde{\tau}(x,t)}  Z^2\alpha^\perp \circ \phi^X_r(x) \diff r - 
Z  \widetilde{\tau}(x,t)  \,Z\alpha^\perp \circ \phi^X_{\widetilde{\tau}(x,t) } (x)\,.
$$
Since the functions $\alpha$ and $Z\alpha^\perp$ are bounded, there exists a constant $C_X(\alpha)>1$
such that 
$$
\begin{aligned}
\vert  D_t(x) \vert &\leq C_X(\alpha)  \sum_{{\bf v} \in \Lambda^\ast} \vert D^{\bf v}_t(x) \vert  \,, \\
\vert ZD_t(x) \vert &\leq C_X(\alpha)  \sum_{{\bf v} \in \Lambda^\ast} \vert D^{\bf v}_t(x) \vert  +
\frac{ \max \vert Z\alpha^\perp\vert}{\min \alpha} \vert  D_t(x) \vert \,.
\end{aligned} 
$$
Combining the previous two estimates,  the proof of Part $(i)$ of the statement is concluded.

\smallskip

Let us now prove Part $(ii)$. Fix $\delta \in (0,1)$. Let $C_\delta (\alpha)$ be the constant given by  Lemma \ref{lemma:D_t_comparison} and 
let $C'_\delta:=C'_\delta (\alpha) \geq C_\delta(\alpha)$.  Assume that $\mu(\mathcal B_{t}(C'_\delta)) < \varepsilon$, for some $t\geq 0$ and $\varepsilon >0$, and remark  that, since $C'_\delta \geq C_\delta(\alpha)$, the assumption of Lemma \ref{lemma:D_t_comparison} holds.  Thus, by  Part $(i)$  of this Lemma, by Lemma~\ref{lemma:quasi_invariance} and by Lemma \ref{lemma:D_t_comparison}, 
there exists a set $\mathcal B_{t,\delta} \supset \mathcal B_{t}(C'_\delta)$ of measure $\mu(\mathcal B_{t,\delta})\leq \Delta(\alpha)\delta^{d_m} + \varepsilon$,  
such that, for all $x\in M\setminus \mathcal B_{t,\delta}$ and  for all $s\in [-1,1]$,  we have
\begin{equation}\label{eq:lemma_distortion_1}
\vert Z D_t \circ \phi_s^Z(x) \vert \leq C(\alpha) \sum_{{{\bf v}}\in \Lambda^\ast} \vert D^{\bf v}_t \circ \phi^Z_s(x) \vert 
\leq C(\alpha) C_X(\alpha) \sum_{{{\bf v}}\in \Lambda^\ast} \vert D^{\bf v}_t(x)\vert
 \leq \frac{C(\alpha) C_X(\alpha)}{\delta} \vert D_t(x) \vert.
\end{equation}
Thus, let us define  $\sigma_0:=\sigma_0(\delta) \in (0,1)$ by setting $\sigma_0:= {\delta}/{\left(2 C(\alpha)C_X(\alpha)\right)}$. Thus, 
 by the Mean-Value Theorem and by the upper bound in formula~\eqref{eq:lemma_distortion_1}, 
 for all $x \notin  \mathcal B_{t,\delta}$ (i.e. on a set of measure at least $1- \Delta(\alpha)\delta^{d_m} - \varepsilon$), for all $s\in [0,\sigma_0]$,
$$
\vert D_t \circ \phi_s^Z(x) - D_t(x)\vert \leq \sigma_0 \cdot \max_{0\leq r\leq s} \vert Z D_t \circ \phi_r^Z(x) \vert \leq \sigma_0 \frac{C(\alpha) C_X(\alpha)}{\delta} \vert D_t(x) \vert \leq \frac{1}{2}  \vert D_t(x) \vert,
$$
from which it follows that 
$$
\vert D_t \circ \phi_s^Z(x)\vert\geq \left\vert\vert D_t(x)\vert - \vert D_t \circ \phi_s^Z(x) - D_t(x)\vert\right\vert \geq \vert D_t(x)\vert - \frac{ \vert D_t(x)\vert}{2} = \frac{ \vert D_t(x)\vert}{2},
$$
or, in other words, $\vert D_t(x)\vert \leq 2 \vert D_t \circ \phi_s^Z(x)\vert $. Combining this last inequality with the one in formula~\eqref{eq:lemma_distortion_1} and setting $C_X'(\alpha):= 2C(\alpha)C_X(\alpha)$ gives the conclusion.

\end{proof}

\subsubsection{Final mixing arguments}\label{sec:finalmixing}
By Lemma~\ref{lemma:correlationseasy} (or Lemma~\ref{lemma:correlations}) and Remark~\ref{rk:mixingconclusion}, in order to prove mixing it is sufficient to study integrals of the form 
$$
\int_0^{\sigma}  f \circ \phi^V_t\circ \phi^Z_s (x) \diff s
$$ 
for every $0<\sigma\leq \sigma_0$ (where $\sigma_0$ will be determined later). 
We write 
\begin{equation}
\label{eq:induction_case_2}
\begin{split}
&\int_0^{\sigma}  f \circ \phi^V_t\circ \phi^Z_s  \diff s = \int_0^{\sigma} \left( \frac{1+ D^2_t}{1+ D^2_t} \,\,
f \circ \phi^V_t \right ) \circ \phi^Z_s  \diff s \\ 
& \quad = \int_0^{\sigma} \left( \frac{ f \circ \phi^V_t}{1+ D^2_t} \right) \circ \phi^Z_s  \diff s  + 
\int_0^{\sigma} \left( \frac{D^2_t}{1+ D^2_t} \,\, f \circ \phi^V_t \right) \circ \phi^Z_s  \diff s\,,
\end{split}
\end{equation}
and estimate  separately the two terms on the right-hand side of the above formula. 

\smallskip
\noindent 
\emph{First term in the RHS of \eqref{eq:induction_case_2}.}
We will prove that for any fixed $\sigma > 0$, the first term on the right-hand side (RHS for short) of \eqref{eq:induction_case_2} converges to zero in measure, as needed to be able to apply Lemma~\ref{lemma:correlationseasy}.

\smallskip
\noindent By Proposition~\ref{prop:D_t_grows}, $1/ (1 + D_t^2) \to 0$ in measure.
In particular, for every $s \in [0,\sigma]$, since $\mu$ is $\phi^Z$-invariant,
$$
\left( \frac{1}{1+ D^2_t} \right) \circ\phi^Z_s \to 0 \text{\ \ \ in measure}.
$$
Since we can bound
$$
\left\lvert \int_0^{\sigma} \left( \frac{f \circ \phi^V_t}{1+ D^2_t }  \right) \circ \phi^Z_s  \diff s  \right\rvert \leq \norm{f}_{\infty} \int_0^{\sigma} \left\lvert  \left(\frac{1}{1+ D^2_t}\right)\circ\phi^Z_s  \right\rvert \diff s,
$$
in order to conclude, we apply the following type of dominated convergence result, whose proof is included in  \S\ref{app:technical_Lemma} (Appendix B) for completeness.
\begin{lemma}[see Appendix B]
\label{technical_lemma}
Let $\{ g_t \colon M \to \R\}_{t \in \R}$ be a family of smooth functions. Assume that the functions $g_t$ are uniformly bounded over all $t\in \R$, and that $g_t \circ \phi^Z_s$ converges to zero  in measure as $t\to \infty$,  for every fixed $s \in \R$. Then, for every $\sigma >0$, we have
$$
\int_0^{\sigma} \left\lvert g_t \circ \phi^Z_s (x) \right\rvert \diff s \to 0 \text{\ \ \ in measure for $t \to \infty$.}
$$
\end{lemma}
\noindent This concludes the estimate of the first term in the RHS.

\smallskip
\noindent 
\emph{Second term in the RHS of \eqref{eq:induction_case_2}.}
Integrating the second term on the right-hand side of \eqref{eq:induction_case_2} by parts, we obtain
$$
\begin{aligned}
\int_0^{\sigma} \left(\frac{D^2_t}{1+ D^2_t} \,\, f \circ \phi^V_t \right)  \circ \phi^Z_s  \diff s =&
\int_0^{\sigma} \left(\frac{D_t}{1+ D^2_t}\right)  \circ\phi^Z_s  \,\,  \left[   (f \circ \phi^V_t ) D_t  \right] \circ\phi^Z_s \diff s  \\
= &  \left[ \left(\frac{D_t}{1+ D^2_t}\right) \circ\phi^Z_{\sigma} \right] \int_0^{\sigma}  \left[(f \circ \phi^V_t ) D_t \right]  \circ \phi^Z_s  \diff s \\
- & \int_0^{\sigma} \frac{\diff }{\diff s}\left[ \left(\frac{D_t}{1+ D^2_t}\right) \circ\phi^Z_s \right]  \left[ \int_0^s  \left[( f \circ \phi^V_t)   D_t \right] \circ \phi^Z_r  \diff r\right]  \diff s .
\end{aligned}
$$
By \eqref{eq:int_along_gamma}, we then have the following identity:
$$
\begin{aligned}
\int_0^{\sigma} \left( \frac{D^2_t}{1+ D^2_t} \,\, f \circ \phi^V_t \right)\circ \phi^Z_s(x)  \diff s = &
\left[\left(\frac{D_t}{1+ D^2_t} \right)  \circ\phi^Z_{\sigma}(x) \right]  \left( \int_{\gamma^{\sigma}_{x,t}} f \hat V \right) \\
\\ & -  \int_0^{\sigma} \frac{\diff }{\diff s}\left[ \left(\frac{D_t}{1+ D^2_t}\right) \circ\phi^Z_s(x)\right]  \left( \int_{\gamma^s_{x,t}} f \hat V \right) \diff s\,.
\end{aligned}
$$
The flow $\phi^V$ is uniquely ergodic, hence we can assume that $f$ is a smooth $V$-coboundary, namely $f = V g$ for some smooth function $g$. 
Under this assumption  the terms $\int_{\gamma^s_{x,t}} f \hat V$ are uniformly bounded over all $t\in \R$ and $s\in [-1,1]$. In fact, we have
$$
\left\lvert \int_{\gamma^s_{x,t}} f \hat V \right\rvert \leq \left\lvert \int_{\gamma^s_{x,t}} \diff g \right\rvert + \left\lvert \int_{\gamma^s_{x,t}} Zg \hat Z \right\rvert \leq 2 \norm{g}_{\infty} + s \norm{Z g}_{\infty}.
$$
Thus we obtain
\begin{equation*}
\begin{split}
\left\lvert \int_0^{\sigma} \left(\frac{D^2_t}{1+ D^2_t}  \,\, f \circ \phi^V_t \right) \circ \phi^Z_s(x)  \diff s \right\rvert  \leq (2 \norm{g}_{\infty} + \sigma \norm{Z g}_{\infty} ) &  \left( \left\lvert \left(\frac{D_t}{ 1+ 
 D^2_t} \right)\circ\phi^Z_{\sigma}(x) \right\rvert +  \right.\\& \left.\qquad   \int_0^{\sigma} \left\lvert \frac{\diff }{\diff s} \left[ \left(\frac{D_t}{1+ D^2_t} \right)\circ\phi^Z_s(x)\right]\right\rvert  \diff s \right).
\end{split}
\end{equation*}
Since by Proposition~\ref{prop:D_t_grows}
$$
\left\lvert \left(\frac{D_t}{1+ D^2_t} \right) \circ\phi^Z_{\sigma} \right\rvert \to 0 \text{\ \ \ in measure},
$$
it remains to estimate the term
\begin{equation*}
\begin{split}
\int_0^{\sigma} \left\lvert \frac{\diff }{\diff s} \left[  \left(\frac{D_t}{1+ D^2_t} \right)\circ\phi^Z_s(x)\right] \right\rvert \diff s & \leq \norm{\frac{1- D^2_t}{1+ D^2_t}}_{\infty}\int_0^{\sigma} \left\lvert  \left(\frac{ZD_t}{1+ D^2_t} \right) \circ\phi^Z_s(x)  \right\rvert \diff s \\
& \leq \int_0^{\sigma} \left\lvert  \left(\frac{ZD_t}{1+ D^2_t} \right) \circ\phi^Z_s(x)  \right\rvert \diff s.
\end{split}
\end{equation*}
\noindent We cannot directly apply Lemma~\ref{technical_lemma} (and therefore verify the assumptions of Lemma~\ref{lemma:correlationseasy}) because the integrand functions might not be uniformly bounded. 
We will require here the full strength of Lemma~\ref{lemma:correlations} and verify that we can apply it by showing that, for every $\varepsilon >0$, there exists $0< \sigma <1$ such that, for every $\eta>0$, there exists $T >0$ such that, for every $t\geq T$ we have that, for any $ s\in [0, \sigma]$, 
\begin{equation}\label{eq:remaintoprove}
\mu \left( x \in M : \int_0^{s} \left\lvert  \left( \frac{Z D_t }{1+ D^2_t}\right) \circ\phi^Z_r(x)  \right\rvert \diff r \geq \eta \right) \leq \varepsilon.
\end{equation}
\noindent Let $\varepsilon > 0$. Let $\Delta(\alpha)$, $d_m$ be as in Lemma~\ref{lemma:D_t_comparison}, and choose $0< \delta < 1$ such that $\Delta(\alpha) 
\delta^{d_m} < \varepsilon/3$. Let $\sigma_0=\sigma_0(\delta)>0$ and $C_\delta'(\alpha)$ be given by Part $(ii)$ of Lemma~\ref{lemma:distortion}. Take $\sigma=\sigma_0$ 
and fix any $\eta >0$. 
By Proposition~\ref{prop:D_t_grows}, we may apply Lemma~\ref{technical_lemma} to the family of functions $g_t = \frac{ D_t}{1+ D^2_t}$ and conclude that there exists $T>0$ such that for all $t \geq T$ we have 
$$
\mu \left( x \in M : \int_0^{\sigma} \left\lvert  \left(\frac{ D_t }{1+ D^2_t}\right) \circ\phi^Z_s(x) \right\rvert \diff s \geq \frac{\delta \eta}{C_X'(\alpha)} \right) \leq \frac{\varepsilon}{3}.
$$
Furthermore, by Lemma~\ref{lemma:sumD_t_grows}, we obtain that, for a possibly larger $T>0$, for all $t \geq T$ we also have
$$
\mu \left( x \in M : \sum_{{\bf v}\in \Lambda^\ast}|D^{{\bf v}}_t(x)| <C'_\delta(\alpha) \right) \leq \frac{\varepsilon}{6}.
$$
Thus, the assumptions of Part $(ii)$ of Lemma~\ref{lemma:distortion} hold and, by the choice of $\delta$, we deduce that, outside a set of measure $\leq \Delta(\alpha) \delta^{d_m} +\varepsilon/6 < 2\varepsilon/3$, for any $s \in [0,\sigma]$ we have
$$
\begin{aligned}
\int_0^{s} \left\lvert \left( \frac{Z D_t }{1+ D^2_t} \right) \circ\phi^Z_r(x) \right\rvert \diff r &\leq \frac{C_X'(\alpha)}{\delta} \int_0^{s} \left\lvert \left( \frac{D_t }{1+ D^2_t} \right) \circ\phi^Z_r(x) \right\rvert \diff r \\ &\leq \frac{C_X'(\alpha)}{\delta} \int_0^{\sigma} \left\lvert  \left( \frac{ D_t }{1+ D^2_t} \right) \circ\phi^Z_s(x) \right\rvert \diff s.
\end{aligned}
$$
Therefore, we conclude that for every $s \in [0, \sigma_0]$, 
$$
\mu \left( x \in M : \int_0^{s} \left\lvert \left( \frac{ ZD_t }{1+ D^2_t} \right) \circ\phi^Z_r(x) \right\rvert \diff r \geq \eta \right) \leq \frac{2\varepsilon}{3} + \frac{\varepsilon}{3} \leq \varepsilon.
$$
This proves the bound in formula~\eqref{eq:remaintoprove} and hence by Remark~\ref{rk:mixingconclusion} completes the proof.

\subsection{Final arguments}\label{sec:final}
We can now  combine the two cases to conclude the proof of Theorem~\ref{thm:main}.

\begin{proof}[Proof of Theorem~\ref{thm:main}]
 Let  $\mathcal{T}_{M,X}$  be the tower of Heisenberg extensions for $M$ based at $X$ fixed at the beginning of this Section~\ref{sec:proof}, whose existence is guaranteed by Corollary \ref{corollary:existence_of_towers}.  
The dense set of smooth time-changes we want to consider is the class  $\mathscr{P}_{\mathcal{T}}$ of trigonometric polynomial with  respect to the tower $\mathcal{T}_{M,X}$. 
By Remark \ref{remark:independent}, $\mathscr{P}_{\mathcal{T}}$ can be taken independent of $X$.
This class is dense in $\mathscr{C}^\infty(M)$ by Corollary~\ref{cor:density}.   
Let  $\alpha \in \mathscr{P}_{\mathcal{T}}$  and assume that it is  \emph{not} measurably trivial. The theorem will be  proved if we show that the corresponding time-change is mixing.  
 
Let $\alpha^{(0)}:=\alpha $ and let  $ \alpha^{(i)}$, for $0\leq i \leq h$, be its projections along the tower $\mathcal{T}_{M,X}$ (defined in \S\ref{sec:decomposition}). 
We claim that there exists a $0\leq i_0\leq h-1$ such that $ (\alpha^{(i_0)})^\perp$ is \emph{not} a coboundary for $X^{(i_0)}$. Indeed, assume by contradiction that  $ (\alpha^{(i)})^\perp$ is a coboundary for $X^{(i)}$ for every $0\leq i\leq h-1$. 
Recall that the base of the tower $M^{(h)}$ is a torus and that the induced vector field $X^{(h)}$ generates a completely irrational linear toral flow. By definition of the class of time-changes, the function $\alpha^{(h)}$ is a trigonometric polynomial on the torus $M^{(h)} = \mathbb{T}^k$ in the classical sense. It is a standard result that $\alpha^{(h)}$ is an almost coboundary (i.e., cohomologous to a constant)  for the linear flow generated by $X^{(h)}$. Hence, its pull-back on $M^{(h-1)}$ is an almost coboundary  for $X^{(h-1)}$. Since we are assuming that $ (\alpha^{(h-1)})^\perp$ is also a measurable coboundary for $X^{(h-1)}$, we deduce that $\alpha^{(h-1)}$ is a measurable almost coboundary for $X^{(h-1)}$. Proceeding in the same way for all $i = h-2, \dots, 0$, we conclude that $\alpha^{(0)}=\alpha$ is an almost coboundary for $X$, in contradiction with the original assumption.

Let  $0\leq i_0\leq h-1$ be \emph{minimal} $i$ such that $ (\alpha^{(i)})^\perp$ is \emph{not} a coboundary for $X^{(i)}$. Applying Proposition \ref{noncoboundary_case} (the \emph{not} coboundary case) to $ {(\alpha^{(i_0)})}^\perp$, we have that the flow on $M^{(i_0)}$ generated by $V^{(i_0)} = \frac{1}{\alpha^{(i_0)} } X^{(i_0)}$ is mixing.
By definition of $i_0$, each $ (\alpha^{(i)})^\perp$, for $0 \leq i \leq i_0-1$, is a coboundary w.r.t.~$X^{(i)}$. Applying now Proposition \ref{coboundary_case} (the coboundary case) to all $i = i_0-1, \dots, 0$, we get in $i_0$ steps that the flow on $M$ generated by $\frac{1}{\alpha} X = V$  is mixing. 
\end{proof}

\section{Appendix A: Proof of Lemma \ref{lemma:Gottschalk-Hedlund}}
\label{sec:Appendix_B}

\begin{proof}[Proof of Lemma \ref{lemma:Gottschalk-Hedlund}]
We notice that we can rewrite
\begin{equation*}
\begin{split}
\frac{1}{T} \int_0^T \mu_\beta \{ |F_t| < C\} \diff t &= \frac{1}{T} \int_0^T \left( \int_M  \one_{(-C,C)} \circ F_t(x) \diff \mu_\beta \right) \diff t \\
&= \int_M \left( \frac{1}{T} \int_0^T \one_{(-C,C)} \circ F_t(x) \diff t \right) \diff \mu_\beta, 
\end{split}
\end{equation*}
therefore, by the Dominated Convergence Theorem, it is sufficient to show that the function $\frac{1}{T} \int_0^t \one_{(-C,C)} \circ F_t(x) \diff t$ converges pointwise to zero.

For all $x \in M$, denote by $\nu_{T,x}$ the probability measure on $M \times \R$ supported on the parametrized curve $t \mapsto (\phi^U_t(x),F_t(x))$, for $t \in [0,T]$, defined by
$$
\nu_{T,x}(A \times [a,b]) = \frac{1}{T} \text{Leb} \left( t \in [0,T] : \phi^U_t(x) \in A \text{ and } F_t(x) \in [a,b] \right), 
$$
for $A \subset M$ and $[a,b] \subset \R$.
We will prove that, for all $x \in M$, $\nu_{T,x}$ converges weakly to~$0$.

Suppose on the contrary that there exist $\overline{x} \in M$ and a strictly increasing sequence $T_n \to \infty$ such that $\nu_{T_n, \overline{x}}$ converges weakly to a measure $\nu$ with non-zero total mass.
We claim that $\nu$ is $\Psi_t$-invariant, where $\Psi_t(x,s) = (\phi_t^U(x), s+F_t(x))$. Indeed, for every continuous function $g$ on $M \times \R$ we have
\begin{equation*}
\begin{split}
\int_{M \times \R} g \diff(\Psi_t)_{\ast} \nu_{T_n,\overline{x}} &= \int_{M \times \R} g \circ \Psi_t \diff \nu_{T_n,\overline{x}} =  \int_{M \times \R} g (\phi_t^U(\overline{x}), s+F_t(\overline{x})) \diff \nu_{T_n,\overline{x}} \\
&= \frac{1}{T_n} \int_0^{T_n} g \left(\phi_{\tau}^U \circ \phi_{t}^U(\overline{x}), s+F_{t}(x) + F_{\tau}(\phi_{t}^U(\overline{x})) \right) \diff \tau 
\end{split}
\end{equation*}
By the cocycle relation \eqref{eq:cocycle_relation} and by definition of $\mu_{T_n,\overline{x}}$ we obtain
\begin{equation*}
\begin{split}
& \left\lvert \int_{M \times \R} g \diff(\Psi_t)_{\ast} \nu_{T_n,\overline{x}} - \int_{M \times \R} g \diff\nu_{T_n,\overline{x}}  \right \rvert \\
& \quad = \frac{1}{T_n} \left\lvert \int_0^{T_n} g \left(\phi_{t+\tau}^U(\overline{x}), s+F_{t+\tau}(\overline{x}) \right)  \diff \tau -  \int_0^{T_n} g \left(\phi_{\tau}^U(\overline{x}), s+F_{\tau}(\overline{x}) \right)  \diff \tau \right \rvert \leq \frac{2 \norm{g}_{\infty} t}{T_n}.
\end{split}
\end{equation*}
Therefore, $\lim_{n \to \infty} (\Psi_t)_{\ast} \nu_{T_n,\overline{x}} = \nu$, and the claim follows from the continuity of $(\Psi_t)_{\ast}$.

Let $\widehat{\nu}$ be an ergodic component of $\nu$. 
By unique ergodicity of $\phi^X$, and hence of $\phi^U$, we have that $\pi_{\ast}\widehat{\nu} = \mu_{\beta}$, where $\pi \colon M \times \R \to M$ is the projection onto the nilmanifold. 
In particular, for almost every $x \in M$ there exists a point $(x,s) \in M\times \R$ which is generic for $\widehat{\nu}$. 
Assume that there exists a fiber $\{x\} \times \R$ over $M$ with more than one generic point, that is, assume that the points $(x,s)$ and $(x,s+r)$ are both generic for  $\widehat{\nu}$. 
This implies that  $\widehat{\nu}$ is $T_r$-invariant, where $T_r$ denotes the vertical translation on the fibers by $r$:
for any continuous, compactly supported function $g \in \mathscr{C}_c(M \times \R)$, we have
$$
\frac{1}{T} \int_0^t g \circ \Psi_t (x,s+r) \diff t \to \int_{M \times \R} g \diff \widehat{\nu},
$$
but also
\begin{equation*}
\begin{split}
&\frac{1}{T} \int_0^t g \circ \Psi_t \circ T_r (x,s) \diff t = \frac{1}{T} \int_0^t g \circ T_r \circ \Psi_t (x,s) \diff t \\
& \quad \to \int_{M \times \R} g \circ T_r  \diff \widehat{\nu} = \int_{M \times \R} g\diff (T_r)_{\ast} \widehat{\nu}.
\end{split}
\end{equation*}
Since $g$ was arbitrary, we deduce $\widehat{\nu} = (T_r)_{\ast} \widehat{\nu}$. As $\widehat{\nu}$ is a finite measure, we must have $r = 0$, namely for almost every $x \in M$ there exists only one point $(x, u(x))\in M \times \R$ which is generic for $\widehat{\nu}$.
The function $u \colon x \mapsto u(x)$ implicitly defined above is measurable, since its graph is a measurable set.
Uniqueness implies that
$$
\Psi_t(x,u(x)) = \left(\phi_t^U(x), u(x) + F_t(x) \right) = \left(\phi_t^U(x), u(\phi_t^U(x)) \right),
$$
from which we deduce $u (\phi_t^U(x)) - u(x) = F_t(x)$, in contradiction with the assumption that $f$ is not a measurable coboundary for $U$.
\end{proof}

\section{Appendix B: Proof of Lemma~\ref{technical_lemma}}
\label{app:technical_Lemma}

\begin{proof}[Proof of Lemma~\ref{technical_lemma}]
Let us fix $\sigma > 0$. Denote by $\nu$ the product measure $\diff \nu = \diff \mu \, \diff r$ on ${\hat M} := M \times [0,\sigma]$. 
We first claim that $g_t \circ \phi^Z_s $ converges to zero in measure on ${\hat M} $. 
Define 
$$
E_t^{\delta} := \{ (x,s) \in {\hat M}  : | g_t\circ \phi^Z_s (x)| > \delta \}.
$$
We have 
$$
\nu (E_t^{\delta}) = \int_{\hat M}  \one_{E_t^{\delta}}(x,s) \diff \nu = \int_0^{\sigma} \left( \int_M \one_{E_t^{\delta}}(x,s) \diff \mu(x) \right) \diff s.
$$
By assumption, the term in brackets converges to zero for all $s \in [0,\sigma]$, hence, by Lebesgue Theorem, $\nu(E_t^{\delta}) \to 0$.

Let $E_t^{\delta}(x) := E_t^{\delta} \cap \{x\} \times [0,\sigma] = \{ s \in [0,\sigma] : | g_t\circ \phi^Z_s (x) | > \delta \}$ and denote by $|E_t^{\delta}(x)|$ its measure. 
Define also 
$$
\text{Bad}^{\delta_1, \delta_2}_t := \{ x \in M : | E_t^{\delta_1}(x) | > \delta_2 \}.
$$
We claim that for all $\delta_1, \delta_2 >0$, $\mu (\text{Bad}^{\delta_1, \delta_2}_t ) \to 0$.
If this was not the case, there would exist $\delta_1, \delta_2, \eta>0$ and an increasing sequence $t_n \to \infty$ such that $\mu(\text{Bad}^{\delta_1, \delta_2}_{t_n} )\geq \eta$ for all $n \in \N$. This would imply that for all $n \in \N$, by Fubini's Theorem,
$$
\nu( E_{t_n}^{\delta_1}) = \int_M | E_{t_n}^{\delta_1}(x) | \diff \mu (x) \geq \int_{\text{Bad}^{\delta_1, \delta_2}_{t_n}} | E_{t_n}^{\delta_1}(x) | \diff \mu (x) \geq \delta_2 \cdot \eta,
$$
in contradiction with $\nu(E_t^{\delta_1}) \to 0$.

Let $C>0$ be such that $\norm{g_t}_{\infty} \leq C$. Fix $\varepsilon, \eta >0$ and choose $\delta_1, \delta_2>0$ such that $C\delta_1 + \sigma \delta_2 < \eta$. Let $T >0$ be such that $\mu (\text{Bad}^{\delta_1, \delta_2}_t ) < \varepsilon$ for all $t \geq T$. 
For any $t \geq T$ and for any $x \notin \text{Bad}^{\delta_1, \delta_2}_t$, we have $| E_t^{\delta_1}(x) | \leq \delta_2$; hence 
\begin{equation*}
\begin{split}
\int_0^{\sigma}\left\lvert g_t\circ \phi^Z_s (x) \right\rvert \diff s &\leq \int_0^{\sigma} \one_{E_t^{\delta_1}(x)}(s) \left\lvert g_t\circ \phi^Z_s (x) \right\rvert \diff s + \int_0^{\sigma} \one_{(E_t^{\delta_1}(x))^c}(s) \left\lvert g_t\circ \phi^Z_s (x) \right\rvert \diff s \\
& \leq C \delta_2 + \sigma \delta_1 < \eta.
\end{split}
\end{equation*}
Therefore,
$$
\mu \left( x \in M :  \int_0^{\sigma} \left\lvert g_t\circ \phi^Z_s (x) \right\rvert \diff s > \eta \right) \leq \mu \left( \text{Bad}^{\delta_1, \delta_2}_t \right) < \varepsilon,
$$
which concludes the proof.
\end{proof}

\subsection*{Acknowledgement}
The authors would like to thank Livio Flaminio for useful discussions and an anonymous referee for his/her comments and questions on an earlier version of this manuscript. We thank the referee also for suggesting some stronger variants of some of the algebraic Lemmas, in particular of Lemma~\ref{thm:Heisenberg_triple}.
For their hospitality  at the very initial stages of this project, C.U.~thanks the Department of Mathematics of the University of Maryland and both A.A.~and C.U.~acknowledge the Hausdorff Institute for Mathematics in Bonn for hosting them during  the program  {\it Geometry and Dynamics of Teichm{\"u}ller Space}. 
G.F.~thanks the Department of Mathematics of the University of Bristol, UK, for its hospitality
during the final stages of preparation of the paper. G.F.~acknowledges support of NSF grant DMS 1600687. 
C.U.~and D.R.~acknowledge the ERC Grant {\it ChaParDyn} and the Leverhulme Trust for the funding which made this research possible. 
C.U. is also supported 
by the Royal Society through a Wolfson Research Merit Award.  
The research leading to these results has received funding from the European Research Council under the European Union
Seventh Framework Programme (FP/2007-2013) / ERC Grant Agreement n. 335989.

 \end{document}